\documentclass[11pt, reqno]{amsart}
\makeatletter
\usepackage{fullpage}
\usepackage{textcmds}  
\usepackage{amsmath, amssymb, amsfonts, amstext, verbatim, amsthm, mathrsfs}
\usepackage[mathcal]{eucal}
\usepackage{stmaryrd}
\usepackage{microtype}
\usepackage{graphicx}
\usepackage[all,cmtip,2cell]{xy}
\UseTwocells
\usepackage{pgf,tikz,pgfplots}
\pgfplotsset{compat=1.15}
\usetikzlibrary{arrows}
\usepackage{tikz-cd}
\usepackage{enumerate}
\usepackage{enumitem}
\usepackage{xspace}
\usepackage{cite}
\usepackage{pifont}
\usepackage[modulo]{lineno}
\usepackage[colorlinks=true,linkcolor=blue,citecolor=blue,urlcolor=blue,citebordercolor={0 0 1},urlbordercolor={0 0 1},linkbordercolor={0 0 1}]{hyperref} 
\usepackage[shortalphabetic]{amsrefs} 
\usepackage[nameinlink]{cleveref}

\tikzset{
    cross/.pic = {
    \draw[-, ultra thick , rotate = 45] (-#1,0) -- (#1,0);
    \draw[- , ultra thick, rotate = 45] (0,-#1) -- (0, #1);
    }
}

\def\makeCal#1{%
\expandafter\newcommand\csname c#1\endcsname{\mathcal{#1}}}
\def\makeBB#1{%
\expandafter\newcommand\csname b#1\endcsname{\mathbb{#1}}}
\def\makeFrak#1{%
\expandafter\newcommand\csname f#1\endcsname{\mathfrak{#1}}}

\count@=0
\loop
\advance\count@ 1
\edef\y{\@Alph\count@}%
\expandafter\makeCal\y
\expandafter\makeBB\y
\expandafter\makeFrak\y
\ifnum\count@<26
\repeat

\theoremstyle{plain}
\newtheorem{thm}{Theorem}[section]
\newtheorem{cor}[thm]{Corollary}
\newtheorem{lem}[thm]{Lemma}

\newtheorem{claim}[thm]{Claim}

\newtheorem{prop}[thm]{Proposition}
\newtheorem{quest}[thm]{Question}

\theoremstyle{definition}
\newtheorem{rem}[thm]{Remark}
\newtheorem{defn}[thm]{Definition}
\newtheorem{hyp}[thm]{Hypotheses}

\newtheorem{ex}[thm]{Example}

\newtheorem{rec}[thm]{Recollection}

\newtheorem*{notn*}{Notation}

\DeclareMathOperator{\Ch}{Ch}
\DeclareMathOperator{\DCoh}{D^b_{coh}}
\DeclareMathOperator{\Dqc}{D_{qc}}

\DeclareMathOperator{\colim}{colim}

\DeclareMathOperator{\End}{End}
\DeclareMathOperator{\Ext}{Ext}

\DeclareMathOperator{\GL}{GL}

\DeclareMathOperator{\Hom}{Hom}

\DeclareMathOperator{\Mod}{-Mod}

\DeclareMathOperator{\pt}{pt}

\DeclareMathOperator{\rank}{rank}
\DeclareMathOperator{\Rep}{Rep}
\DeclareMathOperator{\RHom}{RHom}

\DeclareMathOperator{\Span}{Span}
\DeclareMathOperator{\Spec}{Spec}

\DeclareMathOperator{\Stab}{Stab}
\DeclareMathOperator{\Sym}{Sym}

\DeclareMathOperator{\wt}{wt}
\DeclareMathOperator{\Pic}{Pic}

\newcommand*{\sheafhom}{\mathcal{H} \kern -.5pt om}

\usepackage{pbox}
\usepackage[normalem]{ulem}

\AtBeginDocument{%
   \def\MR#1{}
} 

\def\l@subsection{\@tocline{2}{0pt}{2.5pc}{4pc}{}} 

\makeatletter

\usepackage{babel}
\begin{document}

\title{Full exceptional collections of vector bundles on rank-two linear GIT quotients}

\author{Daniel Halpern-Leistner and Kimoi Kemboi} 

\begin{abstract}
We produce full strong exceptional collections consisting of vector bundles on the geometric invariant theory quotient of certain linear actions of a split reductive group $G$ of rank two. The vector bundles correspond to irreducible $G$-representations whose weights lie in an explicit bounded region in the weight space of $G$. We also describe a method for constructing more examples of linear GIT quotients with full strong exceptional collections of this kind as ``decorated" quiver varieties.
\end{abstract}

\maketitle

\tableofcontents

\section{Introduction}

A well-known result of Beilinson \cite{B78} states that the line bundles $E_i := \cO_{\bP^n}(i)$ for $i=0,\ldots,n$ form a full strong exceptional collection in the derived category of coherent sheaves $\DCoh(\bP^n)$, meaning that
\begin{enumerate}
\item $\Ext^p(E_i,E_j) = 0$ whenever $p\neq 0$ or $i>j$,
\item $\Hom(E_i,E_i) = k$ for all $i$; and
\item $\{E_i\}$ generate $\DCoh(\bP^n)$ as a triangulated category.
\end{enumerate}
A remarkable consequence of this is an equivalence of categories $\DCoh(\bP^n) \cong {\rm D}^b(A\Mod)$ for the finite dimensional algebra $A = \Gamma(\bP^n, \End(E_0 \oplus \cdots \oplus E_n))$. This effectively reduces homological questions about sheaves on $\bP^n$ to analogous questions about finite dimensional $A$-modules.

Beilinson's theorem inspired many generalizations. For instance, thinking of $\bP^n$ as a homogeneous space leads naturally to the question of which homogeneous spaces of semisimple algebraic groups admit a full (strong) exceptional collection. Soon after Beilinson's work, Kapranov produced full strong exceptional collections of vector bundles on homogeneous spaces of type $A$ and on quadrics \cite{K88}. Partial progress has been made for other types of homogeneous spaces in \cite{Ku08}, \cite{PS11}, \cite{M11}, \cite{FM15}, \cite{KP16}, \cite{F19}, \cite{AAGZ13} among other papers.

The perspective we consider is that of $\bP^n$ as the simplest example of a geometric invariant theory (GIT) quotient of a vector space by a reductive group, whence Beilinson's theorem inspires the following:
\begin{quest} \label{Q:main_question}
Given a reductive group $G$ and a linear representation $X$ such that points of $X^{\rm ss}$ have finite stabilizers and $X^{\rm ss}/G$ is proper, when does $\DCoh(X^{\rm ss}/G)$ admit a full strong exceptional collection of vector bundles?
\end{quest}
Note that $G$ will typically not act freely on $X^{\rm ss}$, so the resulting good quotient will be singular and cannot have a full strong exceptional collection. It is therefore natural to expand our context and regard $X^{\rm ss}/G$ as an orbifold, which we do for the rest of the paper.

It was originally conjectured by King \cite{K97} that \Cref{Q:main_question} would have a positive answer when $G = \bG_m^n$ is a torus and $G$ acts freely on $X^{\rm ss}$, so that $X^{\rm ss}/G$ is a smooth projective toric variety. In this case, King's question asked for the exceptional collection to consist of line bundles. It turns out that this conjecture fails (see for instance \cite{HP06} and \cite{Mi11}), and it also fails if one considers only cases when the GIT quotient is Fano \cite{E14}. These examples shift the flavor of the subject towards understanding the nature and the landscape of the examples in which \Cref{Q:main_question} does have an affirmative answer. We consider the following:

\begin{hyp} \label{H:main_setup} Let $X$ be a linear representation of a split reductive group $G$ with weights $\beta_1,\ldots,\beta_n$. Assume there exists a central cocharacter $\lambda_{0}$ of $G$ such that $\langle \lambda_0, \beta_i \rangle < 0$ for all $i$, with the canonical pairing between weights and cocharacters. Let $\omega^\ast :=\det(X)\otimes\det(\mathfrak{g})^{-1}$
be the anticanonical character of $X$.\end{hyp}

Note that given any cocharacter that pairs negatively with all of the $\beta_i$, we can average under the Weyl group action to obtain a central cocharacter $\lambda_0$ with this property. The existence of such a $\lambda_0$ is equivalent to the condition that $\cO_X^T = k$ for a maximal torus $T \subset G$.

Our main contribution is the following:

\begin{thm}[\Cref{P:exceptional_collections}, \Cref{T:main_fano}, \Cref{T:main_nef_fano} ] \label{T:main_summary}
Let $\ell \in M_\bR^W$. Under \Cref{H:main_setup}, there is a particular convex subset $\Omega \subset M_\bR$ (described below as a ``barrel window") such that the set of $G$-equivariant locally free sheaves on $X^{\rm ss}(\ell)$
\[
\left\{ \cO_{X^{\rm ss}(\ell)} \otimes U : \begin{array}{c} U \in \Rep(G) \text{ irreducible with} \\ \text{weights lying in }\Omega \end{array} \right\}
\]
forms a strong exceptional collection in $\DCoh(X^{\rm ss}(\ell)/G)$. If $G$ has rank two, $\ell$ is close to $\omega^\ast$,\footnote{It suffices to assume that $X^{\rm ss}(\ell) = X^{\rm ss}(t\omega^\ast + (1-t) \ell)$ for all $t \in [0,1)$. The general theory of variation of GIT quotient implies that this will be the case for all $\ell$ in any sufficiently small open neighborhood of $\omega^\ast \in M_\bR^W$.} and points in $X^{\rm ss}(\ell)$ have finite stabilizers, then this collection is full.
\end{thm}

In \Cref{S:GL_2} we describe in detail the case of representations of $\GL_2$. When $G = (\bG_m)^2$, many of the resulting GIT quotients are toric Deligne-Mumford stacks in the sense of \cite{BCS03}, and we show in \Cref{S:toric_examples} that in the Fano case one can recover the full strong exceptional collection constructed by Borisov and Hua in \cite{BH09}.

\medskip
\noindent \textbf{Barrel windows.}
\medskip

The region $\Omega$ in \Cref{T:main_summary} has the form $\theta + \nabla_\ell$, where $\theta \in M_\bR^W$ is a rational character of $G$ that is \emph{generic} with respect to $\lambda_0$ (see \Cref{D:lambda_0_generic} for the precise condition) and $\nabla_\ell \subset M_\bR$ is a certain convex subset described in \Cref{D:window}. The closure $\overline{\nabla}_\ell$ is a rational polytope that we call the ``cylinder window". It is defined by inequalities of the form
\[
-\eta_{\lambda}/2 \leq \langle \lambda, \chi \rangle \leq \eta_{\lambda}/2 \text{ with either } \lambda=\lambda_0 \text{ or } \langle \lambda, \ell \rangle = 0,
\]
where $\langle \lambda, \chi \rangle$ denotes the canonical pairing between a cocharacter $\lambda \in N_\bR := M_\bR^\ast$ and a weight $\chi \in M_\bR$. The numbers $\eta_{\lambda}$ are defined in \Cref{D:conormal_weight}. The polytope $\overline{\nabla}_\ell$ is the intersection of a ``strip" bounded by two $\lambda_0$-hyperplanes with the preimage of a certain polytope $\overline{\nabla}' \subset M_\bR / \bR \ell$. Hence the terminology ``cylinder." If we identify $M_\bR / \bR \ell$ with the weight lattice of $G_\ell := \ker(\ell)$, then $\overline{\nabla}'$ agrees with the polytope associated in \cite{HLS20} to $X$ regarded as a representation of $G_\ell$.

The subset $\nabla_\ell \subset \overline{\nabla}$ is obtained by removing certain points on the boundary $\partial \overline{\nabla}_\ell$ that have a large positive or negative pairing with $\lambda_0$. We refer to $\nabla_\ell$ as a ``barrel window" because the set of lattice points $(\theta + \nabla_\ell) \cap M$, which is the only thing that is relevant to determining when the weights of a $U \in \Rep(G)$ lie in $\theta + \nabla_\ell$, agrees with the set of lattice points $(\theta + \overline{\nabla}_{\ell,t}) \cap M$ for a small deformation $\overline{\nabla}_{\ell,t}$ of $\overline{\nabla}_\ell = \overline{\nabla}_{\ell,0}$. The region $\overline{\nabla}_{\ell,t}$ pinches the top and bottom of the cylinder $\overline{\nabla}_\ell$, giving it the appearance of a barrel. We explain this in \Cref{fig:grassmannian_N_6} and \Cref{R:t_perturbed_regions}.

\medskip
\noindent \textbf{The molecular perspective.}
\medskip

We now describe a perspective where representations for which \Cref{Q:main_question} has an affirmative answer can be built as ``molecules" from ``atoms" drawn from a shorter list of basic ``elemental" examples. 

Craw has constructed full strong exceptional collections of vector bundles on ``quiver flag varieties" \cite{C11}, significantly generalizing Kapranov's examples in \cite{K88}. The main observation behind Craw's result is that a quiver flag variety can be realized as an iterated fiber bundle, where each fiber is a Grassmannian. 
One can then reduce the construction of a full exceptional collection on the quiver flag variety to that of constructing (equivariant) full exceptional collections on Grassmannians.

We generalize Craw's observation as follows: One starts with a directed acyclic graph $Q$, i.e. a ``quiver", with set of vertices $Q_0$ and edges $Q_1$. Then one assigns to each vertex $i \in Q_0$ a reductive group $G_i$, a representation $V_i$ of $G_i$, and a ``framing" dimension $w_i \geq 0$. One then defines a linear representation $\Rep(Q)$ of $G_Q := \prod_{i \in Q_0} G_i$ as a product of several smaller spaces: one copy of $\Hom(V_i,V_j)$ for every edge $i \to j$, and one copy of $V_i^{w_i}$ for each vertex $i \in Q_0$.

Now imagine that for each $i \in Q_0$ we are given a GIT parameter $\ell_i$, i.e. a character of $G_i$, and a set of irreducible representations $\Omega_i$ of $G_i$ such that
\[
\{\cO_{V_i^{\rm ss}(\ell_i)} \otimes U : U \in \Omega_i \}
\]
form a full strong exceptional collection in $\DCoh(V_i^{\rm ss}(\ell_i)/G_i)$. 
In \Cref{P:decorated_quiver_fec} we show that there is then a canonical choice of GIT parameter $\ell_Q$ for $G_Q$ and a set $\Omega_Q$ of irreducible representations of $G_Q$ such that $\DCoh(\Rep(Q)^{\rm ss}(\ell_Q) / G_Q)$ also admits a full strong exceptional collection of tautological vector bundles of the form $\cO_{\Rep(Q)^{\rm ss}} \otimes U$ for $U \in \Omega_Q$.

Beilinson's theorem and Kapranov's generalization to Grassmannians provide elements of the form $(\GL_m, V=(k^m)^{\oplus n})$ for $n>m>0$, and Craw's results follow from \Cref{P:decorated_quiver_fec} applied to a quiver whose vertices are labeled with these elements. To the authors' knowledge, the only other elemental examples prior to this paper come from pairs $(G,V)$ where $G = \bG_m^n$ is a torus and the generic stabilizer is finite:
\begin{enumerate}
    \item When $G$ has rank two, or $\dim(V) - \dim(G) \leq 2$, and the GIT quotient is Fano, this is \cite{BH09} and \cite{CM04}.
    \item When $\dim(V)-\dim(G) = 3$ and the quotient is Fano, this is \cite{BT09} and \cite{U14}.
    \item When $G$ has rank three, some examples are given in \cite{LM11}, \cite{DLM09}, and \cite{H17}.
    \item When $\dim(V)-\dim(G) = 4$ and the quotient is Fano, this is \cite{PN17}.
    \item The GIT quotient $(\bP^1)^n /\!/ \bG_m$ with respect to the linearization $\cO_{\bP^1}(1)^{\boxtimes n}$ for $n$ odd, which is a toric Fano manifold of dimension $n-1$ and hence a linear GIT quotient by a torus, admits a full exceptional collection of line bundles \cite{CT20_1}*{Thm.~1.6}.
\end{enumerate}
From this perspective, \Cref{T:main_summary} produces more examples in this periodic table of elements when $G$ has rank $2$. 

\begin{rem}
The GIT quotient $(\bP^1)^n /\!/ \text{PGL}_2$ for $n$ odd is shown to admit a full exceptional collection consisting of vector bundles in \cite{CT20_2}*{Thm.~1.2}. 
Although this GIT quotient is not linear, there is a finite map $(\bP^1)^n /\!/ \text{PGL}_2 \rightarrow \bP(\Sym ^n (\bC^2)) /\!/ \text{PGL}_2 = \Sym^n (\bC^2) /\!/ \GL_2$, the latter of which falls within the scope of \Cref{T:main_fano}, but we have not investigated the relationship between these exceptional collections.
\end{rem}

\medskip
\noindent \textbf{Main steps in the proof of \Cref{T:main_summary}.}
\medskip

First, we show that the locally free sheaves of the form $\cO_{X^{\rm ss}} \otimes U$ for irreducible representations $U$ in the window $\theta + \nabla_\ell$ form a strong exceptional collection. It is straightforward to show that this is the case for the sheaves $\cO_X\otimes U$ in the equivariant category $\DCoh(X/G)$, so the key is a criterion for the vanishing of the invariant local cohomology that we can apply to show $R\Gamma_{X^{\rm us}}(\cO_X \otimes \Hom(U,V))^G = 0$, where $U,V \in \Rep(G)$ are irreducible with weights lying in $\theta + \nabla$, and $X^{\rm us} = X \setminus X^{\rm ss}$ is the unstable locus. We adapt the local cohomology results of Van den Bergh \cite{VdB91} for this purpose in \Cref{P:vanishing}.

Second, we show that when $G$ has rank $2$, $\ell$ is close to $\omega^\ast$, and $X^{\rm ss}(\ell)$ has finite stabilizers, then the sheaves $\cO_{X^{\rm ss}} \otimes U$ for $U$ with weights in $\theta + \nabla_{\omega^\ast}$ generate the derived category $\DCoh(X^{\rm ss}/G)$. We do this in \Cref{T:main_fano} and \Cref{T:main_nef_fano}. In both cases, we first show that $\DCoh(X^{\rm ss}/G)$ is generated by $\cO_{X^{\rm ss}} \otimes U$, where $U \in \Rep(G)$ ranges over all representations whose weights lie in a ``$\lambda_0$-strip" containing $\theta + \nabla_{\omega^\ast}$. Then we show that for any $U \in \Rep(G)$ whose weights lie in this ``$\lambda_0$-strip" but not in the barrel window $\theta+\nabla_{\omega^\ast}$, the locally free sheaf $\cO_{X^{\rm ss}} \otimes U$ on $X^{\rm ss} / G$ is quasi-isomorphic to an explicit complex of locally free sheaves whose terms admit a $G$-equivariant filtration with associated graded pieces of the form $\cO_{X^{\rm ss}} \otimes V$, where the weights of $V \in \Rep(G)$ still lie in the $\lambda_0$-strip and is closer to lying in $\theta + \nabla$.

\subsubsection*{Acknowledgements} We would like to thank Andr{\'e}s Fernandez Herrero, Wei Gu, J{\o}rgen Rennemo, and Michel Van den Bergh for helpful comments. This work was supported by the NSF CAREER grant DMS-1945478, the NSF FRG grant DMS-2052936, and the NSF grant DMS-1762669.

\section{Strong exceptional collections on linear GIT quotients}

\subsection{Preliminaries}

Throughout the paper, we fix a ground field $k$ of characteristic zero. We will use the following notation for a split reductive group $G$, which we do not assume is connected:
\begin{enumerate}
    \item $T \subset B \subset G$ is a fixed choice of split maximal torus and Borel subgroup.
    \item $M$ is the character lattice of $T$, i.e., the weight lattice of $G$, and $N$ is the cocharacter lattice of $T$. We denote $M_\bR = M \otimes_\bZ \bR$ and $N_\bR = N \otimes_\bZ \bR$.
    \item We denote the canonical pairing between $\lambda \in N_\bR$ and $\chi \in M_\bR$ by $\langle \lambda, \chi \rangle$.
    \item $W$ is the Weyl group of $G$ and $w_0$ is the longest element of $W$.
    \item The roots of $B$ will be the \emph{negative} roots and $\rho$ will denote the half sum of positive roots. The dominant chamber will be denoted $M^+$. 
    \item We say that $\lambda \in N_\bR$ is \emph{anti-dominant} if it pairs non-negatively with all negative roots.
    \item $G^\circ \subset G$ denotes the identity component.
\end{enumerate}

We will also use the following less standard notation:
\begin{defn}
If $G$ is a split reductive $k$-group and $\ell\in M_\bR$ we denote by $\ell^\perp\subset N_\bR$ the subset consisting of cocharacters that pair zero with $\ell$. If $\ell \in M_\bQ^W$ is non-zero, we let $G_\ell^\circ \subset G$ be the identity component of the kernel of the character $P \ell : G \to \bG_m$, where $P \in \bZ_{>0}$ is sufficiently divisible so that $P \ell \in M$. We fix a choice of maximal torus and Borel subgroup $T_\ell \subset B_\ell \subset G_\ell^\circ$ such that $T_\ell \subset T$ and $B_\ell \subset B$.
\end{defn}

Note that in the previous definition $\ker(P\ell)$ depends on the choice of $P$, but $\ker(P\ell) \subset \ker(PQ\ell)$ is a finite index subgroup for any $Q>0$, so the identity component $G^\circ_\ell$ is independent of $P$. 

Given a linear representation $X$ of $G$, we will abuse terminology slightly, by passing freely between regarding $X$ as a representation, i.e., an object of $\Rep(G)$, and as the affine $G$-scheme $\Spec(\Sym(X^*))$. We will use the following common construction:

\begin{defn}\label{D:parabolic_and_attracting_locus}
To any cocharacter $\lambda : \bG_m \to G$ one can associate the attracting locus $X^{\lambda \geq 0}$. If $\beta_1,\ldots, \beta_n \in M$ denote the weights of $X$, then 
\[
    \begin{array}{rl} X^{\lambda\geq 0} & = \left\{ x\in X:\lim_{t\rightarrow 0}\lambda(t)\cdot x \text{ exists}\right\} \vspace{.5em} \\ &= \Span \{ \beta_i : \langle \lambda, \beta_i \rangle \geq 0 \} \end{array} .
\]
We use similar notation $X^{\lambda>0}$, $X^{\lambda\leq0}$, etc., to denote the linear subspace of $X$ spanned by the weights whose pairing with $\lambda$ is $>0$, $\leq 0$, etc.. Note that if $\lambda \in N_\bR$ is anti-dominant, then $X^{\lambda \geq 0}$ and $X^{\lambda>0}$ are naturally linear representations of $B$. 
\end{defn}

Let $X$ be a linear representation of $G$, then the set of $G$-linearized line bundles on $X$ is identified with Weyl-invariant weights $M^W$. Given an $\ell\in M^W$, geometric invariant theory (GIT) identifies a $G$-equivariant open subset $X^{\rm ss}(\ell)\subset X$, called the \emph{semistable locus}, that admits a good quotient under the action of $G$. We will let $X^{\rm ss}(\ell)/G$ denote the quotient stack, and $X^{\rm ss}(\ell) /\!/ G$ its good quotient. The quotient $X^{\rm ss}(\ell) /\!/ G$ is projective over $\Spec(\cO_X^G)$, which is $\Spec(k)$ under \Cref{H:main_setup}.

Under \Cref{H:main_setup}, if points in $X^{\rm ss}(\ell)$ have finite stabilizers in $G$, then $X^{\rm ss}(\ell)/G$ is a smooth projective Deligne-Mumford stack (in the sense of \cite{K08}), and this will be our main case of interest. If there are points of $X^{\rm ss}(\ell)$ with positive dimensional stabilizers, it is sometimes, but not always, possible to perturb $\ell$ so that the new semistable locus has finite stabilizer groups, as we will discuss below.

The semistable locus $X^{\rm ss}(\ell) \subset X$ is defined to be the set of points $x\in X$ such that there exists an invariant global section $s\in \Gamma(\cO_X\otimes k \langle d\ell \rangle)^G$ for some $d>0$ which does not vanish at $x$, where $k \langle d\ell \rangle$ is the one dimensional representation with character $d\ell$. The Hilbert-Mumford criterion gives a more explicit description of the semistable locus via the formula
\[
X^{\rm us}(\ell) = \bigcup_{\lambda \in N, \text{ s.t. } \langle \lambda, \ell \rangle <0} G \cdot X^{\lambda \geq 0},
\]
where $X^{\rm us}(\ell) := X \setminus X^{\rm ss}(\ell)$ is the \emph{unstable locus}. In fact, only finitely many subspaces arise as $X^{\lambda \geq 0}$ for some $\lambda \in N$, so finitely many $\lambda$ suffice to cover the unstable locus.

We will make use of some standard results in the theory of variation of GIT quotients, developed in \cite{DH98}. First, we define $X^{\rm ss}(\ell)$ more generally for $\ell \in M_\bR^W$ using the fact that $X^{\rm ss}(\ell)$ for $\ell \in M^W$ depends only on the cell in which $\ell$ lies with respect to a decomposition of $M_\bR^W$ into a finite union of rational polyhedral cones. We can thus think about varying $\ell$ continuously. If $\ell_t \in M_\bR^W$ for $t \in [0,1]$ is a continuous family of real characters, then $X^{\rm ss}(\ell_t)$ is constant for $0<t\ll1$, and $X^{\rm ss}(\ell_t) \subset X^{\rm ss}(\ell_0)$ for $0 < t\ll 1$ with equality if points of $X^{\rm ss}(\ell_0)$ have finite stabilizers.

\begin{ex}\label{ex:vgit_near_anticanonical}
Under \Cref{H:main_setup}, we will consider $X^{\rm ss}(\omega^\ast) /G$. If points of $X^{\rm ss}(\omega^\ast)$ already have finite stabilizers, then $X^{\rm ss}(\omega^\ast + \epsilon \ell) = X^{\rm ss}(\omega^\ast)$ for any $\ell \in M^W_\bR$ and $0<\epsilon\ll 1$. However, if $X^{\rm ss}(\omega^\ast)$ has points with positive stabilizers, then $X^{\rm ss}(\omega^\ast + \epsilon \ell)$ for $0 < \epsilon \ll 1$ will depend on $\ell$. If $V \subset M_\bR^W$ is the linear subspace spanned by the smallest cone in the wall-and-chamber decomposition of $M_\bR^W$ coming from GIT, then $M_\bR^W / V$ inherits a wall-and-chamber decomposition into a finite union of rational polyhedral cones such that $X^{\rm ss}(\omega^\ast + \epsilon \ell)$ is constant for all $\ell$ whose image in $M^W_\bR$ lies in one of these cones.
\end{ex}

\subsection{The barrel window} \label{S:windows} 
Here we introduce some subsets of $M_\bR$, which we refer to as the ``cylinder" and ``barrel" windows, that are relevant for our main results. We work under \Cref{H:main_setup}.

\begin{defn}\label{D:conormal_weight}
Associate to any cocharacter $\lambda\in N_{\mathbb{R}}$ the $T$-weight
\[
    \zeta_{\lambda}:= -\det \left( X^{\lambda \leq 0} \right) + \det \left(\mathfrak{g}^{\lambda <0} \right) ,
\]
where $\mathfrak{g}$ is the Lie algebra of $G$ and the expression
$\det(U) := \wedge^{\dim U} U$ is the sum of the weights (with multiplicity) appearing in $U$. Define \[\eta_\lambda:= \left\langle \lambda, \zeta_\lambda \right\rangle.\]
More concretely, $\eta_\lambda = - \sum_{\beta \in \Psi} \min(0,\langle \lambda,\beta \rangle) + \sum_{\alpha \in \Phi} \min(0,\langle \lambda, \alpha \rangle)$, where $\Psi$ is the set of weights of $X$ and $\Phi$ the set of roots of $G$.
\end{defn}

If $\lambda$ is a one-parameter subgroup associated to a KN-stratum $S_\lambda$ with center $Z_\lambda$ in the sense of \cite{T00}, $\eta_\lambda$ is the total $\lambda$-weight of the conormal bundle of $S_\lambda$ in $X$ restricted to $Z_\lambda$.

\begin{defn}[$\lambda$-strip] \label{D:lambda_strip}
For any $\lambda\in N_\bR$, we refer to the closed subset
\[
B_\lambda := \left\{ \chi \in M_\bR : \lvert \left\langle \lambda,\chi\right\rangle \rvert  \le  \eta_{\lambda}/2 \right\} \subset M_\bR
\]
as the $\lambda$-\emph{strip}. Note that either $B_{\lambda} \subset B_{-\lambda}$ or $B_{-\lambda} \subset B_{\lambda}$, with equality if and only if $\eta_{-\lambda}=-\eta_{\lambda}$.
\end{defn}

\begin{defn}[Barrel window] \label{D:window}
Let $\lambda_0 \in N_\bR^W$ be a cocharacter that pairs strictly negatively with all weights of $X$. For any $\ell \in M_\bR^W$ with $\langle \lambda_0,\ell \rangle \neq 0$, we define the \emph{cylinder window} to be the closed subset
\[\overline{\nabla}_{\ell} :=\left\{ \chi\in B_{\lambda_0} :  \lvert \left\langle \lambda',\chi\right\rangle \rvert \leq \frac{\eta_{\lambda'}}{2} \text{ for all } \lambda' \in \ell^\perp \right\},\]
and we define the \emph{barrel window} to be the non-closed subset
\[\nabla_{\ell} := \left\{ \chi\in B_{\lambda_0} : \forall \lambda' \in \ell^{\perp}, \left\{
\begin{array}{l}
\lvert \left\langle \lambda',\chi\right\rangle \rvert < \frac{\eta_{\lambda'}}{2} ,\text{ or} \vspace{0.03in}\\
\left\langle \lambda',\chi\right\rangle = \frac{\eta_{\lambda'}}{2} \text{ and } \langle \lambda_0,\chi \rangle \leq \frac{\langle \lambda_0, \zeta_{\lambda'} \rangle}{2} , \text{ or} \vspace{0.03in}\\
\left\langle \lambda',\chi\right\rangle = \frac{-\eta_{\lambda'}}{2} \text{ and } \langle \lambda_0,\chi \rangle \geq  \frac{-\langle \lambda_0, \zeta_{\lambda'} \rangle}{2}
\end{array}
\right.  \right\}. \]

When $\ell = \omega^\ast$ and $\lambda_0 \in N_\bR^W$ is a cocharacter that satisfies the condition in \Cref{H:main_setup}, we will simply use $\overline{\nabla} := \overline{\nabla}_{\omega^\ast}$ and $\nabla := \nabla_{\omega^\ast}$.
\end{defn}

\begin{rem}
The cylinder window is the intersection of the $\lambda_0$-strip with the preimage of a polytope $\overline{\nabla}' \subset M_\bR/\bR\ell$. If we consider $X$ as a representation of $\ker(\ell) \subset G$, then under the identification of $M_\bR/\bR\ell$ with the weight space of $\ker(\ell)$, the polytope $\overline{\nabla}'$ corresponds to the one studied in \cite{HLS20} and it is related to the zonotope studied in \cite{SVdB17} if $X$ is quasi-symmetric (e.g. self-dual) as a representation of $\ker(\ell)$. 
\end{rem}

\begin{ex} \label{ex:grassmannian_N_6}
Let $G=\GL_2$ over the field $k$ and consider the $G$-linear representation $X=(k^2)^{\oplus 6}$. The weights of $X$ are $(1,0), (0,1)\in M\cong \bZ^2$ each with multiplicity $6$. The roots of $G$ are $(1,-1),(-1,1)$, where we take $(1,-1)$ to be the positive root. The subset $(\omega^\ast)^\perp\subset N_\bR$ consists of positive multiples of the cocharacters $\lambda'_1=(-1,1)$ and $\lambda'_2=(1,-1)$. 
Take $\lambda_0=(-1,-1)\in N_\bR^W$. A direct computation yields $\zeta_{\lambda_0} = (-6,-6)$, $\zeta_{\lambda_1'} = (-5,-1)$, and $\zeta_{\lambda'_2}=(-1,-5)$, thus $\eta_{\lambda_0} = 12$ and $\eta_{\lambda'_1}=\eta_{\lambda'_2}=4$. The resulting cylinder and barrel windows are described in \Cref{fig:grassmannian_N_6}.

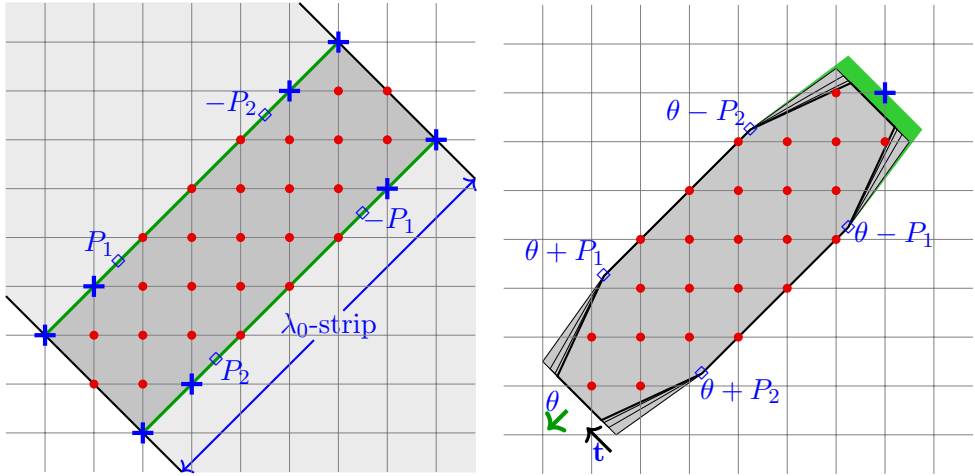
\begin{figure}[ht]
\begin{tabular}{ll}
\begin{tikzpicture}[scale=.60]

\filldraw[black!8!white] (-4.8, -1.2) -- (-4.8, 4.8) -- (1.2, 4.8) -- (4.8, 1.2) -- (4.8, -4.8) -- (-1.2,-4.8) -- cycle;
\filldraw[black!23!white] (-4,-2) -- (2,4) -- (4,2) -- (-2,-4) -- cycle;
\draw[-, very thick, green!60!black] (-4.0,-2.0) -- (2.0,4.0);
\draw[-, very thick, green!60!black] (4.0,2.0) -- (-2.0,-4.0);
\draw[-, thick, black] (-1.2,-4.8) -- (-4.8, -1.2);
\draw[-, thick, black] (1.2,4.8) -- (4.8, 1.2);

\draw[->, thick, blue] (2.1, -1.5) -- (4.8, 1.2);
\draw[->, thick, blue] (1.5, -2.1) -- (-1.2, -4.8);
\node[blue] at (1.8,-1.8) {$\lambda_0$-strip};

\draw[step=1, ultra thin, color=gray] (-4.8,-4.8) grid (4.8,4.8);

\foreach \i in { (-1,1), (-2,-1), (-1,0), (0,1), (1,2),  (-1,-1), (0,0), (1,1), (2,2), (-1,-2), (0,-1), (1,0) ,(2,1), (1,-1),(0,2),(2,0), (-2,0), (0,-2), (-2,-2),(-3,-2),(-2,-3), (-3,-3), (3,3), (2,3), (3,2)}
{ \filldraw[red!90!black] \i circle (2.3pt); }

\draw (-3,-1) pic[rotate=45, blue] {cross=4pt};
\draw (-1,-3) pic[rotate=45,blue] {cross=4pt};
\draw (-4,-2) pic[rotate=45,blue] {cross=4pt};
\draw (-2,-4) pic[rotate=45,blue] {cross=4pt};
\draw (3,1) pic[rotate=45,blue] {cross=4pt};
\draw (1,3) pic[rotate=45,blue] {cross=4pt};
\draw (4,2) pic[rotate=45,blue] {cross=4pt};
\draw (2,4) pic[rotate=45,blue] {cross=4pt};

\node[blue] at (-2.5,-0.5) {$\diamond$};
\node[blue] at (-0.5,-2.5) {$\diamond$};
\node[blue] at (2.5,0.5) {$\diamond$};
\node[blue] at (0.5,2.5) {$\diamond$};

\node[blue] at (2.7+0.35,0.6-0.25) {$-P_1$};
\node[blue] at (-2.5-0.39,-0.4+0.28) {$P_1$};
\node[blue] at (0-0.20,2.4+0.35) {$-P_2$};
\node[blue] at (-0.3+0.20,-2.4-0.35) {$P_2$};
\end{tikzpicture}
&
\begin{tikzpicture}[scale = .60]

\filldraw[green!60!gray] (-3.75 ,-2.25) -- (-2.5-0.25,-0.5-0.25) -- (0.5-0.25,2.5-0.25) -- (2.25,3.75) -- (3.75,2.25) -- (2.5-0.25,0.5-0.25) -- (-0.5-0.25,-2.5-0.25) -- (-2.25,-3.75) -- cycle;
\filldraw[black!21!white] (-3.75-0.25,-2.25-0.25) -- (-2.5-0.25,-0.5-0.25) -- (0.5-0.25,2.5-0.25) -- (2.25-0.25,3.75-0.25) -- (3.75-0.25,2.25-0.25) -- (2.5-0.25,0.5-0.25) -- (-0.5-0.25,-2.5-0.25)  -- (-2.25-0.25,-3.75-0.25) -- cycle;

\draw[-, ultra thin , color=black] (-3.75-0.25,-2.25-0.25) -- (-2.5-0.25,-0.5-0.25) -- (0.5-0.25,2.5-0.25) -- (2.25-0.25,3.75-0.25) -- (3.75-0.25,2.25-0.25) -- (2.5-0.25,0.5-0.25) -- (-0.5-0.25,-2.5-0.25)  -- (-2.25-0.25,-3.75-0.25) -- cycle;
\draw[-, very thin, color=black] (-3.75 +0.15-0.25,-2.25-0.15-0.25) -- (-2.5-0.25,-0.5-0.25) -- (0.5-0.25,2.5-0.25) -- (2.25+0.15-0.25,3.75-0.15-0.25) -- (3.75-0.15-0.25,2.25+0.15-0.25) -- (2.5-0.25,0.5-0.25) -- (-0.5-0.25,-2.5-0.25)  -- (-2.25-0.15-0.25,-3.75+0.15-0.25) -- cycle;
\draw[-, thin, color=black] (-3.75 +0.25-0.25,-2.25-0.25-0.25) -- (-2.5-0.25,-0.5-0.25) -- (0.5-0.25,2.5-0.25) -- (2.25+0.25-0.25,3.75-0.25-0.25) -- (3.75-0.25-0.25,2.25+0.25-0.25) -- (2.5-0.25,0.5-0.25) -- (-0.5-0.25,-2.5-0.25)  -- (-2.25-0.25-0.25,-3.75+0.25-0.25) -- cycle;
\draw[-, thick, color=black] (-3.75 +0.3-0.25,-2.25-0.3-0.25) -- (-2.5-0.25,-0.5-0.25) -- (0.5-0.25,2.5-0.25) -- (2.25+0.3-0.25,3.75-0.3-0.25) -- (3.75-0.3-0.25,2.25+0.3-0.25) -- (2.5-0.25,0.5-0.25) -- (-0.5-0.25,-2.5-0.25)  -- (-2.25-0.3-0.25,-3.75+0.3-0.25) -- cycle;

\draw[step=1, ultra thin, color=gray] (-4.8,-4.8) grid (4.8,4.8);

\draw[->, very thick, black] (-2.35-0.25,-4.10-0.25) -- (-2.85-0.25,-3.6-0.25);
\draw[->, ultra thick, color = green!60!black] (-3.50,-3.50) -- (-3.9,-3.9);
\node[blue] at (-2.6-0.25,-4.1-0.25) {\textbf{t}};
\node[blue] at (-3.8,-3.3) {$\theta$};

\foreach \i in { (-1,1), (-2,-1), (-1,0), (0,1), (1,2),  (-1,-1), (0,0), (1,1), (2,2), (-1,-2), (0,-1), (1,0) ,(2,1), (1,-1),(0,2),(2,0), (-2,0), (0,-2), (-2,-2),(-3,-2),(-2,-3), (-3,-3), (2,3), (3,2)}
{ \filldraw[red!90!black] \i circle (2.3pt); }
\draw (3,3) pic[rotate=45,blue] {cross=4pt};

\node[blue] at (-2.5-0.25,-0.5-0.25) {$\diamond$};
\node[blue] at (-0.5-0.25,-2.5-0.25) {$\diamond$};
\node[blue] at (2.5-0.25,0.5-0.25) {$\diamond$};
\node[blue] at (0.5-0.25,2.5-0.25) {$\diamond$};

\node[blue] at (3.1+0.35-0.25,0.6-0.25-0.25) {$\theta-P_1$};
\node[blue] at (-2.9-0.39-0.25,-0.3+0.28-0.25) {$\theta+P_1$};
\node[blue] at (-0.40-0.25,2.5+0.35-0.25) {$\theta-P_2$};
\node[blue] at (0.1+0.20-0.25,-2.5-0.35-0.25) {$\theta+P_2$};
\end{tikzpicture}
\end{tabular}

\caption{\footnotesize The diagram to the left is the cylinder window $\overline{\nabla}$ for \Cref{ex:grassmannian_N_6}. The $T$-weights indicated by blue crosses are those in the cylinder window that are excluded from the barrel window, thus the $T$-weights indicated by red dots are precisely those in the barrel window $\nabla$. We have marked the special weights $P_i := \zeta_{\lambda_i}/2$ for $i=1,2$ that become relevant in the passage from the cylinder window to the barrel window as is depicted in the diagram to the right. This diagram indicates a perturbation by some $\theta \in M_\bR^W$ of the regions  $\overline{\nabla}_{\ell,t}$ defined in \Cref{R:t_perturbed_regions} for selected values of $0 < t \ll 1$. In this particular example, $\theta= s \omega^\ast$, where $-s$ is a small positive number. Here the weight indicated by a blue cross is excluded in the $\theta$-perturbed barrel window $\theta + \nabla$.} \label{fig:grassmannian_N_6}
\end{figure}
\end{ex}

\begin{defn}[$\lambda_0$-generic] \label{D:lambda_0_generic}
Given a pair $(\lambda_0,\ell) \in N_\bR^W \times M_\bR^W$ such that $\langle \lambda_0, \ell \rangle \neq 0$, we say that $\theta \in M_\bR^W$ is \emph{$\lambda_0$-generic} with respect to $\ell$ if $(\theta + H) \cap M = \emptyset$ for every affine hyperplane $H$ of the form $\{\chi \in M_\bR : \langle \lambda_0, \chi-\zeta/2\rangle = 0\}$, where $\zeta = \zeta_{\lambda'}$ for some $\lambda' \in \ell^\perp$ (including $\lambda'=0$). We will say $\theta$ is $\lambda_0$-generic if it is $\lambda_0$-generic with respect to $\omega^\ast$.
\end{defn}

The notion of $\lambda_0$-generic is equivalent to $\theta$ avoiding all of the hyperplanes $\chi + \zeta_{\lambda'}/2 + \lambda_0^\perp$ for all $\chi \in M$ and $\lambda' \in \ell^\perp$. If $\lambda_0$ is rational, then this hyperplane arrangement is locally finite. In fact, the image of the lattice $M \subset M_\bQ$ under the linear map $\langle \lambda_0, - \rangle : M^\bQ \to \bQ$ will be a free subgroup of rank $1$, $\Gamma \subset \bQ$, and the hyperplane arrangement is contained in the preimage of $\frac{1}{2} \Gamma$ under this projection.

\begin{rem}[A characterization of the barrel window] \label{R:t_perturbed_regions}
Choose a $W$-invariant norm on $N_\bQ$. Let us rescale $\lambda_0$ so that $\lvert \lambda_0 \rvert=1$. For $t \ge 0$, define the closed subset
\[\overline{\nabla}_{\ell,t} =\left\{ \chi\in \overline{\nabla}_\ell :  \lvert \left\langle \lambda' + t\lambda_0,\chi\right\rangle \rvert \leq \frac{\langle \lambda'+t\lambda_0, \zeta_{\lambda'} \rangle}{2} \text{ } \forall \lambda' \in \ell^\perp  \text{ with } \lvert \lambda' \rvert = 1 \right\}.\]
One can show that if $\theta \in M_\bR^W$ is $\lambda_0$-generic with respect to $\ell$, then 
\[
\left(\theta + \overline{\nabla}_{\ell,t} \right) \cap M = \left(\theta + \nabla_\ell \right) \cap M \text{ for all } 0<t \ll 1.
\]
For $t$ small, the region $\overline{\nabla}_{\ell,t}$ represents a slight narrowing of the cylinder window $\overline{\nabla}_\ell$ at the top and bottom (see \Cref{fig:grassmannian_N_6}), which is why we refer to $\nabla_\ell$ as the ``barrel" window.
\end{rem}

\begin{proof}
Let $S$ denote the unit sphere in $\ell^\perp$. For each $\chi \in M_\bR$, define the function $F_{\chi -\theta} : [0, \infty) \times S \rightarrow \bR$ as
\[
F_{ \chi- \theta} (t, \lambda')=  \lvert \langle \lambda' + t \lambda_0, \chi -\theta \rangle \rvert - \langle \lambda' + t \lambda_0, \zeta_{\lambda'}/2 \rangle .
\]
Note that $F_{\chi - \theta}$ is continuous in both $t$ and $\lambda$.

A direct computation yields $\left( \theta + \overline{\nabla}_{\ell, t} \right) \cap M \subset \left( \theta + \nabla_\ell \right) \cap M$ for any $t>0$ and any $\theta \in M_\bR^W$. The result then follows from the following observations
\begin{enumerate}
    \item For $\chi \in \theta + \overline{\nabla}_\ell$ and any $\lambda'\in S$, if  $0 < t \le s$, then $F_{\chi -\theta} (t, \lambda') > 0$ implies $F_{\chi -\theta} (s, \lambda') > 0$.
    \item If $\theta$ is $\lambda_0$-generic with respect to $\ell$, then the containment 
    \begin{equation} \label{eq:t_regions_2} \left( \theta +  \bigcup\nolimits_{t>0} \overline{\nabla}_{\ell, t} \right) \cap M \subset \left( \theta + \nabla_\ell \right) \cap M
    \end{equation}
    is equality. 
\end{enumerate}
Indeed, assuming (1) holds implies that $\theta + \overline{\nabla}_{\ell, s} \subset \theta + \overline{\nabla}_{\ell, t}$ if $0 < t \leq s$. There are finitely many $T$-weights in $\theta + \bigcup_{t>0} \overline{\nabla}_{t,\ell}$, so one can find a $\tau > 0$ with $\left( \theta + \bigcup_{t>0} \overline{\nabla}_{\ell,t} \right) \cap M = \left( \theta + \overline{\nabla}_{\ell, s} \right) \cap M$ for all $s \in (0, \tau]$. The result then follows if (2) holds. 

\medskip
\noindent{\textit{Proof of (1) and (2)}:}
\medskip

The first claim follows by a direct calculation. For the second claim, let $\chi$ be a $T$-weight in $\theta + \left( \nabla_\ell \setminus \bigcup_{t>0} \overline{\nabla}_{\ell, t} \right)$. Observe that $\chi \not\in \theta + \rm Int \overline{\nabla}_{\ell}$. Indeed, $F_{ \chi- \theta} (t, \lambda')$
is continuous in both $t$ and $\lambda$ and by definition, $\chi - \theta \in \rm Int \overline{\nabla}_{\ell}$ if and only if $F_{ \chi - \theta} (0, \lambda') < 0$ for all $\lambda' \in S$. In particular, if $\chi - \theta \in \rm Int \overline{\nabla}_{\ell} $, then because $S$ is compact, $F_{\chi - \theta} (t,\lambda') < 0$ for all $0 \le t \ll 1$, so $\chi -\theta \in \overline{\nabla}_{\ell, t}$ for $t$ sufficiently small resulting in a contradiction.  

The fact that $\chi -\theta \not\in \rm Int \overline{\nabla}_{\ell}$ implies that the set 
\[
S_{ \chi -\theta} := \left\{ \lambda' \in S : \lvert \langle \lambda', \chi -\theta \rangle \rvert - \langle \lambda', \zeta_{\lambda'}/2 \rangle  = 0 \right\}
\]
is non-empty. We claim that there is a cocharacter $\lambda' \in S_{\chi-\theta}$ satisfying 
\begin{equation} \label{eq:t_regions_3} \lvert \langle \lambda_0, \chi -\theta \rangle \rvert - \langle \lambda_0, \zeta_{\lambda'}/2 \rangle = 0.\end{equation}
Because $\chi$ is a $T$-weight in $\theta + \left( \nabla_\ell \setminus \bigcup_{t>0} \overline{\nabla}_{\ell, t} \right)$ and because (1) holds, one has that for all $t>0$, there is a $\lambda'\in S$ (possibly depending on $t$) satisfying $F_{\chi -\theta}(s,\lambda') > 0$ for all $s \ge t$. Let $\{ t_i \}_{i=1}^\infty \subset (0,1]$ be a sequence converging to $0$ and for each $t_i$, let $\lambda'_i \in S$ be a cocharacter satisfying $F_{\chi -\theta}(s,\lambda'_i)  > 0$ for all $s \ge t_i$. Let $\lambda'$ denote the limit of a convergent subsequence of $\{ \lambda'_i \}$. Then $\lambda'$ satisfies  
\[ F_{\chi - \theta} (s, \lambda') \ge 0 \textrm{ for all } s \ge 0. \]
This, together with the fact that $\lvert \langle \lambda' , \chi -\theta \rangle \rvert \le \langle \lambda' , \zeta_{\lambda'}/2 \rangle$ (because $\chi -\theta \in \overline{\nabla}_{\ell}$ and $\lambda' \in S$) implies (a) $\lvert \langle \lambda' , \chi -\theta \rangle \rvert = \langle \lambda' , \zeta_{\lambda'}/2 \rangle$ and (b) $\lvert \langle  \lambda_0, \chi -\theta \rangle \rvert \ge \langle \lambda_0, \zeta_{\lambda'}/2 \rangle$. By assumption, $\chi - \theta \in \nabla_\ell$, so (a) implies $\lvert \langle  \lambda_0, \chi -\theta \rangle \rvert \le \langle \lambda_0, \zeta_{\lambda'}/2 \rangle$, which in addition to (b) implies $\lambda'$ satisfies \eqref{eq:t_regions_3}.

The existence of a cocharacter $\lambda' \in S_{\chi -\theta}$ satisfying \eqref{eq:t_regions_3} implies that the $T$-weight $\chi$ lies in one of the hyperplanes $\theta \pm \zeta_{\lambda'}/2 + \lambda_0^\perp$ and this is ruled out by the genericity assumption on $\theta$. Thus, there are no $T$-weights in $\theta + \left( \nabla_\ell \setminus \bigcup_{t>0} \overline{\nabla}_{\ell, t} \right)$.
\end{proof}

\subsection{A vanishing criterion for local cohomology}

Here we use the results of \cite{VdB91} to state a vanishing criterion for local cohomology. Recall from \Cref{D:conormal_weight} the canonical weight $\zeta_\lambda$ associated to a coweight $\lambda \in N$.

\begin{prop} \label{P:vanishing}
Assume \Cref{H:main_setup}. Let $\ell \in M_\bQ^W$ be such that $\langle \lambda_0, \ell \rangle < 0$ and let $U$ be a representation of $G$. If for every weight $\chi$ appearing in $U$ and for all cocharacters $\lambda' : \bG_m \to T_\ell$, there is a $\tau>0$ such that for all $ 0 < t < \tau$,
    \[
     \langle \lambda'+t \lambda_0, \chi-\zeta_{\lambda'} \rangle < 0,
    \]
then $R\Gamma_{X^{\rm us}}(\cO_X \otimes U)^G = 0$, where $X^{\rm us} = X^{\rm us}(\ell)$ is the $\ell$-unstable locus. Furthermore, if $G$ is connected and $U$ is irreducible, it suffices to check these conditions for the lowest weight $\chi$ appearing in $U$, and to verify the condition only for $\lambda' : \bG_m \to T_\ell$ that act with weights $\leq 0$ on the Lie algebra of $B$.
\end{prop}

\begin{rem} \label{R:vanishing}
Because $\lambda_0$ is central, the condition on $\chi$ and $\lambda'$ can be rephrased as saying that $0 \geq \langle \lambda', \chi - \zeta_{\lambda'} \rangle = \langle \lambda', \chi+\det(X^{\lambda'<0}) + \det(\mathfrak{g}^{\lambda'>0}) \rangle$, and if equality holds then $0 > \langle \lambda_0, \chi - \zeta_{\lambda'} \rangle = \langle \lambda_0 , \chi + \det(X^{\lambda'\leq 0}) \rangle$. In the special case where $\lambda'$ is the trivial cocharacter, the condition simplifies to $\langle \lambda_0, \chi + \det(X) \rangle < 0$.
\end{rem}

The proof depends on a slight generalization of the classical characterization of semistability for a point in $\bP^n$ under the action of a reductive group.

\begin{lem} \label{L:instability}
Assume \Cref{H:main_setup} and let $\ell \in M_\bR^W$ be such that $\langle \lambda_0, \ell \rangle < 0$, then $x \in X$ is $\ell$-unstable if and only if $0 \in \overline{\{G^\circ_\ell \cdot x\}}$.
\end{lem}
\begin{proof}
By the Hilbert-Mumford criterion, it suffices to show that $\ell$-instability is equivalent to the existence of a cocharacter $\lambda' : \bG_m \to G^\circ_\ell$ such that $\lim_{t \to 0} \lambda(t) \cdot x = 0$. In one direction, if there is such a cocharacter $\lambda' : \bG_m \to G^\circ_\ell$, then for $N \gg 0$, $\lim_{t \to 0} \lambda'(t)^N \lambda_0(t) x = 0$ as well, and it is destabilizing with respect to $\ell$ by hypothesis.

Conversely, after a positive rescaling, any $\ell$-destabilizing cocharacter for $x$ has the form $\lambda = \lambda' + a \lambda_0$, where $\lambda'$ is a cocharacter of $G^\circ_\ell$ and $a>0$. We can decompose $x$ into joint eigenvectors
\[
x = \sum_{v,w\in \bZ, w<0} x_{v,w},
\]
where $x_{v,w}$ has weight $v$ for $\lambda'$ and weight $w$ for $\lambda_0$. Then the fact that $\lim_{t\to 0} \lambda(t) \cdot x$ exists means that $v + aw \geq 0$ whenever $x_{v,w} \neq 0$, and hence that $v \geq -aw>0$. It follows that $\lim_{t\to 0} \lambda'(t)\cdot x = 0$.
\end{proof}

\begin{proof}[Proof of \Cref{P:vanishing}]
Let $G^\circ \subset G$ denote the identity component. The homomorphism $\bG_m \times G_\ell^\circ \to G^\circ$ taking $(z,g) \mapsto \lambda_0(z) g$ is surjective. For any $G$-representation $M$, we have
\[
M^G = (M^{G^\circ})^{G/G^\circ} = (M^{\bG_m \times G^\circ_\ell})^{G/G^\circ}.
\]
In addition, the unstable locus $X^{\rm us}(\ell)$ is unaffected by isogenies of $G$, so it suffices to show the vanishing of $(R\Gamma_{X^{\rm us}}(\cO_X \otimes U))^{\bG_m \times G^\circ_\ell}$. The criterion of the proposition is also unaffected by isogenies, so we may therefore replace $G$ with $\bG_m \times G^\circ_\ell$ for the remainder of the proof. We may also assume $T=\bG_m \times T_\ell$ and $B=\bG_m \times B_\ell$.

We use the methods of \cite{VdB91} to bound the weights of $R\Gamma_{X^{\rm us}}(\cO_X)$ as a $G$-representation.

First of all, \Cref{L:instability} implies that the $\ell$-unstable locus for the action of $G$ on $X$, in the sense of GIT, agrees with the unstable locus for the action of $G^\circ_\ell$ on $X$ as studied in \cite{VdB91}. Then \cite{VdB91}*{Cor.~6.8} says that every irreducible $G^\circ_\ell$-representation that occurs in $R\Gamma_{X^{\rm us}}(\cO_X)$ has a highest weight of the form
\begin{equation} \label{E:highest_weight}
\chi_{hi} = \chi' + \sum_{\alpha \in S \subset \Phi} \alpha
\end{equation}
where $S \subset \Phi$ is a subset of the roots of $G^\circ_\ell$, and $\chi'$ is some character occurring in the $T_\ell$-representation $R\Gamma_{X^{\lambda'>0}}(\cO_X)$ for some cocharacter $\lambda'$ of $T_\ell$. The proof of this claim uses the spectral sequence associated to a certain double complex, constructed in \cite{VdB91}*{Lem.~5.2.2}, whose differentials are equivariant for the larger group $G$, not just for $G^\circ_\ell$. As a result, the arguments used to establish \cite{VdB91}*{Cor.~6.8} apply verbatim to conclude that every irreducible $G$-representation occurring in $R\Gamma_{X^{\rm us}}(\cO_X)$ has highest weight of the form \eqref{E:highest_weight}, where $\chi'$ is a character for the maximal torus $T \subset G$ that occurs in the $T$-representation $R\Gamma_{X^{\lambda'>0}}(\cO_X)$ for some cocharacter $\lambda'$ of $T_\ell$, and we regard roots of $G^\circ_\ell$ canonically as weights for $T=\bG_m \times T_\ell$ by letting the first factor of $\bG_m$ act trivially.

If $U$ is irreducible with lowest weight $\chi$, then to show that $R\Gamma_{X^{\rm us}}(\cO_X \otimes U)^G=0$, it suffices to verify that the highest weight of $U^\ast$, which is $-\chi$, can not have the form \eqref{E:highest_weight} discussed in the previous paragraph. For any closed subscheme $i : S \hookrightarrow X$ defined by an ideal $I_S \subset \cO_X$, the formula $R\Gamma_S \cO_X \cong \colim \RHom(\cO_X/I_S^n,\cO_X)$ gives a convergent non-negative filtration whose $n^{th}$ associated graded complex is $\RHom(i_\ast(I_S^n/I_S^{n+1}), \cO_X)$ for $n\geq 0$. The closed immersion $X^{\lambda' >0} \hookrightarrow X$ is regular, so Grothendieck's formula for $i^{!}(\cO_X)$ identifies the associated graded complex of $R\Gamma_{X^{\lambda'>0}}(\cO_X)$ with
\[
\cO_{X^{\lambda'>0}} \otimes \Sym(X^{\lambda'\leq 0}) \otimes \det(X^{\lambda' \leq 0})[-\dim(X^{\lambda'\leq 0})].
\]
It therefore suffices to show that $-\chi$ does not appear as a weight in the $T$-representation
\begin{equation} \label{E:local_cohomology_gr}
\Sym\left((X^{\lambda'>0})^\ast \oplus X^{\lambda' \leq 0}\right) \otimes \det(X^{\lambda' \leq 0}) \otimes \wedge^\ast \mathfrak{g}
\end{equation}
for any cocharacter $\lambda'$ of $T_\ell$ (including $\lambda'=0$). The highest $\lambda'$-weight appearing in \eqref{E:local_cohomology_gr} is $-\eta_{\lambda'} = \langle \lambda' , \det(X^{\lambda'\leq 0}) + \det(\mathfrak{g}^{\lambda'>0}) \rangle$, and the summand with this $\lambda'$-weight is
\begin{equation} \label{E:local_cohomology_gr2}
\Sym\left( X^{\lambda' = 0} \right) \otimes \det(X^{\lambda' \leq 0}) \otimes \det(\mathfrak{g}^{\lambda'>0}) \otimes \wedge^\ast (\mathfrak{g}^{\lambda'=0}).
\end{equation}
The highest $\lambda_0$-weight in \eqref{E:local_cohomology_gr2} is $\langle \lambda_0, \det(X^{\lambda'\leq 0}) \rangle < 0$. It follows that if either 1) $\langle \lambda', \chi \rangle < \eta_{\lambda'}$ or 2) $\langle \lambda', \chi \rangle = \eta_{\lambda'}$ and $\langle \lambda_0, \chi \rangle < - \langle \lambda_0, \det(X^{\lambda'\leq 0}) \rangle$, then $-\chi$ does not appear as a weight of \eqref{E:local_cohomology_gr}.

To complete the proof we must show that if $G$ is connected and $U$ is irreducible with lowest weight $\chi$, then all of the weights appearing in $U$ satisfy the condition with respect to any cocharacter of $T'$. By the Weyl character formula, all of the weights of $U$ lie in the convex hull of $W \cdot \chi$,\footnote{The Weyl character formula says that $\Ch(U) \cdot \sum_{w \in W} \varepsilon(w) w(\rho) = \sum_{w \in W} \varepsilon(w) w(\chi+\rho)$, where $\rho$ is the half-sum of negative roots. If one chooses a generic coweight $\lambda$ and keeps only the terms with the lowest pairing with $\lambda$, one gets $\Ch(U)^{\lambda\rm{-min}} \cdot \varepsilon(w) w(\rho) = \varepsilon(w) w(\chi+\rho)$, where $w \in W$ is the unique element such that $w(\lambda)$ pairs positively with all dominant weights. So for a dense set of coweights $\lambda$, the minimum $\lambda$ weight in $U$ is always $\geq \langle \lambda, w(\chi) \rangle$ for some $w \in W$. This implies that the convex hull of the weights in $\Ch(U)$ lies in the convex hull of $W \cdot \chi$.} so it suffices to verify the inequalities for the characters $w \chi$ with $w \in W$. Because $w \lambda_0 = \lambda_0$ and $w \zeta_{\lambda'} = \zeta_{w \lambda'}$, the inequality $\langle w \lambda' + t \lambda_0, \chi - \zeta_{w \lambda'} \rangle<0$ is equivalent to $\langle \lambda' + t \lambda_0, w^{-1} \chi - \zeta_{\lambda'} \rangle<0$, so we may assume $\lambda'$ pairs non-negatively with the Lie algebra of $B$. In that case, $\langle\lambda',w \chi\rangle \leq \langle \lambda',\chi \rangle$ and $\langle\lambda_0,w \chi\rangle = \langle \lambda_0,\chi \rangle$, so the inequality for the characters $w\chi$ follows from that for $\chi$.
\end{proof}

\begin{ex} \label{Ex:allowable_region}
Consider the $\GL_2$ representation $X = \Sym^3 \bC^2$ and take $\lambda_0 = (-1,-1)$. The subspace $(\omega^\ast) ^\perp \subset N_\bR$ consists of zero and positive multiples of $\lambda_1^{\omega^\ast} = (-1,1)$ and $\lambda_2^{\omega^\ast} = (1,-1)$, so these cocharacters suffice to describe a region in $M_\bR$ consisting of weights that satisfy the inequality in \Cref{P:vanishing}. Any $\GL_2$-representation whose weights lie in this region gives a vector bundle that has vanishing local cohomology with respect to $X^{\rm us}$. By \Cref{R:vanishing}, this ``allowable" region is
\[
    \left\{ \chi \in M_\bR : \textrm{for } \lambda' = 0, \lambda_1^{\omega^\ast}, \lambda_2^{\omega^\ast}, \left\{ \begin{array}{ll}
        \langle \lambda', \chi - \zeta_{\lambda'} \rangle < 0, \textrm{ or } \\
        \langle \lambda', \chi - \zeta_{\lambda'} \rangle = 0 \textrm{ and } \langle \lambda_0, \chi - \zeta_{\lambda'} \rangle < 0
    \end{array}  \right.\right\}  .
\]
In this particular example, $\zeta_{\lambda_1^{\omega^\ast}} = (-4,-2)$ and $\zeta_{\lambda_2^{\omega^\ast}} = (-2,-4)$. The resulting region is shown in \Cref{fig:allowable_region}. 

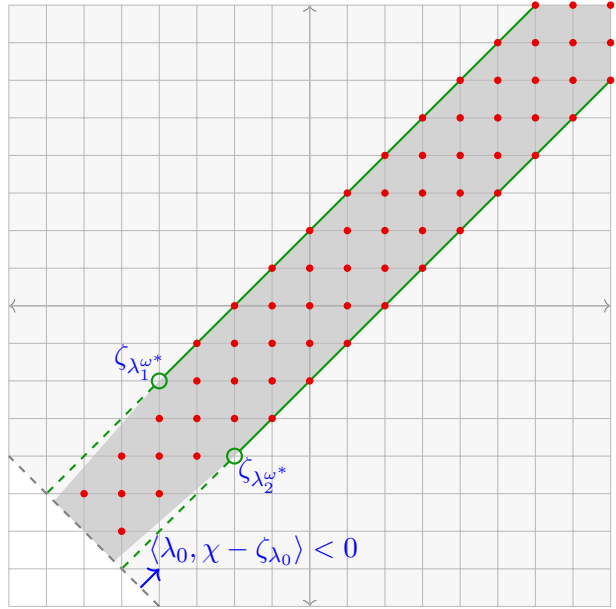
\begin{figure}[ht]
\begin{tikzpicture}[scale=1]

\fill[fill=gray,fill opacity=0.3]   (-3.4,-2.6) -- (-2,-1) -- (3,4) -- (4,4) -- (4,3) -- (-1,-2) -- (-2.6,-3.4) -- cycle;

\draw[<->, thin, color=gray] (-4,0) -- (4,0);
\draw[<->, thin, color=gray] (0,-4) -- (0,4);
\draw[step=0.5, thin, color=gray!50!white] (-4,-4) grid (4,4);

\draw[-, thick, green!60!black] (-1.93,-0.93) -- (3,4);
\draw[-, thick, green!60!black] (-0.93,-1.93) -- (4,3);
\draw[dashed, thick, green!60!black] (-2.07,-1.07) -- (-7/2,-5/2);
\draw[dashed, thick, green!60!black] (-1.07,-2.07) -- (-5/2, -7/2);

\draw[ thick, green!60!black] (-2,-1) circle (2.8 pt);
\draw[ thick, green!60!black] (-1,-2) circle (2.8 pt);
\node[blue] at (-4.5/2,-1.5/2) {$\zeta_{\lambda^{\omega^\ast}_1}$};
\node[blue] at (-1.2/2,-4.5/2) {$\zeta_{\lambda^{\omega^\ast}_2}$};

\draw[dashed, thick, color=gray] (-2,-4) -- (-4,-2);
\fill[fill=gray,fill opacity=0.06]   (-2,-4) -- (-4,-2) -- (-4,4) -- (4,4) -- (4,-4) -- cycle;
\draw[->, thick, color=blue] (-4.5/2,-7.5/2) -- (-4/2,-7/2);
\node[blue] at (-1.5/2,-6.5/2) {$\langle \lambda_0, \chi - \zeta_{\lambda_0} \rangle < 0$};

\foreach \i in { (-3/2, -1/2), (-2/2,0/2), (-1/2,1/2), (0/2,2/2), (1/2,3/2), (2/2,4/2), (3/2,5/2), (4/2,6/2), (5/2,7/2), (6/2,8/2), (-6/2,-5/2), (-5/2,-4/2), (-4/2,-3/2), (-3/2, -2/2), (-2/2,-1/2), (-1/2,0/2), (0/2,1/2), (1/2,2/2), (2/2,3/2), (3/2,4/2), (4/2,5/2), (5/2,6/2), (6/2,7/2), (7/2, 8/2), (-5/2,-5/2), (-4/2,-4/2), (-3/2,-3/2), (-2/2, -2/2), (-1/2,-1/2), (-0/2,0/2), (1/2,1/2), (2/2,2/2), (3/2,3/2), (4/2,4/2), (5/2,5/2), (6/2,6/2), (7/2,7/2), (8/2, 8/2), (-5/2,-6/2), (-4/2,-5/2), (-3/2,-4/2), (-2/2, -3/2), (-1/2,-2/2), (-0/2,-1/2), (1/2,0/2), (2/2,1/2), (3/2,2/2), (4/2,3/2), (5/2,4/2), (6/2,5/2), (7/2,6/2), (8/2, 7/2), (-1/2, -3/2), (-0/2,-2/2), (1/2,-1/2), (2/2,0/2), (3/2,1/2), (4/2,2/2), (5/2,3/2), (6/2,4/2), (7/2,5/2), (8/2,6/2)}
{ \filldraw[red!90!black] \i circle (1.2 pt); }

\end{tikzpicture}

\caption{\footnotesize This diagram illustrates the ``allowable" region for \Cref{Ex:allowable_region}. The red circles indicate the $T$-weights that satisfy the inequality in \Cref{P:vanishing}.} \label{fig:allowable_region}
\end{figure}

\end{ex}

\begin{cor} [Fully-faithfulness] \label{C:fully_faithful_general}
Assume \Cref{H:main_setup}, and let $\ell \in M_\bR^W$ be such that $\langle \lambda_0, \ell \rangle <0$. Let $\theta \in M_\bR^W$ be $\lambda_0$-generic with respect to $\ell$. If $V,W \in \Rep(G)$ are two representations whose weights are contained in $\theta + \nabla_\ell \subset M_\bR$, then restriction gives a quasi-isomorphism 
\[
\RHom_X(\cO_X \otimes V, \cO_X \otimes W)^G \cong \RHom_{X^{\rm ss}(\ell)}(\cO_{X^{\rm ss}(\ell)} \otimes V, \cO_{X^{\rm ss}(\ell)} \otimes W)^G.
\]
\end{cor}

\begin{proof}
This is equivalent to showing that $R\Gamma_{X^{\rm us}(\ell)}(\cO_X \otimes W \otimes V^\ast)^G=0$. Any weight $\chi$ of $W \otimes V^\ast$ has the form $\chi_2 - \chi_1$ for some $\chi_1,\chi_2 \in (\theta + \nabla_\ell) \cap M$. So it suffices to show that any such $\chi$ satisfies the inequalities of \Cref{P:vanishing}. We must show that for any $\lambda' \in \ell^\perp$, one has
\[
\langle \lambda', \chi_2 - \chi_1 \rangle \leq \eta_{\lambda'},
\]
and if equality holds then $\langle \lambda_0, \chi_2-\chi_1-\zeta_{\lambda'} \rangle < 0$.

The first inequality follows immediately from \Cref{D:window}, because $\langle \lambda', \chi_2 \rangle - \langle \lambda',\chi_1 \rangle \leq \lvert \langle \lambda', \chi_2 - \theta \rangle \rvert + \lvert \langle \lambda',\chi_1 - \theta \rangle \rvert$. Equality holds if and only if $\langle \lambda', \chi_2-\theta \rangle = \eta_{\lambda'}/2$ and $\langle \lambda', \chi_1-\theta \rangle = -\eta_{\lambda'}/2$. If this is the case, then \Cref{D:window} implies that $\langle \lambda_0, \chi_2-\theta \rangle \leq \langle \lambda_0, \zeta_{\lambda'} \rangle / 2$ and $\langle \lambda_0, \chi_1 - \theta \rangle \geq - \langle \lambda_0, \zeta_{\lambda'} \rangle / 2$. Subtracting the second inequality from the first gives $\langle \lambda', \chi_2 - \chi_1 \rangle \leq \langle \lambda_0, \zeta_{\lambda'} \rangle$, and equality holds if and only if $\langle \lambda_0, \chi_2 - \theta \rangle = \langle \lambda_0, \zeta_{\lambda'} \rangle/2$ and $\langle \lambda_0, \chi_1-\theta \rangle = - \langle \lambda_0, \zeta_{\lambda'} \rangle / 2$. Because $\theta$ is $\lambda_0$-generic, there is no weight $\chi_2 \in M$ such that $\langle \lambda_0,\chi_2-\theta-\zeta_{\lambda'}/2 \rangle = 0$, so $\langle \lambda', \chi_2 - \chi_1 \rangle < \langle \lambda_0, \zeta_{\lambda'} \rangle$.
\end{proof}

\subsection{Strong exceptional collections}

For an irreducible $G$-representation $U$, the central cocharacter $\lambda_0$ acts with constant weight on $U$, which we will denote $\wt_{\lambda_0}(U)$.

\begin{prop} \label{P:exceptional_collections}
Assume \Cref{H:main_setup}, let $\ell \in M_\bR^W$ be such that $\langle \lambda_0, \ell \rangle < 0$, and let $\theta \in M_\bR^W$ be $\lambda_0$-generic with respect to $\ell$ (\Cref{D:lambda_0_generic}). Let $U_1,\ldots,U_N$ denote the irreducible representations of $G$ whose weights lie in $\theta+\nabla_\ell \subset M_{\bR}$, indexed so that $\wt_{\lambda_0}(U_j) \leq \wt_{\lambda_0}(U_i)$ for all $i<j$. Then the $G$-equivariant locally free sheaves $\cO_{X^{\rm ss}(\ell)} \otimes U_1, \ldots, \cO_{X^{\rm ss}(\ell)} \otimes U_N$ are a strong exceptional collection in $\DCoh(X^{\rm ss} (\ell) /G)$.
\end{prop}
\begin{proof}
This follows immediately from \Cref{C:fully_faithful_general} and \Cref{L:big_collection} below.
\end{proof}

\begin{lem} \label{L:big_collection}
One can choose a total ordering on the set of irreducible representations of $G$ such that $U \prec W$ implies $\wt_{\lambda_0}(W)\leq \wt_{\lambda_0}(U)$. With this ordering, the locally free sheaves $\cO_X \otimes U$, where $U$ ranges over all irreducible representations of $G$, form a full strong exceptional collection in $\DCoh(X/G)$.
\end{lem}
\begin{proof}
The fact that such a total ordering exists is elementary -- one can use a lexicographic ordering first by $-\wt_{\lambda_0}(-)$, and then break ties using any other total ordering of irreducible representations. The fact that this forms a full exceptional collection follows from the computation for two irreducible representations $U$ and $W$
\begin{equation} \label{E:hom_spaces}
\RHom_{X/G}(\cO_X \otimes U, \cO_X \otimes W) \cong (\Sym(X^\ast) \otimes W \otimes U^\ast)^G.
\end{equation}
If $\wt_{\lambda_0}(W)>\wt_{\lambda_0}(U)$ then \eqref{E:hom_spaces} vanishes, because the $\lambda_0$ weights of $W \otimes U^\ast$ are strictly positive and the $\lambda_0$-weights of $\Sym(X^\ast)$ are non-negative by hypothesis. If $\wt_{\lambda_0}(U) = \wt_{\lambda_0}(W)$ then \eqref{E:hom_spaces} is isomorphic to $(W \otimes U^\ast)^G \cong \Hom_{\Rep(G)}(U,W)$, which vanishes if $U \ncong W$ and is isomorphic to $k$ if $U \cong W$.
\end{proof}

\section{Decorated quiver varieties and iterated fiber bundles}

Let $Z$ be a smooth and proper Deligne-Mumford stack with an action of a group $G$. Then for any Deligne-Mumford stack $Y$ with a principal $G$-bundle $P \to Y$, one can consider the associated fiber bundle $P \times^G Z \to Y$, whose fiber is $Z$.\footnote{We take the definition of a $G$-action on $Z$ to be the data of a morphism of algebraic stacks $\cZ \to BG$ along with an isomorphism $\pt \times_{BG} \cZ \cong Z$. Then by definition $P \times^G Z := Y \times_{BG} \cZ$.} We will assume that we have an exceptional collection $E_1,\ldots,E_n \in \DCoh(Z/G)$ that restricts to a full exceptional collection in $\DCoh(Z)$. When $G$ is a connected reductive group, $\pi_1(G)$ is free, and $Z$ is a projective scheme, any full exceptional collection in $\DCoh(Z)$ is the restriction of an exceptional collection in $\DCoh(Z/G)$ \cite{P11}*{Lemma 2.2}.

\begin{lem} \label{L:fec_bundle}
For any Deligne-Mumford stack $Y$ and any $G$-bundle $P \to Y$, the associated fiber bundle $P \times^G Z$ has a semiorthogonal decomposition
\[
\DCoh(P \times^G Z) = \langle \DCoh(Y) \boxtimes^G E_1,\ldots, \DCoh(Y) \boxtimes^G E_n \rangle,
\]
where $\DCoh(Y) \boxtimes^G E_i$ denotes the essential image of the fully faithful functor $p_1^\ast(-) \otimes p_2^\ast(E_i) : \DCoh(Y) \mapsto \DCoh(P \times^G Z)$, and $p_1 : P \times^G Z \to Y$ and $p_2 : P \times^G Z \to Z/G$ denote the canonical projections.
\end{lem}
\begin{proof}
By the conservative descent theorem of \cite{BS20}*{Thm.~B}, it suffices to prove the claim after base change along a faithfully flat morphism $Y' \to Y$. We may therefore assume that $Y = \Spec(R)$ is a smooth affine scheme and the $G$-bundle $P$ is trivial, so $P \times^G Z \cong \Spec(R) \times Z$. The projection formula implies that $\RHom_{\Spec(R) \times Z}(R \otimes_k E_i,R \otimes_k E_j) \cong R \otimes_k \RHom_Z(E_i,E_j)$, which vanishes if $j>i$ and is isomorphic to $R$ if $i=j$. The fact that $R$ split-generates $\DCoh(R)$ thus implies that the functors $(-) \otimes_k E_i$ are fully faithful and that their essential images are semiorthogonal. To complete the proof, it suffices to show that $R \otimes E_1,\ldots, R \otimes E_n$ split generate $\DCoh(\Spec(R) \times Z)$. This follows from the fact that $\Spec(R) \times Z \to Z$ is an affine morphism, because $F \in \Dqc(\Spec(R) \times Z)$ is zero if and only if its pushforward to $Z$ is zero.
\end{proof}

As a consequence of \Cref{L:fec_bundle}, if $Y$ has a full exceptional collection $\DCoh(Y) = \langle F_1, \ldots, F_m \rangle$, then the objects $F_i \boxtimes^G E_j$, with indices $(i,j)$ given the reverse lexicographic ordering, form a full exceptional collection in $\DCoh(P \times^G Y)$. If $\{E_i\}$ and $\{F_j\}$ are strongly exceptional, or consist of locally free sheaves, then the same holds for $E_i \boxtimes^G F_j$.

Our next goal is to illustrate a large class of examples in which one can apply \Cref{L:fec_bundle} to construct full exceptional collections.

\begin{defn}
A \emph{quiver} consists of a finite set $Q_0$ of ``vertices" and a finite set $Q_1$ of ``arrows," along with source and target maps $s,t:Q_1 \to Q_0$. A \emph{decorated quiver} consists of a quiver $Q$ along with a triple $(G_i,V_i,w_i)$ for all $i \in Q_0$, where $G_i$ is a reductive group, $V_i$ is a $G_i$-representation, and $w_i \in \bZ$ with $w_i \geq 0$. The \emph{representation space} of a decorated quiver is the vector space
\[
\Rep(Q) := \prod_{\alpha \in Q_1} \Hom(V_{s(\alpha)},V_{t(\alpha)}) \times \prod_{i \in Q_0} V_i^{w_i},
\]
which is naturally a representation of the group $G_Q := \prod_{i \in Q_0} G_i$. We will sometimes denote a decorated quiver $(Q_1,Q_0,\{(G_i,V_i,w_i\})$ simply by $Q$, when the context is clear.
\end{defn}

Let $Q = (Q_1,Q_0,\{(G_i,V_i,w_i\})$ be a decorated quiver. We assume the underlying quiver has no oriented cycles, so we may choose an identification $Q_0 \cong \{1,\ldots,N\}$ such that arrows only point from $i$ to $j$ for $j>i$. For each $i \in Q_0$, let
\begin{gather*}
\Rep(Q)_i := \Hom(k^{w_i} \oplus \bigoplus_{\substack{\alpha \in Q_1 \\ t(\alpha) = i}} V_{s(\alpha)}, V_{i}), \\
n_i := w_i + \sum_{\substack{\alpha \in Q_1 \\ t(\alpha) = i}} \dim(V_{s(\alpha)}).
\end{gather*}
Note that $\Rep(Q)_p$ is equipped with a natural action of $G_{\leq p} := \prod_{1 \leq i \leq p} G_i$, and $\Rep(Q)_i \cong V_i^{n_i}$ as a $G_i$-representation. The natural forgetful maps $\Rep(Q) \to \Rep(Q)_p$ induce an isomorphism $\Rep(Q) \cong \prod_p \Rep(Q)_p$.

\begin{prop} \label{P:iterated_fibration}
Let $Q$ be a decorated quiver with no oriented cycles, and vertices ordered as above, and assume that each $V_i$ has strictly negative weights with respect to some central cocharacter of $G_i$. For all $i=1,\ldots,N$, let $\ell_i$ be a character of $G_i$ such that $(V_i^{n_i})^{G_i\rm{-ss}}(\ell_i)$ has finite stabilizers in $G_i$, and let $\ell = a_1 \ell_1 + \cdots + a_N \ell_N$ for real numbers $a_i$ satisfying $0<a_N \ll a_{N-1} \ll \cdots \ll a_1$.

Then $\rho \in \Rep(Q)$ is $\ell$-semistable with respect to the $G_Q$-action if and only if for every $i \in Q_0$ the image of $\rho$ in $\Rep(Q)_i$ is $\ell_i$-semistable for the action of $G_i$. Points in $\Rep(Q)^{\rm ss}(\ell)$ have finite stabilizer groups in $G_Q$. If $Q_{\leq p}$ denotes the full decorated sub-quiver containing vertices $\{1,\ldots,p\}$, then the forgetful map
\[
\Rep(Q_{\leq p})^{G_{\leq p}\rm{-ss}} / G_{\leq p} \to \Rep(Q_{\leq p-1})^{G_{\leq p-1}\rm{-ss}} / G_{\leq p-1},
\]
where semistability is always defined with respect to the restriction of $\ell$ to the relevant subgroup, is the fiber bundle associated to the principal $G_{\leq p-1}$-bundle
\[
\Rep(Q_{\leq p-1})^{G_{\leq p-1}\rm{-ss}} \to \Rep(Q_{\leq p-1})^{G_{\leq p-1}\rm{-ss}}(\ell_{\leq p-1}) / G_{\leq p-1},
\]
whose fiber is the Deligne-Mumford stack $\Rep(Q)_p^{G_{p}\rm{-ss}} / G_p$ with $G_{\leq p-1}$-action coming from the natural $G_{\leq p}$ action on $\Rep(Q)_p$.
\end{prop}

\begin{proof}
The claims of this proposition are formal consequences of the claim that the canonical isomorphism $\Rep(Q_{\leq p}) \cong \Rep(Q_{\leq p-1}) \times \Rep(G)_p$ induces an isomorphism
\[
\Rep(Q_{\leq p})^{G_{\leq p}\rm{-ss}} \cong \Rep(Q_{\leq p-1})^{G_{\leq p-1}\rm{-ss}} \times \Rep(G)_p^{G_p\rm{-ss}}.
\]
To show this, represent $\rho \in \Rep(Q_{\leq p})$ as a pair $(\rho_{\leq p-1},\rho_p)$. If $\rho_p$ is destabilized by a cocharacter $\lambda_p$ of $G_p$, then the cocharacter $(0,\lambda_p)$ destabilizes $\rho$. On the other hand, if $\rho_p$ is semistable, and $\lambda_{\leq p-1}$ is a cocharacter of $G_{\leq p-1}$ that destabilizes $\rho_{\leq p-1}$, then the fact that $\Rep(Q)_p^{G_p\rm{-ss}}/G_p$ is a proper Deligne-Mumford stack implies that there is some cocharacter $\lambda_p$ of $G_p$ such that $(\rho_{\leq p-1},\rho_p)$ has a limit under $(\lambda_{\leq p-1},\lambda_p)$. Assuming $a_1,\ldots,a_{p-1}$ fixed, if we choose $a_p>0$ sufficiently small, then
\[
\langle \lambda_{\leq p-1} , a_1 \ell_1+\cdots+a_{p-1}\ell_{p-1}\rangle + a_p \langle \lambda_p, \ell_p \rangle <0,
\]
hence $(\lambda_{\leq p-1},\lambda_p)$ destabilizes $\rho$ with respect to $a_1 \ell_1 + \cdots + a_p \ell_p$. Because finitely many cocharacters of $G_{\leq p-1}$ suffice to destabilize every unstable point of $\Rep(Q_{\leq p-1})$, this shows that we can choose $a_p >0$ sufficiently small so that if $\rho_{\leq p-1}$ is unstable and $\rho_p$ is semistable, then $\rho$ is unstable. Taken together, this shows that $\rho$ is unstable if either $\rho_p$ or $\rho_{\leq p-1}$ is unstable, and hence
\[
\Rep(Q_{\leq p})^{G_{\leq p}\rm{-ss}} \subset \Rep(Q_{\leq p-1})^{G_{\leq p-1}\rm{-ss}} \times \Rep(Q)_p^{G_p \rm{-ss}}.
\]
Our hypotheses, and the inductive hypothesis, guarantees that the quotient of the right-hand side by $G_{\leq p}$ is a proper Deligne-Mumford stack. But it is a general fact in GIT that the coarse moduli space of the open substack $\Rep(Q_{\leq p})^{\rm ss}/G_{\leq p}$ is proper over $\Spec(\cO_{\Rep(Q_{\leq p})}^{G_{\leq p}}) = \Spec(k)$. It follows that the open immersion above must be equality.
\end{proof}

We now combine the previous results to show how, given as basic building blocks some examples of linear GIT quotients that admit full strong exceptional collections, one can construct many more examples.

\begin{prop} \label{P:decorated_quiver_fec}
Under the hypotheses of \Cref{P:iterated_fibration}, assume that for each $i=1,\ldots,N$ we are given a collection of irreducible $G_i$-representations $U_{i,j}$ for $j=1,\ldots,m_i$ such that $\cO_{V_i^{n_i}} \otimes_k U_{i,1},\ldots,\cO_{V_i^{n_i}} \otimes U_{i,m_i}$ restricts to a full strong exceptional collection in $\DCoh((V_i^{n_i})^{\rm ss}(\ell_i)/G_i)$. Then the locally free sheaves 
\[
\cO_{\Rep(Q)} \otimes (U_{1,j_1} \boxtimes \cdots \boxtimes U_{N,j_N})
\]
associated to the $G_Q$-representations $U_{1,j_1} \boxtimes \cdots \boxtimes U_{N,j_N}$ restrict to a full strong exceptional collection in $\DCoh(\Rep(Q)^{G_Q-\rm{ss}}(\ell)/G_Q)$, using the reverse lexicographic ordering on the indexing tuples $(j_1,\ldots,j_N)$.
\end{prop}
\begin{proof}
The claim is proved inductively for $\Rep(Q_{\leq p})^{\rm ss}(\ell_{\leq p}) / G_{\leq p}$, for $p=1,\ldots,N$. The inductive step is an application of \Cref{L:fec_bundle} to the fiber bundle $\Rep(Q_{\leq p})^{\rm ss}/G_{\leq p} \to \Rep(Q_{\leq p-1})^{\rm ss}/G_{\leq p-1}$ of $\Rep(Q)_i^{\rm ss}(\ell_i)/G_i$ that is constructed in \Cref{P:iterated_fibration}.
\end{proof}

\begin{ex}
When $G_i = \GL_{v_i}$ for some integer $v_i \geq 1$, and $V_i$ is the standard representation, then $\ell$-stability of $\rho \in \Rep(Q)$ corresponds to the condition that each map
\[
\rho_i \oplus \bigoplus_{\substack{\alpha \in Q_1 \\ t(\alpha) = i}} \rho_\alpha :k^{w_i} \oplus \bigoplus_{\substack{\alpha \in Q_1 \\ t(\alpha) = i}} V_{s(\alpha)} \to V_i 
\]
is surjective. Thus $\Rep(Q)^{\rm ss}(\ell)/G_Q$ is a quiver flag variety, as studied in \cite{C11}*{Thm.~2.2}, and \Cref{P:iterated_fibration} recovers the main result \cite{C11}*{Thm.~1.1}.
\end{ex}

\begin{ex}
\Cref{T:main_fano} and \Cref{T:main_nef_fano} provide many examples of $G_i$ and representations $V_i$ such that $(V_i^{n_i})^{\rm ss}/G_i$ admits a full strong exceptional collection of the form $\cO_{V_i^{n_i}} \otimes U_{i,j}$, which greatly increases the scope of \Cref{P:decorated_quiver_fec}.
\end{ex}

\section{Fano GIT quotients by \texorpdfstring{$G$}{G} of rank \texorpdfstring{$2$}{2}} \label{S:Fano_rank2}

The goal of this section is to show that the strong exceptional collection from \Cref{P:exceptional_collections} is full whenever the rank of the group $G$ is $2$, and the GIT parameter is the anticanonical character $\omega^\ast$ with points in the semistable locus having finite $G$-stabilizers. In this case, the resulting GIT quotients are Fano.

\begin{thm}\label{T:main_fano}
Let $G$ be a split reductive group of rank $2$ over $k$ and let $X$ be a linear representation of $G$ satisfying \Cref{H:main_setup}, which we regard as an affine $G$-scheme. Assume that the weights of $X$ span $M_\bR$. Then $G$ acts with finite stabilizers on $X^{\rm ss}(\omega^\ast)$ if and only if no weights of $X$ are proportional to $\omega^\ast$, and in this case, the vector bundles
\[
\left\{ \cO_{X^{\rm ss}(\omega^\ast)} \otimes U : U \in \Rep(G) \text{ irreducible with weights lying in }\theta + \nabla \right\}
\]
split-generate $\DCoh(X^{\rm ss}(\omega^\ast)/ G)$. In particular, if $\theta \in M_\bR^W$ is $\lambda_0$-generic, then the strong exceptional collection from \Cref{P:exceptional_collections} is full.
\end{thm}

We begin with a recollection of certain minimal locally free resolutions of equivariant sheaves.

\subsection{Minimal resolutions of equivariant sheaves}

For an anti-dominant cocharacter $\lambda\in N$, that is $\langle \lambda,\alpha \rangle\geq0$ for all negative roots $\alpha$, the $\lambda$-attracting locus $X^{\lambda \geq 0}$ is a $B$-module and $G \times^B X^{\lambda \geq 0}$ is a vector bundle over $G/B$. Consider the diagram of $G$-equivariant maps
\[
\begin{tikzcd}
G\times^{B}X^{\lambda\geq0} \arrow[r, hook, "j"] \arrow[d]
& G/B\times X \arrow[d, "\pi"] \arrow[r, "p"] & G/B \\ \pi(Z(\lambda)) \arrow[r, hook]
&X  
\end{tikzcd} ,
\]
where $Z(\lambda)$ denotes the image $j (G \times^{B} X^{\lambda\geq0})$, the maps
$\pi$ and $p$ are usual projections, and $p \circ j$ is the bundle map $G \times^{B} X^{\lambda \geq 0} \rightarrow G/B$.

Let $\xi(\lambda)$ be the dual of the sheaf of sections of $G \times^B X^{\lambda < 0}$ over $G/B$. By \cite{W03}*{Prop.~5.1.1}, there is a Koszul resolution of $\cO_{Z(\lambda)}$ as an $\cO_{G/B \times X}$-module
\begin{equation}
\cK (\lambda)_{\bullet} : 0 \rightarrow \textstyle \bigwedge^{r} \left( p^{*} \xi ( \lambda ) \right) \rightarrow \cdots \rightarrow \textstyle \bigwedge^{2} \left( p^{*} \xi (\lambda ) \right) \rightarrow p^{*} \xi (\lambda ) \rightarrow \cO_{G/B \times X} , 
\end{equation}
where $r = \rank \xi (\lambda)$.

\begin{defn}
For any character $\chi \in M$, define $\cL(\chi)$ to be the associated sheaf of sections of the line bundle $G \times^B k \langle \chi \rangle$ over $G/B$, where $k \langle \chi \rangle$ is the one dimensional $B$-representation associated to $\chi$.
\end{defn}

Let $\chi \in M$ be any character. The complex
\[
\cK \left( \lambda, \cL(\chi) \right)_{\bullet} := \cK \left( \lambda \right)_{\bullet}\otimes_{\cO_{G/B\times X}} p^{*} \cL(\chi)
\]
is a locally free resolution of $\cO_{Z (\lambda)} \otimes_{\cO_{G/B \times X}} p^{*} \cL(\chi)$. The results in \cite{W03}*{Thm.~5.1.2,~Thm.~5.4.1} apply to $\cK\left(\lambda,\cL(\chi)\right)_{\bullet}$ yielding the following. 

\begin{prop} \label{P:minimal_complex}
The complex $R \pi_\ast (\cO_{Z(\lambda)} \otimes_{\cO_{G/B\times X}} p^\ast \cL(\chi))$ is quasi isomorphic to a minimal $G$-equivariant complex $C_{\lambda,\chi}$ of free graded $\cO_X$-modules with
\begin{equation} \label{E:C_complex}
\left( C_{\lambda, \chi } \right)_{i} = \bigoplus_{j \in \bZ_{\geq 0} } H^{j}(G/B, \cL(\chi) \otimes \textstyle \bigwedge^{i+j} \xi (\lambda) ) \otimes_k \cO_X .
\end{equation} 
In addition, $C_{\lambda,\chi}$ has homology supported in $G \cdot X^{\lambda \geq 0}$. Since $\lambda$ is anti-dominant, then so is $-w_0\lambda$ and we define the complex 
\begin{equation} \label{E:D_complex}
    D_{\lambda,\chi} := C_{-w_0\lambda,-w_0\chi}.
\end{equation}
\end{prop}

\subsubsection{Filtrations of $C_{\lambda, \chi}$ and $D^\vee_{\lambda, \chi}$}

Recall that the Weyl group $W$ induces a $\ast$-action on the weight lattice defined as follows: given $\mu\in M$ and $w \in W$, 
\[
    w \ast\mu := w\left( \mu+\rho \right) - \rho ,
\]
where $\rho$ is the half sum of positive roots. If the stabilizer of $\mu+\rho$ in $W$ is trivial, then there is a unique $w\in W$ such that $w\ast\mu$ is dominant.

\begin{defn}\label{D:mu_plus}
If $\Stab(\mu+\rho)$ is trivial, define $\mu^+$ to be $w \ast \mu$, where $w$ is the unique element of the Weyl group such that $w \ast \mu$ is dominant. Otherwise, we say $\mu^+$ does not exist. When $G$ is connected and $\mu^+$ exists,
\begin{itemize}
    \item[] $V(\mu^+)$ denotes the irreducible representation with highest weight $\mu^+$.
\end{itemize}
We will use the convention that $V(\mu^+)=0$ if $\mu^+$ does not exist.
\end{defn}

\begin{rec} [The Borel-Weil-Bott Theorem] 
Suppose $G$ is connected and let $\mu\in M$. If $\mu+\rho$ has a non-trivial stabilizer in $W$, then all the cohomology groups $H^i\left(G/B,\mathcal{L}(\mu)\right)$ vanish. Otherwise, there exists a unique non-vanishing cohomology group $H^{\ell(w)}\left(G/B,\cL(\mu)\right)$, where $\ell(w)$ is the length of the unique element $w \in W$ such that $w\ast\mu$ is dominant. In this case, $H^{\ell(w)}\left(G/B,\cL(\mu)\right) \cong V\left( \mu^+ \right)$.
\end{rec}

The lemmas below are proved in \cite{HLS20}*{Prop.~3.8,~Prop.~3.9} but we restate them here for ease of reference.

\begin{lem}\label{L:complex_1}
Assuming $G$ is connected, the complex $C_{\lambda,\chi}$ has one term of the form $\cO_X \otimes V(\chi)$ and the remaining terms have $G$-equivariant filtrations whose associated graded pieces are locally free sheaves of the form $\cO_X\otimes V(\mu^+)$, where $\mu=\chi-\beta_{i_1}-\beta_{i_2}-\cdots-\beta_{i_p}$ with $p > 0$, $i_1,\ldots,i_p$ distinct, and $\langle\lambda,\beta_{i_j} \rangle < 0$ for all $j =1, \ldots, p$.
\end{lem}

\begin{proof}
Note that $\cL (\chi) \otimes \bigwedge^{\bullet} \xi (\lambda)$
admits a $G$-equivariant filtration by locally free sheaves of the form $\cL \left( \chi - \beta_{i_1} - \beta_{i_2}- \ldots - \beta_{i_p} \right)$,
where $i_1,\cdots,i_p$ are distinct and $\left \langle \lambda, \beta_{i_j} \right \rangle <0$. The result then follows by applying Borel-Weil-Bott.
\end{proof}

\begin{lem}\label{L:complex_2}
Assuming $G$ is connected, the complex $D^\vee_{\lambda,\chi}$ has one term of the form $\cO_X \otimes V(\chi)$ and the remaining terms have $G$-equivariant filtrations whose associated graded pieces are locally free sheaves of the form $\cO_X\otimes V(\mu^+)$, where $\mu=\chi+\beta_{i_1}+\beta_{i_2}+\cdots+\beta_{i_p}$ with $p > 0$, $i_1,\ldots, i_p$ distinct, and $\langle \lambda, \beta_{i_j} \rangle > 0$ for all $j=1,\ldots,p$.
\end{lem}

\begin{proof}
Note that $\cL (- w_0 \chi) \otimes \bigwedge^{\bullet} \xi (- w_0 \lambda)$ admits a $G$-equivariant filtration by locally free sheaves of the form $\cL \left( - w_0 \chi - \beta_{i_1} - \beta_{i_2}- \ldots - \beta_{i_p} \right)$,
where $\left \langle -w_0 \lambda, \beta_{i_j} \right \rangle < 0$ and $i_1,\cdots,i_p$ are distinct. The result then follows by applying Borel-Weil-Bott and the isomorphism $V((- w_0 \alpha)^+)^\vee \cong V(\alpha^+)$ (\cite{HLS20}*{Prop.~3.7}), which holds for any weight $\alpha \in M$ for which $\alpha^+$ is defined. 
\end{proof}

\subsubsection{A criterion for $C_{\lambda, \chi}$ and  $D^\vee_{\lambda, \chi}$ to be unstably supported}

The lemma below implies that when $G$ has rank $2$ and $\lambda \in \ell^\perp$, the complexes $C_{\lambda, \chi}$ and $D_{\lambda, \chi}^\vee$ have homology supported in $X^{\rm us} (\ell)$.

\begin{lem}\label{L:unstable_locus}
Let $X$ be a linear representation of a reductive group $G$ of rank $2$ that satisfies \Cref{H:main_setup} and let $\ell \in M_\bR^W$ with $\langle \lambda_0, \ell \rangle < 0$. Then $G$ acts with finite stabilizers on $X^{\rm ss}(\ell)$ if and only if the weights of $X$ span $M_\bR$ and no weight of $X$ is proportional to $\ell$. In this case,
\[
 G \cdot X^{\lambda' \geq 0} \subset X^{\rm us} (\ell)
\]
for any non-zero cocharacter $\lambda' \in N_\bR $ such that $\langle \lambda',\ell \rangle = 0$.
\end{lem}
\begin{proof}
Let $\beta_1, \ldots, \beta_n$ denote the weights of $X$ written with multiplicity. Then a point $x \in X$ has a decomposition
\[
    x = x_1 + \ldots + x_n,
\]
where for $i = 1, \ldots, n$, $x_i$ is in the eigenspace corresponding to the weight $\beta_i$. 

Assume there is a non-zero weight $\beta$ that is proportional to $\ell$. Let $x$ be a non-zero point whose only non-zero component in the decomposition above is in the eigenspace corresponding to $\beta$. We claim that $x$ is semistable and has infinite stabilizer group. Indeed, for any cocharacter $\lambda \in N$, $\lim_{t \rightarrow 0} \lambda(t) \cdot x$ exists if and only if $\langle \lambda, \beta \rangle \geq 0 $. The latter is equivalent to $\langle \lambda, \ell \rangle \geq 0$ because under \Cref{H:main_setup} the assumption that $\langle \lambda_0,\ell \rangle < 0$ guarantees that $\beta$ is positively proportional to $\ell$. Thus $x$ is semistable by the Hilbert-Mumford criterion. Moreover, $x$ is fixed by any non-trivial cocharacter of the subgroup $G_{\ell}^\circ= \ker(\ell)^\circ$, so it is strictly semistable and it also has infinite $G$-stabilizers.

On the other hand, suppose there is a point in $X^{\rm ss}(\ell)$ with positive dimensional stabilizer. If one chooses $x \in X^{\rm ss}(\ell)$ to be a point whose orbit has minimal dimension, and hence whose stabilizer has maximal dimension, then $G \cdot x \subset X^{\rm ss}(\ell)$ is closed. Such a point has a reductive stabilizer group, which is positive dimensional by hypothesis. By replacing $x$ with $g \cdot x$ for some $g \in G$, we may assume that a maximal torus of $\textrm{Stab} (x)$ is contained in $T$, thus $x$ is fixed by a non-trivial cocharacter $\lambda \in N$.

Now $x$ has a decomposition $x = x_{i_1} + \ldots + x_{i_p}$, where for $j = 1, \ldots, p$, $x_{i_j} \neq 0$ and $\langle \lambda, \beta_{i_j} \rangle =0$. Since $x$ is semistable, this cocharacter must necessarily pair non-negatively with $\ell$ by the Hilbert-Mumford criterion. In fact we must have $ \langle \lambda, \ell \rangle = 0$, otherwise $-\lambda$ would be a destabilizing cocharacter. Since $M$ has rank $2$, the weights $\beta_{i_1}, \ldots, \beta_{i_p}$ must lie in the span of $\ell$.

When $X$ has no weights in the span of $\ell$, if one starts with a cocharacter $\lambda'$ with $\langle \lambda',\ell\rangle = 0$, one can perturb $\lambda'$ slightly to obtain a $\lambda \in N$ with $\langle \lambda, \ell \rangle < 0$ and $X^{\lambda' \geq 0} = X^{\lambda \geq 0}$. The latter lies in the unstable locus by the Hilbert-Mumford criterion, thus $G \cdot X^{\lambda' \geq 0} \subset X^{\rm us} (\ell)$.
\end{proof}

\subsection{Proof of \texorpdfstring{\Cref{T:main_fano}}{Theorem 4.1}}

We begin with the following observation. 

\begin{lem} \label{L:reduction_to_connected}
It suffices to prove \Cref{T:main_fano} for connected groups $G$.
\end{lem}

\begin{proof}
Let $G$ be a split reductive group and $G^\circ$ be the connected component containing the identity. Note that $T \subset G^\circ$, so the windows defined in \Cref{S:windows} for $G$ and $G^\circ$ coincide. We claim that $X^{G\mathrm{-ss}} = X^{G^\circ \mathrm{-ss}}$. Under the isomorphism $X/G^\circ \cong (X \times (G/G^\circ))/G$, $X^{G^\circ\mathrm{-ss}}$ corresponds to the $G$-semistable locus of $X \times (G/G^\circ)$. Therefore, if we consider the $G$-equivariant finite surjective \'etale morphism $f : X \times (G/G^\circ) \to X$, which corresponds to the canonical morphism $X/G^\circ \to X/G$, we must show that $(X \times G/G^\circ)^{G\mathrm{-ss}} = f^{-1}(X^{G\mathrm{-ss}})$. This follows from the Hilbert-Mumford criterion for affine $G$-schemes, because the finiteness of $G/G^\circ$ implies that if $x \in X \times (G/G^\circ)$ and $\lambda$ is a cocharacter in $G$, then $\lim_{t\to 0} \lambda(t) \cdot x$ exists if and only if $\lim_{t \to 0} \lambda(t)\cdot f(x)$ exists. 

We continue to use the notation $f$ for the finite surjective \'etale morphism $X^{\rm ss} / G^\circ \cong (X^{\rm ss} \times (G/G^\circ)) / G \to X^{\rm ss}/G$. We claim that for any collection $\{V_\alpha\}_{\alpha\in\cA}$ of vector bundles on $X^{\rm ss}/G^\circ$ that split-generate $\DCoh(X^{\rm ss} /G^\circ)$, the collection $\{f_\ast(V_\alpha)\}_{\alpha \in \cA}$ split-generates $\DCoh(X^{\rm ss}/G)$. Because $\Dqc(X^{\rm ss}/G^\circ)$ is compactly generated by $\DCoh(X^{\rm ss}/G^\circ)$, the fact that $\{V_\alpha\}_{\alpha \in \cA}$ split generates $\DCoh(X^{\rm ss}/G^\circ)$ is equivalent to the fact that for any $\cF \in \Dqc(X^{\rm ss} /G^\circ)$, $\Hom_{X^{\rm ss}/ G^\circ} (V_\alpha,\cF) = 0$ for all $\alpha \in \cA$ implies $\cF = 0$. The same holds for split generation of $\DCoh(X^{\rm ss}/G)$. 

Now let $\cH \in \Dqc (X^{\rm ss}/G)$ be an object satisfying $\Hom_{X^{\rm ss}/G}(f_\ast V_\alpha, \cH) = 0$ for all $\alpha \in \cA$. Then
\begin{align*}
    0 &= \Hom_{X^{\rm ss}/G}(f_\ast V_\alpha, \cH) \\ & \cong \Hom_{X^{\rm ss}/ G^\circ}(V_\alpha, f^! \cH) \\ & \cong \Hom_{X^{\rm ss}/ G^\circ}(V_\alpha, f^\ast \sheafhom (f_\ast \cO_{X^{\rm ss} / G^\circ}, \cH)),
\end{align*}
where the second isomorphism follows from the description of $f^!$ for finite flat morphisms. This implies $f^\ast \sheafhom(f_\ast \cO_{X^{\rm ss} / G^\circ}, \cH) \cong 0$. Because $f$ is surjective and \'{e}tale, we have $\sheafhom(f_\ast \cO_{X^{\rm ss} / G^\circ}, \cH) = 0$. Moreover, $f$ is finite and \'{e}tale, so $f_\ast \cO_{X^{\rm ss} / G^\circ}$ is locally free, and this implies $\cH = 0$. It follows that $\{f_\ast(V_\alpha)\}_{\alpha \in \cA}$ split generates $\DCoh(X^{\rm ss}/G)$.

If $\DCoh(X^{\rm ss}/ G^\circ)$ is split-generated by vector bundles of the form $\cO_{X^{\rm ss}} \otimes U$, where $U$ is an irreducible $G^\circ$-representation whose weights lie in $\theta + \nabla$, then we have seen that the collection $f_\ast (\cO_{X^{\rm ss}} \otimes U)$ will split generate $\DCoh(X^{\rm ss}/ G)$. We compute
\[
    f_\ast \left( \cO_{X^{\rm ss}} \otimes U \right) \cong \cO_{X^{\rm ss}} \otimes U \otimes k[G / G^\circ].  
\]
Because $T$ acts trivially on $k[G/ G^\circ]$, the weights of the $G$-representation $U \otimes k[G / G^\circ]$ lie in $\theta + \nabla$ if and only if the weights of $U$ lie in $\theta + \nabla$.
\end{proof}

\begin{proof}[Proof of \Cref{T:main_fano}]

That $X^{\rm ss}(\omega^\ast)$ has finite $G$-stabilizers if and only if $X$ has no weights proportional to $\omega^\ast$ follows from \Cref{L:unstable_locus} and the fact that \Cref{H:main_setup} implies zero is not a weight of $X$.

Recall that for $\chi$ a dominant character, $V(\chi)$ denotes the irreducible representation of $G$ with highest weight $\chi$. The restriction functor $\DCoh(X/G) \rightarrow \DCoh(X^{\rm ss}/ G)$ is essentially surjective because $X^{\rm ss}/G \subset X/G$ is an open substack. Moreover, $X$ is affine, so $\DCoh(X/G)$ is split-generated by vector bundles of the form $\cO_X \otimes U$, where $U$ is an irreducible $G$-representation.  Thus, because all irreducible representations of a connected $G$ are highest weight representations $V(\chi)$, all vector bundles of the form $\cO_{X^{\rm ss}} \otimes V(\chi)$ with $\chi$ dominant split-generate $\DCoh(X^{\rm ss}/ G)$.

As observed in the proof of \Cref{P:vanishing}, the weights of $V(\chi)$ are contained in the convex hull of $W \cdot \chi$. Because $\theta + \nabla$ is $W$-invariant and convex, the weights of $V(\chi)$ lies in $\theta + \nabla$ if and only if $\chi \in \theta +\nabla$. Therefore by \Cref{L:reduction_to_connected}, it suffices to show that for $G$ connected and for any $\chi \in M^+$, the vector bundle $\cO_X \otimes V(\chi)$ lies in the smallest triangulated subcategory of $\DCoh(X/G)$ that contains i) complexes that are set-theoretically supported on the unstable locus, and ii) the vector bundles $\cO_X \otimes V(\mu)$ for $\mu\in M^{+}\cap(\theta+\nabla)$.

For connected $G$, it will be convenient to introduce the following notation, for any subset $S \subset M_\bR$:
\[
\cC(S) := \left\{ \begin{array}{c} \text{smallest full triangulated subcategory of $\DCoh(X/G)$} \\ \text{ containing all unstably supported complexes and the} \\ \text{locally free sheaves } \cO_X \otimes V(\chi) \text{ for all } \chi \in S \cap M^+ \end{array} \right\}.
\]
Thus we must show $\cC(\theta+\nabla) = \cC(M_\bR)$. We do this in three steps. 
\begin{enumerate}
    \item \textit{Reduction to the $\lambda_0$-strip}: $\cC(M_\bR) = \cC(\theta + B_{\lambda_0})$.
    \item \textit{Reduction to the cylinder window}: $\cC(\theta + B_{\lambda_0}) = \cC(\theta+\overline{\nabla})$.
    \item \textit{Reduction to the barrel window}: $\cC(\theta+\overline{\nabla}) = \cC(\theta+\nabla)$.
\end{enumerate}
The $\lambda_0$-strip, the cylinder window, and the barrel window are as defined in \Cref{S:windows}. Because $G$ has rank $2$ and $\omega^\ast \neq 0$, the subset $(\omega^\ast)^\perp \subset N_\bR$ is one-dimensional. In this case, the barrel window has a more concrete description as follows. 

Let $\lambda'$ be any cocharacter that spans $(\omega^\ast)^\perp$. The rank $2$ hypothesis implies $\eta_{\lambda'} = \eta_{-\lambda'}$. Thus the barrel window $\theta + \nabla$ consists of those $\chi \in M_\bR$ such that $\chi -\theta \in B_{\lambda_0}$ and satisfying either
\begin{itemize}
    \item[(i)] $\lvert \left\langle \lambda',\chi - \theta \right\rangle \rvert < \frac{1}{2} \eta_{\lambda'}$, or
    \item[(ii)] $\left\langle \lambda',\chi - \theta \right\rangle = \frac{1}{2} \eta_{\lambda'}$  and 
    \[ \frac{1}{2} \langle \lambda_0, \det X^{-\lambda' \le 0}  \rangle \le \langle \lambda_0,\chi - \theta \rangle \leq \frac{1}{2} \langle \lambda_0, - \det X^{\lambda' \le 0} \rangle, \text{ or }\]
    \item[(iii)] $\left\langle \lambda',\chi-\theta \right\rangle = - \frac{1}{2} \eta_{\lambda'}$ and  
    \[ \frac{1}{2} \langle \lambda_0, \det X^{\lambda' \le 0}  \rangle \le \langle \lambda_0,\chi -\theta  \rangle \le  \frac{1}{2}  \langle \lambda_0, - \det X^{-\lambda' \le 0} \rangle . \]
\end{itemize}
Note that the terms involving $\mathfrak{g}^{\lambda' < 0}$ do not appear in the inequalities in (ii) and (iii) because $\lambda_0$ is a central cocharacter.

The aforementioned reduction steps will be carried out inductively using the following quantities attached to each $\chi \in M^+$:  
\begin{equation} \label{eq:alpha_chi}
\alpha_\chi  := 2\lvert \langle \lambda_0, \chi-\theta \rangle \rvert / \eta_{\lambda_0} , \textrm{ and }
\end{equation}
\begin{equation}\label{eq:r_chi}
    r_\chi :=  \frac{ \lvert \left\langle \lambda', \chi -\theta \right\rangle \rvert - \left\langle \lambda', \rho \right\rangle  }{\left\langle \lambda',  - \det X^{\lambda' \le 0} \right\rangle}  ,
\end{equation}
where $\rho$ is one half of positive roots and $\lambda' \in (\omega^\ast)^\perp$ is an anti-dominant cocharacter satisfying $\langle \lambda', \chi - \theta \rangle \le  0$. The latter is possible because if the Weyl group is non-trivial, there is a unique anti-dominant $\lambda'$ up to positive scaling that must pair non-negatively with any dominant character, and if the Weyl group is trivial, then both $\lambda'$ and $-\lambda'$ are anti-dominant. We also observe that $r_\chi$ is well-defined and non-negative. Indeed, $\langle \lambda', - \det X^{\lambda' \le 0} \rangle > 0$ because the weights of $X$ span $M_\bR$ and $\langle \lambda', \rho \rangle \le 0 $ because $\lambda'$ is anti-dominant.

\begin{rem}
The quantity $r_\chi$ is a modification of the induction parameter used in the combinatorial generation algorithms in \cites{SVdB17,HLS20,SVdB21} that is suitable for our particular window. 
\end{rem}

\begin{lem}[Reduction to the $\lambda_0$-strip] \label{L:reduction_to_lambda_0}
$\cC(M_\bR) = \cC(\theta + B_{\lambda_0})$.
\end{lem}

\begin{proof}[Proof of \Cref{L:reduction_to_lambda_0}]
\begin{figure}[ht]
\begin{tikzpicture}[scale=.65]

\filldraw[gray!15!white] (-17/2,3/2) -- (-17/2, 17/2) -- (-3/2,17/2) -- (17/2, -3/2) -- (17/2,-17/2) -- (3/2,-17/2) --cycle;
\fill[fill=green,fill opacity=0.06] (-17/2,11/2) -- (-17/2, 17/2) -- (-11/2,17/2) -- (17/2, -11/2) -- (17/2,-17/2) -- (11/2,-17/2) --cycle;

\draw[<->, thin, color=gray] (-8.5,0) -- (8.5,0);
\draw[<->, thin, color=gray] (0,-8.5) -- (0,8.5);
\draw[step=.5, ultra thin, color=gray!50!white] (-8.5,-8.5) grid (8.5,8.5);

\draw[-, thick, green!60!black] (-11/2,17/2) -- (17/2,-11/2);
\draw[-, thick, green!60!black] (11/2,-17/2) -- (-17/2,11/2);
\draw[dotted, very thick, blue] (-17/2,-16/2) -- (16/2,17/2);

\draw[->, thick, color=blue] (-7.9/2,9.1/2) -- (-5.5/2,11.5/2);
\draw[->, thick, color=blue] (-9/2,8/2) -- (-11.5/2,5.5/2);
\node[blue] at (-8.5/2,8.5/2) {$\lambda_0$-strip};

\draw[->, thick, color=blue] (-3.8/2,6.2/2) -- (2/2,12/2);
\draw[->, thick, color=blue] (-5.2/2,4.8/2) -- (-12/2,-2/2);
\node[blue] at (-4.5/2,5.5/2) {$(7/3) B_{\lambda_0}$};

\node[blue] at (13.5/2, 7/2) {$M_\bR^+$};

\fill[fill=gray,fill opacity=0.1] (-17/2,-17/2) -- (-17/2 , 17/2) -- (17/2, 17/2);

\draw (0/2,-14/2) pic[rotate=0,blue] {cross=3pt};
\node[blue] at (-0.5/2, -14.9/2) {$\chi_2$};

\draw (3/2,-14/2) pic[rotate=0,red!90!black] {cross=3pt};
\draw (2/2,-13/2) pic[rotate=0,red!90!black] {cross=3pt};
\draw (1/2,-12/2) pic[rotate=0,red!90!black] {cross=3pt};
\draw (0/2,-11/2) pic[rotate=0,red!90!black] {cross=3pt};
\draw (5/2,-13/2) pic[rotate=0,red!90!black] {cross=3pt};
\draw (4/2,-12/2) pic[rotate=0,red!90!black] {cross=3pt};
\draw (3/2,-11/2) pic[rotate=0,red!90!black] {cross=3pt};
\draw (2/2,-10/2) pic[rotate=0,red!90!black] {cross=3pt};
\draw (1/2,-9/2) pic[rotate=0,red!90!black] {cross=3pt};
\draw (6/2,-11/2) pic[rotate=0,red!90!black] {cross=3pt};
\draw (5/2,-10/2) pic[rotate=0,red!90!black] {cross=3pt};
\draw (4/2,-9/2) pic[rotate=0,red!90!black] {cross=3pt};
\draw (3/2,-8/2) pic[rotate=0,red!90!black] {cross=3pt};
\draw (6/2,-8/2) pic[rotate=0,red!90!black] {cross=3pt};
\fill[fill=red,fill opacity=0.1] (0/2,-14/2) -- (0/2 , -11/2) -- (1/2,-9/2) -- (3/2,-8/2) -- (6/2,-8/2) -- (6/2,-11/2) -- (5/2,-13/2) -- (3/2, -14/2) -- cycle;

\draw (9/2,5/2) pic[rotate=0,blue] {cross=3pt};
\node[blue] at (9.6/2, 5.8/2) {$\chi_1$};

\draw (6/2,5/2) pic[rotate=0,red!90!black] {cross=3pt};
\draw (7/2,4/2) pic[rotate=0,red!90!black] {cross=3pt};
\draw (8/2,3/2) pic[rotate=0,red!90!black] {cross=3pt};
\draw (9/2,2/2) pic[rotate=0,red!90!black] {cross=3pt};
\draw (4/2,4/2) pic[rotate=0,red!90!black] {cross=3pt};
\draw (5/2,3/2) pic[rotate=0,red!90!black] {cross=3pt};
\draw (6/2,2/2) pic[rotate=0,red!90!black] {cross=3pt};
\draw (7/2,1/2) pic[rotate=0,red!90!black] {cross=3pt};
\draw (8/2,0/2) pic[rotate=0,red!90!black] {cross=3pt};
\draw (3/2,2/2) pic[rotate=0,red!90!black] {cross=3pt};
\draw (4/2,1/2) pic[rotate=0,red!90!black] {cross=3pt};
\draw (5/2,-0/2) pic[rotate=0,red!90!black] {cross=3pt};
\draw (6/2,-1/2) pic[rotate=0,red!90!black] {cross=3pt};
\draw (3/2,-1/2) pic[rotate=0,red!90!black] {cross=3pt};
\fill[fill=red,fill opacity=0.1] (9/2,5/2) -- (6/2 , 5/2) -- (4/2,4/2) -- (3/2,2/2) -- (3/2,-1/2) -- (6/2,-1/2) -- (8/2,-0/2) -- (9/2, 2/2) -- cycle;

\draw (-7/2,-7/2) pic[rotate=0,blue] {cross=3pt};
\node[blue] at (-7.1/2, -8.1/2) {$\chi_3$};

\draw (-7/2,-4/2) pic[rotate=0, black] {cross=3pt};
\draw[->,thin, color=blue] (-6.7/2,-3.9/2) .. controls (-5.2/2,-4.2/2) .. (-4.8/2,-5.8/2);
\draw[thick, black] (-6/2,-5/2) circle (4 pt);
\draw (-5/2,-6/2) pic[rotate=0, red!90!black] {cross=3pt};
\draw (-4/2,-7/2) pic[rotate=0, red!90!black] {cross=3pt};
\draw (-6/2,-2/2) pic[rotate=0, black] {cross=3pt};
\draw[->,thin, color=blue] (-5.7/2,-1.9/2) .. controls (-3.2/2,-2.2/2) .. (-2.8/2,-4.8/2);
\draw (-5/2,-3/2) pic[rotate=0, black] {cross=3pt};
\draw[->,thin, color=blue] (-4.7/2,-2.9/2) .. controls (-4/2,-3/2) .. (-3.8/2,-3.8/2);
\draw (-4/2,-4/2) pic[rotate=0, red!90!black] {cross=3pt};
\draw (-3/2,-5/2) pic[rotate=0, red!90!black] {cross=3pt};
\draw (-2/2,-6/2) pic[rotate=0, red!90!black] {cross=3pt};
\draw (-4/2,-1/2) pic[rotate=0, black] {cross=3pt};
\draw[->,thin, color=blue] (-3.7/2,-0.9/2) .. controls (-2.2/2,-1.2/2) .. (-1.8/2,-2.8/2);
\draw[thick, black] (-3/2,-2/2) circle (4 pt);
\draw (-2/2,-3/2) pic[rotate=0, red!90!black] {cross=3pt};
\draw (-1/2,-4/2) pic[rotate=0, red!90!black] {cross=3pt};
\draw (-1/2,-1/2) pic[rotate=0, red!90!black] {cross=3pt};
\fill[fill=red,fill opacity=0.1] (-7/2,-7/2) --  (-1/2 , -1/2) -- (-1/2,-4/2) -- (-2/2,-6/2) -- (-4/2,-7/2) -- cycle;
\fill[fill=red,fill opacity=0.2] (-7/2,-7/2) -- (-5/2 , -6/2) -- (-3/2, -5/2);

\end{tikzpicture}
\caption{\footnotesize This illustrates the reduction to the $\lambda_0$-strip for the $\GL_2$ representation $X = \Sym^3 \bC^2$ with $\lambda_0 = (-1,-1)$ and $\theta = 0$. The roots of $G$ are $(1,-1),(-1,1)$, where we take $(1,-1)$ to be the positive root. The weights of $X$ are $(0,3),(1,2),(2,1),(3,0)$, so $\eta_{\lambda_0} =12$. Here we are using the standard identification $M_\bR \cong N_\bR \cong \bR^2$. The $T$-weights indicated by blue crosses satisfy $\langle \lambda_0, \chi \rangle = \pm \alpha_\chi \eta_{\lambda_0}/2$, where $\alpha_\chi = 7/3$. In all three cases, the red crosses indicate the weights $\mu^+$ appearing in the respective complexes as described in \Cref{L:reduction_to_lambda_0}. For $\chi_1$, the complex is $C_{\lambda_0, \chi_1}$ and all $\mu$ are dominant, so $\mu^+ = \mu$. For $\chi_3$, the complex is $C_{\lambda_0, \chi_3 + \det X}$ and some of the $\mu$ are not dominant. These weights are indicated by the black crosses when $\mu^+$ exists and by the black circles otherwise. The $\rho$-shifted Weyl group action $(-)^+$ is reflection along the dotted blue line as indicated by the arrows. One sees that all the resulting $\mu^+$ lie in a strictly smaller $\lambda_0$-strip.} \label{fig:lambda_0_reduction}
\end{figure}
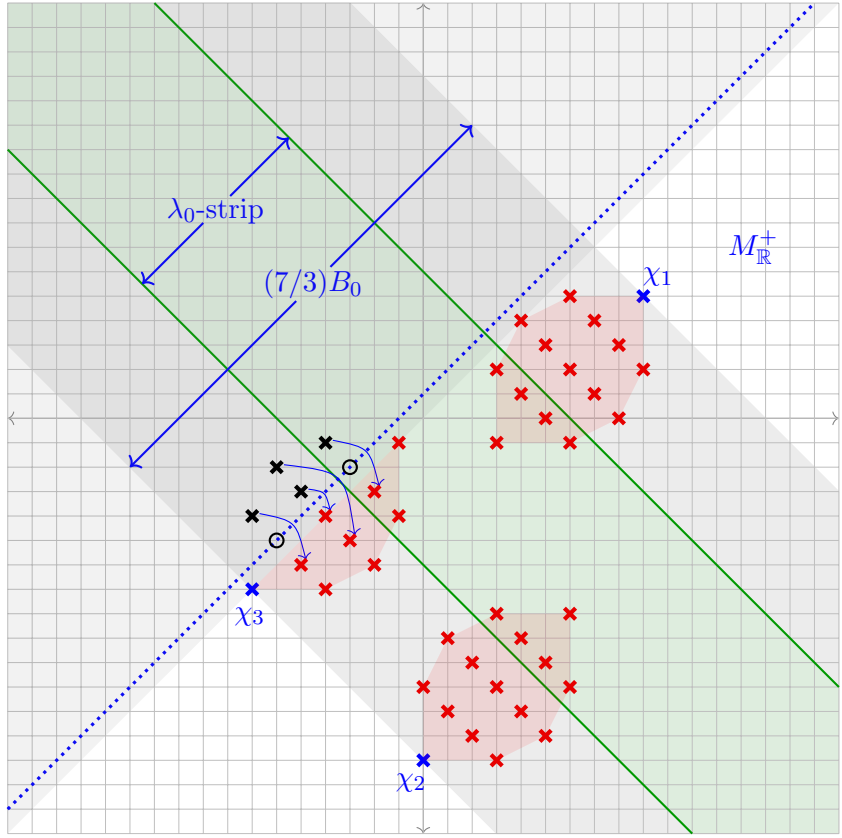

Let $\chi \notin \theta + B_{\lambda_0}$ be a dominant $T$-weight. We wish to show that the tautological vector bundle $\cO_X \otimes V(\chi)$ is in $\cC (\theta + B_{\lambda_0})$. We proceed by induction on the quantity $\alpha_\chi$ in \eqref{eq:alpha_chi}. Note that $\chi\notin\theta+B_{\lambda_0}$ if and only if $\alpha_{\chi}>1$.

If $\left\langle \lambda_{0}, \chi - \theta \right\rangle = \alpha_{\chi}\eta_{\lambda_0}/ 2$, consider the $G$-equivariant complex $C_{\lambda_0, \chi}$. The assumption that $\lambda_0$ pairs strictly negatively with all weights of $X$ (\Cref{H:main_setup}) implies that the subset $X^{\lambda_0\geq0}$ coincides with the stratum of $X^{\rm us}(\omega^\ast)$ consisting only of the origin, so $C_{\lambda_0, \chi}$ is unstably supported. \Cref{L:complex_1} says that  $C_{\lambda_0, \chi}$ has one term of the form $\cO_X \otimes V(\chi)$ and other terms that are direct sums of vector bundles of the form $\cO_X \otimes V(\mu^+)$, where $\mu = \chi -\beta_{i_1}-\cdots-\beta_{i_p}$ with $p>0$ and $i_1, \ldots, i_p$ distinct. Because $\alpha_\chi > 1$, one can check that the weights $\mu^+$ lie in the interior of the strip $\theta + \alpha_\chi B_{\lambda_0}$. 

Otherwise, $\left\langle \lambda_{0}, \chi - \theta \right\rangle = - \alpha_{\chi}\eta_{\lambda_0}/ 2$ and one takes the complex $C_{\lambda_0, \chi + \det X}$. It is an unstably supported complex that, by \Cref{L:complex_1}, relates $\cO_X\otimes V(\chi)$ with sheaves of the form $\cO_X \otimes V(\mu^+)$ for $\mu = \chi + \beta_{i_1} + \cdots + \beta_{i_p}$ with $p>0$ and $i_1, \ldots, i_p$ distinct. As in the previous case, these weights lie in the interior of the strip $\theta + \alpha_\chi B_{\lambda_0}$.

Because the set of possible $\alpha_\chi$ that arise above is discrete, we can apply induction to conclude that $\cO_X \otimes V(\chi) \in \cC(\theta + B_{\lambda_0})$ for any $\chi \in M^+$, hence $\cC(M_\bR) = \cC(\theta + B_{\lambda_0})$. \Cref{fig:lambda_0_reduction} illustrates an example of this reduction. 
\end{proof}

\begin{lem}[Reduction to the cylinder window] \label{L:reduction_to_cylinder}
$\cC(\theta + B_{\lambda_0}) = \cC(\theta+\overline{\nabla})$.
\end{lem}
\begin{proof}[Proof of \Cref{L:reduction_to_cylinder}]
Let $\chi$ be a dominant $T$-weight in $\theta + ( B_{\lambda_0} \setminus \overline{\nabla})$. We would like to show that the tautological vector bundle $\cO_X \otimes V(\chi)$ is in $\cC(\theta + \overline{\nabla})$. We will make use of the quantity $r_\chi$ in \eqref{eq:r_chi}, thus we choose $\lambda' \in (\omega^\ast)^\perp$ to be anti-dominant such that $\langle \lambda', \chi - \theta \rangle \le  0$.

Note that the inequality $r_\chi \leq 1/2$ is equivalent to
\[
    \lvert \langle \lambda', \chi  -\theta \rangle \rvert \le  \langle \lambda', -\det X^{\lambda' \le 0} \rangle / 2  + \langle \lambda', \rho \rangle = \eta_{\lambda'}/ 2 ,
\]
where the equality used the identification $2\rho = \det \mathfrak{g}^{\lambda' < 0} $, which holds because $\lambda' $ is anti-dominant. Thus for $\chi \in \theta+B_{\lambda_0}$, $\chi \in \theta + \overline{\nabla}$ if and only if $r_\chi \leq 1/2$.

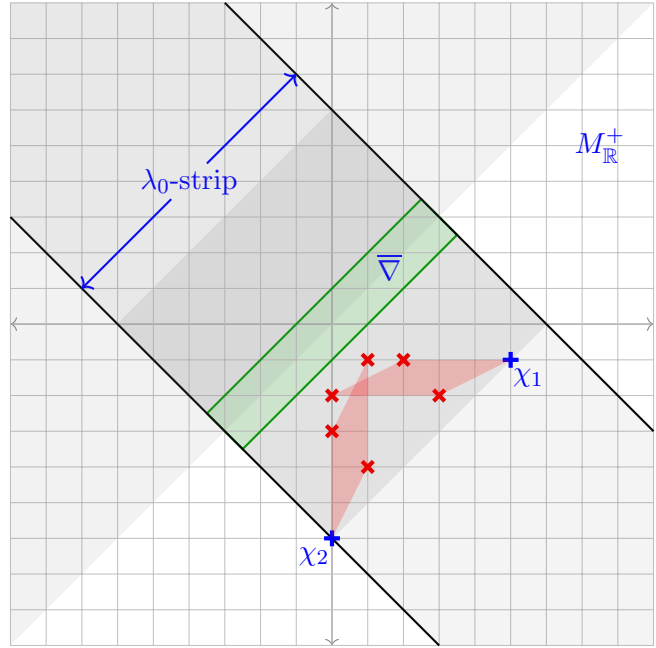
\begin{figure}[ht]
\begin{tikzpicture}[scale=.95]

\filldraw[gray!15!white] (-6/2,-0/2) --  (0/2,6/2) -- (6/2,0/2) -- (-0/2,-6/2) -- cycle;
\fill[fill=green,fill opacity=0.1]  (-3.5/2,-2.5/2) --  (2.5/2,3.5/2) -- (3.5/2,2.5/2) -- (-2.5/2,-3.5/2) -- cycle;

\draw[<->, thin, color=gray] (-4.5,0) -- (4.5,0);
\draw[<->, thin, color=gray] (0,-4.5) -- (0,4.5);
\draw[step=.5, ultra thin, color=gray!50!white] (-4.5,-4.5) grid (4.5,4.5);

\draw[-, thick, green!60!black] (-3.5/2,-2.5/2) --  (2.5/2,3.5/2) -- (3.5/2,2.5/2) -- (-2.5/2,-3.5/2) -- cycle;
\node[blue] at (1.6/2,1.6/2) {$\overline{\nabla}$};
\draw[-, thick, black] (-9/2,3/2) --  (3/2,-9/2);
\draw[-, thick, black] (-3/2,9/2) --  (9/2,-3/2);

\fill[fill=gray,fill opacity=0.09] (-9/2,3/2) -- (-9/2, 9/2) -- (-3/2,9/2) -- (9/2, -3/2) -- (9/2, -9/2) -- (3/2,-9/2) -- cycle;
\draw[->, thick, color=blue] (-3.5/2,4.5/2) -- (-1/2,7/2);
\draw[->, thick, color=blue] (-4.5/2,3.5/2) -- (-7/2,1/2);
\node[blue] at (-4.0/2,4.0/2) {$\lambda_0$-strip};

\node[blue] at (7.5/2, 5/2) {$M_\bR^+$};

\fill[fill=gray,fill opacity=0.1] (-9/2,9/2) -- (9/2 , 9/2) -- (-9/2, -9/2);

\draw (5/2,-1/2) pic[rotate=45,blue] {cross=3pt};
\node[blue] at (5.5/2,-1.5/2) {$\chi_1$};
\draw (2/2,-1/2) pic[rotate=0,red!90!black] {cross=3pt};
\draw (3/2,-2/2) pic[rotate=0,red!90!black] {cross=3pt};
\draw (0/2,-2/2) pic[rotate=0,red!90!black] {cross=3pt};
\fill[fill=red,fill opacity=0.2] (5/2,-1/2) --  (2/2, -1/2) -- (0/2,-2/2) -- (3/2,-2/2) -- cycle;

\draw (0/2,-6/2) pic[rotate=45,blue] {cross=3pt};
\node[blue] at (-0.5/2,-6.5/2) {$\chi_2$};
\draw (-0/2,-3/2) pic[rotate=0,red!90!black] {cross=3pt};
\draw (1/2,-4/2) pic[rotate=0,red!90!black] {cross=3pt};
\draw (1/2,-1/2) pic[rotate=0,red!90!black] {cross=3pt};
\fill[fill=red,fill opacity=0.2] (0/2,-6/2) --  (0/2, -3/2) -- (1/2,-1/2) -- (1/2,-4/2) -- cycle;
\end{tikzpicture}
\caption{\footnotesize  This diagram illustrates the reduction to the cylinder window $\overline{\nabla}$ for the $\GL_2$ representation $X = \Sym^3 \bC^2$ with $\lambda_0 = (-1,-1)$, $\lambda' = (-1, 1)$, and $\theta = 0$. As in \Cref{fig:lambda_0_reduction}, $\eta_{\lambda_0} = 12$ and $\eta_{\lambda'} = 2$. In this example, $Q_{\lambda'} = 0$. The character $\chi_1$ falls into Case 2a and one selects the complex $C_{\lambda', \chi_1}$. On the other hand, $\chi_2$ falls into Case 2b and one selects the complex $D^\vee_{\lambda', \chi_2}$. In both cases, the resulting weights $\mu^+$ are indicated by the red crosses.} \label{fig:cylinder_reduction}
\end{figure}

The complexes $D_{\lambda', \chi}^\vee$ or $C_{\lambda', \chi}$ from \Cref{P:minimal_complex} are both unstably supported by \Cref{L:unstable_locus}. Thus the key idea is to choose either $D_{\lambda', \chi}^\vee$ or $C_{\lambda', \chi}$ in such a way that the corresponding weights $\mu^+$ from \Cref{L:complex_1} and \Cref{L:complex_2} still lie in the $\lambda_0$-strip and have $r_{\mu^+} < r_\chi$.

The number
\begin{equation}\label{eq:Q_lambda}
    Q_{\lambda'} := \frac{1}{2} \eta_{\lambda_0} + \langle \lambda_0, \det X^{\lambda'<0} \rangle .
\end{equation} 
will allow us to decide the appropriate complex. Note that $Q_{\lambda'}=0$ when the Weyl group is non-trivial. There are two cases to consider.

\medskip
\noindent \textit{2a. The case $- \eta_{\lambda_0}/ 2 \le \langle \lambda_0, \chi -\theta \rangle \le Q_{\lambda'}$}:
\medskip

Take the complex $C_{\lambda', \chi}$. \Cref{L:complex_1} says that $C_{\lambda', \chi}$ has one term of the form $\cO_X \otimes V(\chi)$ and other terms have a filtration whose associated graded pieces are of the form $\cO_X \otimes V(\mu^+)$, where $\mu = \chi - \beta_{i_1}-\beta_{i_2}-\cdots-\beta_{i_p}$ with $p > 0$, $i_1,\ldots,i_p$ distinct, and $\langle\lambda',\beta_{i_j} \rangle < 0$ for all $j =1, \ldots, p$. 

\begin{claim}\label{Claim:r_mu_less_1}
For $\mu$ as in case 2a, and assuming $r_\chi \geq 1/2$, if $\mu^+$ exists it satisfies the following. 
\begin{enumerate}
    \item $\mu^+$ is in the $\lambda_0$-strip $\theta + B_{\lambda_0}$, and
    \item $r_{\mu^+}  \le r_{\chi}$ with strict inequality if $r_\chi > 1/2$. 
\end{enumerate}
\end{claim}

\medskip
\noindent \textit{2b. The case $Q_{\lambda'} < \langle \lambda_0, \chi -\theta \rangle \le  \eta_{\lambda_0}/ 2$}:
\medskip

Take the complex $D_{\lambda', \chi}^\vee$. \Cref{L:complex_2} says that $D_{\lambda', \chi}^\vee$ has one term of the form $\cO_X \otimes V(\chi)$ and other terms that have a filtration whose associated graded pieces are of the form $\cO_X \otimes V(\mu^+)$, where $\mu = \chi + \beta_{i_1}+\beta_{i_2}+\cdots+\beta_{i_p}$ with $p > 0$, $i_1,\ldots,i_p$ distinct, and $\langle\lambda',\beta_{i_j} \rangle > 0$ for all $j =1, \ldots, p$.  

\begin{claim}\label{Claim:r_mu_less_2}
For $\mu$ as in case 2b, and assuming $r_\chi \geq 1/2$, if $\mu^+$ exists it satisfies the following.
\begin{enumerate}
    \item $\mu^+$ is in the $\lambda_0$-strip $\theta + B_{\lambda_0}$, and
    \item $r_{\mu^+}  \le r_{\chi}$ with strict inequality if $r_\chi > 1/2$. 
\end{enumerate}
\end{claim}

With these choices of complexes, \Cref{Claim:r_mu_less_1} and \Cref{Claim:r_mu_less_2} together with the fact that the set of $r_\chi$ that arise is discrete imply by induction that $\cO_X \otimes V(\chi) \in \cC(\theta + \overline{\nabla})$ for any $\chi \in M^+ \cap (\theta + B_{\lambda_0})$. Hence $\cC(\theta + B_{\lambda_0}) = \cC(\theta + \overline{\nabla})$. \Cref{fig:cylinder_reduction} illustrates an example of this reduction.

We now verify \Cref{Claim:r_mu_less_1} and \Cref{Claim:r_mu_less_2}.

\begin{proof}[Proof of \Cref{Claim:r_mu_less_1}]
For the first claim, note that $\langle \lambda_0, \mu^+ \rangle = \langle \lambda_0, \mu \rangle$ because $\lambda_0$ is central, thus it suffices to show $\mu$ is in the $\lambda_0$-strip. From the description of $\mu$, we have the inequalities
\begin{equation} \label{eq:lambda_0_bounds}
     \langle \lambda_0, \chi - \theta \rangle < \langle \lambda_0, \mu - \theta \rangle \le \langle \lambda_0, \chi - \theta \rangle + \langle \lambda_0, - \det X^{\lambda' < 0} \rangle, 
\end{equation}
which together with the assumed bounds $- \eta_{\lambda_0}/ 2 \le \langle \lambda_0, \chi - \theta \rangle \le Q_{\lambda'}$ imply that $\mu \in \theta + B_{\lambda_0}$. The strict inequality on the left holds because in the description of $\mu$ above, $p > 0$ and all weights of $X$ are assumed to pair strictly negatively with $\lambda_0$.

For the second claim, we observe from the description of $\mu$ that
\[
\langle \lambda', \chi-\theta \rangle \le \langle \lambda', \mu-\theta \rangle \leq \langle \lambda', \chi-\theta \rangle - \langle \lambda', \det(X^{\lambda'<0}) \rangle.
\]
Using the fact that $\langle \lambda', \chi - \theta \rangle \le 0$, we have $\langle \lambda', \chi -\theta \rangle = - r_\chi \langle \lambda', -\det X^{\lambda' \le 0}\rangle  - \langle \lambda', \rho \rangle$, so we can rewrite the bounds on $\langle \lambda', \mu - \theta \rangle$ as
\begin{align}
    - r_\chi \langle \lambda', -\det X^{\lambda' \le 0}\rangle  - \langle \lambda', \rho \rangle & < \langle \lambda', \mu - \theta \rangle \nonumber\\ 
    & \le \left(1 - r_\chi\right) \langle \lambda', - \det X^{\lambda' \le 0} \rangle  - \langle \lambda', \rho \rangle \vspace{.5em} \nonumber \\ 
    & \le  r_\chi \langle \lambda', - \det X^{\lambda' \le 0} \rangle - \langle \lambda', \rho \rangle \label{eq:mu_bound} 
\end{align}
For the last inequality, we have used the assumption $r_\chi \ge 1/2$, and this inequality is strict if $r_\chi > 1/2$.

If $\rho = 0$, then $\mu^+ = \mu$ and it is immediate from \eqref{eq:mu_bound} that $r_{\mu^+} \le r_\chi$ with strict inequality if $r_\chi > 1/2$. 

If $\rho \neq 0$, then there is a single non-trivial element of the Weyl group $w_0$, namely the reflection along $\Span (\omega^\ast)$. Thus $\theta$ must be in the span of $\omega^\ast$. In this case, if $\mu^+ = \mu$, then $\mu$ is dominant and because $\lambda'$ is anti-dominant, we must have $\langle \lambda', \mu - \theta \rangle \le 0$. Applying this to \eqref{eq:mu_bound}, we have 
\[
    \lvert \langle \lambda', \mu - \theta \rangle \rvert  <  r_\chi \langle \lambda', - \det X^{\lambda' \le 0} \rangle  + \langle \lambda', \rho \rangle ,
\]
and it implies $r_{\mu^+} < r_\chi$. Otherwise, $\mu^+ = w_0 (\mu + \rho) - \rho$, and
\[
    \langle \lambda', \mu^+ \rangle = \langle \lambda', - \mu - 2\rho \rangle .
\]
Applying this to \eqref{eq:mu_bound}, we have
\begin{equation} 
    - r_\chi \langle \lambda', -\det X^{\lambda' \le 0} \rangle - \langle \lambda', \rho \rangle \le \langle \lambda', \mu^+ + \theta \rangle < r_\chi \langle \lambda', -\det X^{\lambda' \le 0} \rangle - \langle \lambda', \rho \rangle \label{eq:bounds_for_mu_plus_theta}
\end{equation}
with strict inequality on the left if $r_\chi > 1/2$. Since $\theta \in \Span (\omega^\ast)$, the pairing $\langle \lambda', \theta \rangle$ is zero, and since $\lambda'$ is anti-dominant $\langle \lambda', \rho \rangle < 0$. Combining these with $\langle \lambda', \mu^+ \rangle = \langle \lambda', - \mu - 2\rho \rangle$ and the triangle inequality, we obtain the middle inequality in the following 
\[
    r_{\mu^+} := \frac{\lvert \langle \lambda', \mu^+ - \theta \rangle \rvert - \langle \lambda', \rho \rangle}{\langle \lambda', - \det X^{\lambda' \le 0} \rangle }  \leq  \frac{\lvert \langle \lambda', \mu^+ + \theta \rangle + \langle \lambda', \rho \rangle \rvert}{\langle \lambda', - \det X^{\lambda' \le 0} \rangle } \le r_\chi.
\]
The rightmost inequality then follows from \eqref{eq:bounds_for_mu_plus_theta}, thus it is strict if $r_\chi > 1/2$.
\end{proof}

\begin{proof}[Proof of \Cref{Claim:r_mu_less_2}]
For the first claim, it suffices to show $\mu$ is in the $\lambda_0$-strip. The key observation here is
\begin{equation}\label{eq:Q_lambda_lower_bound}
    - \frac{1}{2} \eta_{\lambda_0} - \langle \lambda_0, \det X^{\lambda'>0} \rangle \le Q_{\lambda'} ,
\end{equation}
which follows from $-\eta_{\lambda_0} = \langle \lambda_0, \det X \rangle \le \langle \lambda_0, \det X^{\lambda' > 0} + \det X^{\lambda' < 0} \rangle .$ In fact, the inequality above is an equality by the assumption that there are no weights of $X$ in the span of $\omega^\ast$. Now from the description of $\mu$, we have
\[
    \langle \lambda_0, \chi- \theta \rangle + \langle \lambda_0, \det X^{\lambda' > 0} \rangle \le \langle \lambda_0, \mu - \theta \rangle < \langle \lambda_0, \chi -\theta \rangle,
\]
which together with the assumed bounds for $\langle \lambda_0, \chi - \theta \rangle$ and \eqref{eq:Q_lambda_lower_bound} imply that $\mu \in \theta + B_{\lambda_0}$.

For the second claim, we observe that $ \langle \lambda', \det X^{\lambda' \ge 0} \rangle = \langle \lambda', - \det X^{\lambda' \le 0} \rangle$. With this identification, one obtains the inequalities \eqref{eq:mu_bound} for these $\mu$ and the rest of the proof follows exactly as that of \Cref{Claim:r_mu_less_1}.
\end{proof}
\end{proof}

\begin{lem}[Reduction to the barrel window] \label{L:reduction_to_barrel}
$\cC(\theta+\overline{\nabla}) = \cC(\theta+\nabla)$.
\end{lem}

\begin{proof}[Proof of \Cref{L:reduction_to_barrel}]

Recall the quantity $\alpha_\chi$ in \eqref{eq:alpha_chi} and let us choose an indexing of the finite set $M^+ \cap (\theta + (\overline{\nabla} \setminus \nabla)) = \{\chi_1,\ldots,\chi_m\}$ such that $\alpha_{\chi_i} \geq \alpha_{\chi_j}$ for $i<j$. As in previous reductions, we will produce for each $\chi_i$ an unstably supported complex that relates $\cO_X \otimes V(\chi_i)$ with sheaves that admit a $G$-equivariant filtration with associated graded pieces of the form $\cO_X \otimes V(\mu^+)$ with either $r_{\mu^+} < r_{\chi_i}$ (in which case $\mu^+ \in \theta + \nabla$) or $r_{\mu^+} = 1/2$ and $\alpha_{\mu^+} < \alpha_\chi$ (in which case $\mu^+ = \chi_j$ for some $j>i$). It will then follow by induction that for any $n$,
\[
\cC(\theta+\overline{\nabla}) = \cC((\theta+\nabla) \cup \{\chi_{n+1},\ldots,\chi_m\}).
\]
The case $m=n$ would give $\cC(\theta+\nabla) = \cC(\theta+\overline{\nabla})$.

Let $\chi \in \{ \chi_1, \ldots, \chi_m \}$. By definition of the barrel window, $\lvert \langle \lambda', \chi - \theta \rangle \rvert = \eta_{\lambda'}/2$, thus $r_\chi = 1/2$, where we again take $\lambda' \in (\omega^\ast)^\perp$ to be an anti-dominant cocharacter such that $\langle \lambda', \chi - \theta \rangle \le 0$. By the rank-two description of the barrel window, one of the following cases (3a) or (3b) must hold.

\begin{figure}[ht]
\begin{tikzpicture}[scale=1]

\fill[fill=gray,fill opacity=0.2]  (-3.5,-2.5) --  (2.5,3.5) -- (3.5,2.5) -- (-2.5,-3.5) -- cycle;
\fill[fill=green,fill opacity=0.1]   (-3.3,-2.7) -- (-2,-1) -- (1,2) -- (2.7,3.3) -- (3.3,2.7) -- (2,1) -- (-1,-2) -- (-2.7,-3.3) -- cycle;

\fill[fill=gray,fill opacity=0.1] (-4,-4) -- (-4 , 4) -- (4, 4);

\draw[<->, thin, color=gray] (-4,0) -- (4,0);
\draw[<->, thin, color=gray] (0,-4) -- (0,4);
\draw[step=1, ultra thin, color=gray!50!white] (-4,-4) grid (4,4);
\draw[-, ultra thin, gray] (-3.5,-2.5) --  (2.5,3.5) -- (3.5,2.5) -- (-2.5,-3.5) -- cycle;
\draw[-, thick, green!60!black] (-3.3,-2.7) -- (-2,-1) -- (1,2) -- (2.7,3.3) -- (3.3,2.7) -- (2,1) -- (-1,-2) -- (-2.7,-3.3) -- cycle;
\draw[dotted, very thick, blue] (-4,-3) -- (3,4);

\node[blue] at (2,1) {$\diamond$};
\node[blue] at (1,2) {$\diamond$};
\node[blue] at (-2,-1) {$\diamond$};
\node[blue] at (-1,-2) {$\diamond$};
\node[blue] at (3.4,-1.6) {$M_\bR^+$};

\draw[-, thick, red!90!black] (3,2) -- (1,1);
\draw[-, thick, red!90!black] (0,-1) -- (1,1);
\draw (3,2) pic[rotate=45,blue] {cross=3.5pt};
\node[blue] at (3.3,1.8) {$\chi$};
\draw (1,1) pic[rotate=0,red!90!black] {cross=3.5pt};
\draw (-2,1) pic[rotate=0,black] {cross=3.5pt};
\draw[->,thin, color=blue] (-1.9,1.0) .. controls (-0.5, 0.5) .. (0.0,-0.9);
\draw (0, -1) pic[rotate=0,red!90!black] {cross=3.5pt};
\draw (0,2) pic[rotate=0,black] {cross=3.5pt};
\draw[->,thin, color=blue] (0.1,2.0) .. controls (0.8, 1.8) .. (1,1.1);
\fill[fill=black,fill opacity=0.23] (3,2) -- (1,1) -- (-2,1) -- (0,2) -- cycle;

\end{tikzpicture}
\caption{\footnotesize The diagram illustrates the reduction from the cylinder window to the barrel window for our running $\GL_2$ representation $X = \Sym^3 \bC^2$ from \Cref{fig:lambda_0_reduction} and \Cref{fig:cylinder_reduction}. Here $\lambda' = (-1,1)$ and $Q_{\lambda'} = 0$. We have indicated by the blue diamonds the special weights $\pm \zeta_{\lambda'}/2 = \pm (-2,-1)$ and $\pm \zeta_{-\lambda'}/2 = \pm (-1,-2)$ that are relevant in the passage from the cylinder to the barrel window. The character $\chi$ lies in $M^+ \cap (\overline{\nabla} \setminus \nabla)$ and it falls into Case 3a. Thus, one takes the complex $C_{\lambda', \chi}$. The weights $\mu$ appearing as in \Cref{L:complex_1} are not all dominant and we have indicated these by the black crosses. As in \Cref{fig:lambda_0_reduction}, the $\rho$-shifted action $(-)^+$ for these weight is the reflection along the dotted line as shown by the blue arrows. The resulting $\mu^+$ are indicated by the red crosses. }
\end{figure}
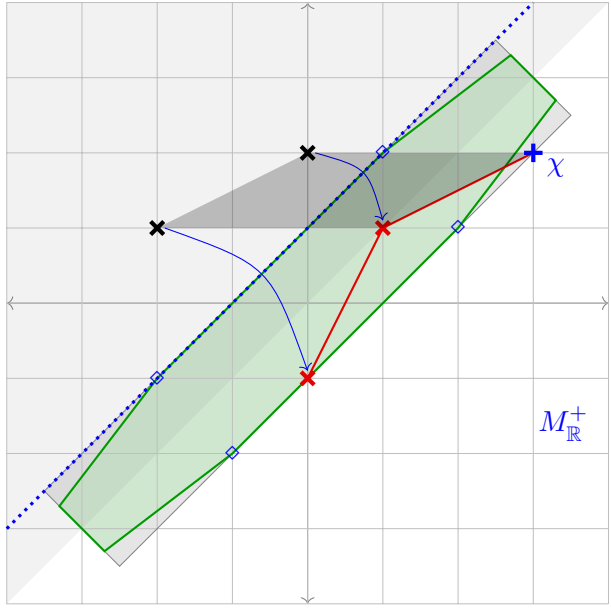

\medskip
\noindent \textit{3a. The case $- \eta_{\lambda_0}/ 2 \le \langle \lambda_0, \chi -\theta \rangle < \langle \lambda_0,  \det X^{\lambda' \leq 0} \rangle /2 $}:
\medskip

Take the complex $C_{\lambda', \chi}$. It is unstably supported by \Cref{L:unstable_locus}. Moreover, \Cref{L:complex_1} says that $C_{\lambda', \chi}$ has one term of the form $\cO_X \otimes V(\chi)$ and other terms that have a filtration whose associated graded pieces are of the form $\cO_X \otimes V(\mu^+)$, where $\mu = \chi - \beta_{i_1}-\beta_{i_2}-\cdots-\beta_{i_p}$ with $p > 0$, $i_1,\ldots,i_p$ distinct, and $\langle\lambda',\beta_{i_j} \rangle < 0$ for all $j =1, \ldots, p$.

Using the fact that $\eta_{\lambda_0} = \langle \lambda_0, - \det X \rangle$ and the definition of $Q_{\lambda'}$ from \eqref{eq:Q_lambda}, we observe that $Q_{\lambda'} > \langle \lambda_0, \det X^{\lambda' \leq 0} \rangle /2$, so that we are in the case (2a) of \Cref{L:reduction_to_cylinder} above. We may then apply \Cref{Claim:r_mu_less_1} to show that $\mu^+$ is still in the $\lambda_0$-strip and $r_{\mu^+} \le 1/2$. Moreover, we can use \eqref{eq:lambda_0_bounds} and the assumed bound $\langle \lambda_0, \chi - \theta \rangle < \langle \lambda_0,  \det X^{\lambda' \leq 0} \rangle /2$ to obtain
\begin{align*}
    \langle \lambda_0, \chi -\theta \rangle < \langle \lambda_0, \mu -\theta \rangle & <   \langle \lambda_0, \det X^{\lambda' \le 0} \rangle / 2 + \langle \lambda_0, -\det X^{\lambda' < 0} \rangle \\
    & = \langle \lambda_0, - \det X^{\lambda' \leq 0} /2 \rangle < -\langle \lambda_0, \chi-\theta \rangle,
\end{align*}
where we have used the assumption that $X^{\lambda'=0}=0$ for the equality. Thus, $\alpha_{\mu^{+}} = \alpha_{\mu} < \alpha_\chi$.

\medskip
\noindent \textit{3b. The case $\langle \lambda_0, - \det X^{\lambda' \geq 0} \rangle / 2 <  \langle \lambda_0, \chi -\theta \rangle \le \eta_{\lambda_0}/ 2$}:
\medskip

Take the complex $D_{\lambda', \chi}^\vee$. It is unstably supported by \Cref{L:unstable_locus}. Moreover, \Cref{L:complex_2} says that $D_{\lambda', \chi}^\vee$ has one term of the form $\cO_X \otimes V(\chi)$ and other terms that have a filtration whose associated graded pieces are of the form $\cO_X \otimes V(\mu^+)$, where $\mu = \chi + \beta_{i_1}+\beta_{i_2}+\cdots+\beta_{i_p}$ with $p > 0$, $i_1,\ldots,i_p$ distinct, and $\langle\lambda',\beta_{i_j} \rangle > 0$ for all $j =1, \ldots, p$.  

Using the fact that $\eta_{\lambda_0} = \langle \lambda_0, - \det X \rangle$ and the definition of $Q_{\lambda'}$, we observe that $Q_{\lambda'} < \langle \lambda_0, - \det X^{\lambda' \geq 0} \rangle / 2$, so that we are in the case (2b) of \Cref{L:reduction_to_cylinder} above. We may then apply \Cref{Claim:r_mu_less_2} to show that $\mu^+$ is still in the $\lambda_0$-strip and $r_{\mu^+} \le 1/2$. Moreover, we can use \eqref{eq:lambda_0_bounds}, the assumed bound $\langle \lambda_0, - \det X^{\lambda' \geq 0} \rangle / 2 < \langle \lambda_0, \chi - \theta \rangle$, and the hypothesis $X^{\lambda'=0}=0$ to obtain
\begin{align*}
    \langle \lambda_0, \chi -\theta \rangle > \langle \lambda_0, \mu -\theta \rangle & > \langle \lambda_0, - \det X^{\lambda' \ge 0} \rangle / 2 + \langle \lambda_0, \det X^{\lambda' > 0} \rangle \\
    & = \langle \lambda_0, \det X^{\lambda' \ge 0} / 2 \rangle > - \langle \lambda_0, \chi-\theta \rangle,
\end{align*}
as in the case (3a). Thus, $\alpha_{\mu^{+}} = \alpha_\mu < \alpha_\chi$.
\end{proof}
\Cref{L:reduction_to_lambda_0}, \Cref{L:reduction_to_cylinder}, and \Cref{L:reduction_to_barrel} imply $\cC(\theta + \nabla) = \cC(M_\bR)$, and \Cref{T:main_fano} follows.

\end{proof}

\subsection{Example: representations of \texorpdfstring{$\GL_2$}{GL2}}\label{S:GL_2}

Recall that any finite dimensional representation $X$ of $\GL_{2}$ has the form
\begin{equation}\label{eq:GL_2_rep_decomposition}
    X \cong \bigoplus_{i=1}^{P} \Sym^{n_{i}}(k^{2})\otimes_k \textrm{det}^{m_{i}},
\end{equation}
where $n_{i}$ and $m_{i}$ are integers and $n_{i}\ge 0$. $N_{\bR}^{W}$ is one-dimensional, so there are, up to scaling by positive real numbers, only two choices for a central cocharacter $\lambda_0$ to which \Cref{H:main_setup} applies. Without any loss of generality, let us assume $\lambda_0=(-1,-1)$. Then $\lambda_0$ pairs strictly negatively with the weights of $X$ if and only if $m_i + n_i/2 > 0$ for all $i$. For representations of this kind, $\omega^\ast$ is a positive multiple of the determinant character. $X^{\rm ss}(\omega^\ast)$ has finite stabilizers if and only if $n_i$ is odd for all $i$, and in this case the unstable locus is $\GL_2 \cdot X^{\lambda' \geq 0}$ for $\lambda'=(-1,1)$. Thus \Cref{T:main_fano} gives the following:

\begin{cor} \label{ex:GL2_reps_example}
Let $X$ be a representation of $\GL_2$ whose corresponding decomposition \eqref{eq:GL_2_rep_decomposition} satisfies, for all $i = 1, \ldots, P$, the conditions 
\begin{enumerate}
    \item $m_i+n_i/2 > 0$, and 
    \item $n_i$ is odd.
\end{enumerate}
Then $X^{\rm ss}(\omega^\ast)$ is non-empty and has finite stabilizers, and $\DCoh(X^{\rm ss}(\omega^\ast)/\GL_2)$ has a full strong exceptional collection consisting of vector bundles.
\end{cor}

\begin{figure}[ht]
\begin{tabular}{ll}
\begin{tikzpicture}[scale=.60]

\filldraw[black!9!white] (-7/2,-3/2) --  (3/2,7/2) -- (7/2,3/2) -- (-3/2,-7/2) -- cycle;
\draw[-, thick, black] (1/2,9/2) -- (9/2,1/2);
\draw[-, thick, black] (-1/2,-9/2) -- (-9/2,-1/2);

\draw[<->, thin, color=gray] (-4.6,0) -- (4.6,0);
\draw[<->, thin, color=gray] (0,-4.6) -- (0,4.6);
\draw[step=.5, ultra thin, color=gray!50!white] (-4.6,-4.6) grid (4.6,4.6);

\draw[-, ultra thick, green!60!black] (-7/2,-3/2) --  (3/2,7/2);
\draw[-, ultra thick, green!60!black] (7/2,3/2) -- (-3/2,-7/2);

\fill[fill=gray,fill opacity=0.1] (-9/2,-1/2) -- (-9/2, 9/2) -- (1/2,9/2) -- (9/2, 1/2) -- (9/2,-9/2) -- (-1/2,-9/2) --cycle;

\draw[->, thick, color=blue] (3/2,-4/2) -- (-1.5/2,-8.5/2);
\draw[->, thick, color=blue] (4/2,-3/2) -- (8.5/2,1.5/2);
\node[blue] at (3.5/2,-3.5/2) {$\lambda_0$-strip};

\foreach \i in { (-4/2,-0/2),(-3/2,1/2),(-2/2,2/2),(-1/2,3/2),(0/2,4/2),(-6/2,-4/2),(-5/2,-3/2),(-4/2,-2/2),(-3/2,-1/2),(-2/2,0/2),(-1/2,1/2),(0/2,2/2),(1/2,3/2),(2/2,4/2),(3/2,5/2), (-5/2,-5/2),(-4/2,-4/2),(-3/2,-3/2),(-2/2,-2/2),(-1/2,-1/2),(-0/2,0/2),(1/2,1/2),(2/2,2/2),(3/2,3/2),(4/2,4/2), (-4/2,-6/2),(-3/2,-5/2),(-2/2,-4/2),(-1/2,-3/2),(-0/2,-2/2),(1/2,-1/2),(2/2,0/2),(3/2,1/2),(4/2,2/2),(5/2,3/2), (-5/2,-4/2),(-4/2,-3/2),(-3/2,-2/2),(-2/2,-1/2),(-1/2,0/2),(-6/2,-3/2),(-5/2,-2/2),(-4/2,-1/2),(-3/2,-0/2), (-2/2,1/2), (-4/2,-5/2),(-3/2,-4/2),(-2/2,-3/2),(-1/2,-2/2),(0/2,-1/2),(-3/2,-6/2),(-2/2,-5/2),(-1/2,-4/2),(-0/2,-3/2),(1/2,-2/2),(5/2,4/2),(4/2,3/2),(3/2,2/2),(2/2,1/2),(1/2,0/2),(4/2,1/2),(3/2,0/2), (2/2,-1/2), (4/2,5/2),(3/2,4/2),(2/2,3/2),(1/2,2/2),(0/2,1/2),(3/2,6/2),(2/2,5/2),(1/2,4/2),(0/2,3/2),(-1/2,2/2),(5/2,2/2),(6/2,3/2), (0/2,-4/2),(1/2,-3/2),(2/2,-2/2),(3/2,-1/2),(4/2,0/2), (4/2,6/2), (5/2, 5/2), (6/2, 4/2)}
{ \filldraw[red!90!black] \i circle (2pt); }
\draw (1/2,5/2) pic[rotate=45,blue] {cross=3pt};
\draw (2/2,6/2) pic[rotate=45,blue] {cross=3pt};
\draw (3/2,7/2) pic[rotate=45,blue] {cross=3pt};
\draw (-5/2,-1/2) pic[rotate=45,blue] {cross=3pt};
\draw (-6/2,-2/2) pic[rotate=45,blue] {cross=3pt};
\draw (-7/2,-3/2) pic[rotate=45,blue] {cross=3pt};
\draw (-1/2,-5/2) pic[rotate=45,blue] {cross=3pt};
\draw (-2/2,-6/2) pic[rotate=45,blue] {cross=3pt};
\draw (-3/2,-7/2) pic[rotate=45,blue] {cross=3pt};
\draw (5/2,1/2) pic[rotate=45,blue] {cross=3pt};
\draw (6/2,2/2) pic[rotate=45,blue] {cross=3pt};
\draw (7/2,3/2) pic[rotate=45,blue] {cross=3pt};

\node[blue] at (-1/4,-9/4) {$\diamond$};
\node[blue] at (-9/4,-1/4) {$\diamond$};
\node[blue] at (1/4,9/4) {$\diamond$};
\node[blue] at (9/4,1/4) {$\diamond$};

\node[blue] at (9/4+0.5,1/4-0.15) {$-P_1$};
\node[blue] at (-9/4-0.4,-1/4+0.20) {$P_1$};
\node[blue] at (1/4-0.45,9/4+0.43) {$-P_2$};
\node[blue] at (-1/4+0.35,-9/4-0.25) {$P_2$};

\end{tikzpicture}

&
\begin{tikzpicture}[scale=.60]
\filldraw[black!9!white] (-6.75/2-0.25/2,-3.25/2-0.25/2) -- (-4.5/2-0.25/2,-0.5/2-0.25/2) -- (0.5/2-0.25/2,4.5/2-0.25/2)  -- (3.25/2-0.25/2,6.75/2-0.25/2) -- (6.75/2-0.25/2,3.25/2-0.25/2) -- (4.5/2-0.25/2,0.5/2-0.25/2) -- (-0.5/2-0.25/2,-4.5/2-0.25/2) -- (-3.25/2-0.25/2,-6.75/2-0.25/2) -- cycle;
\draw[-, very thick, color=black] (-6.75/2,-3.25/2) -- (-4.5/2,-0.5/2) -- (0.5/2,4.5/2) -- (3.25/2,6.75/2) -- (6.75/2,3.25/2) -- (4.5/2,0.5/2) -- (-0.5/2,-4.5/2) -- (-3.25/2,-6.75/2) -- cycle;
\draw[-, very thick, color=green!60!black] (-6.75/2-0.25/2,-3.25/2-0.25/2) -- (-4.5/2-0.25/2,-0.5/2-0.25/2) -- (0.5/2-0.25/2,4.5/2-0.25/2) -- (3.25/2-0.25/2,6.75/2-0.25/2) -- (6.75/2-0.25/2,3.25/2-0.25/2) -- (4.5/2-0.25/2,0.5/2-0.25/2) -- (-0.5/2-0.25/2,-4.5/2-0.25/2) -- (-3.25/2-0.25/2,-6.75/2-0.25/2) -- cycle;

\draw[<->, thin, color=gray] (-4.6,0) -- (4.6,0);
\draw[<->, thin, color=gray] (0,-4.6) -- (0,4.6);
\draw[step=.5, ultra thin, color=gray!50!white] (-4.6,-4.6) grid (4.6,4.6);

\fill[fill=gray,fill opacity=0.2]  (-4.6,-4.6) -- (4.6,4.6) -- (4.6,-4.6) -- cycle;

\draw[->,ultra thick, color=green!60!black] (-2.5-0.2,-2.5-0.2) -- (-6/2-0.1,-6/2-0.1);
\node[color=blue] at (-2.9,-3.4) {$\theta$};

\foreach \i in { (-4/2,-0/2),(-3/2,1/2),(-2/2,2/2),(-1/2,3/2),(0/2,4/2),(-6/2,-4/2),(-5/2,-3/2),(-4/2,-2/2),(-3/2,-1/2),(-2/2,0/2),(-1/2,1/2),(0/2,2/2),(1/2,3/2),(2/2,4/2),(3/2,5/2), (-5/2,-5/2),(-4/2,-4/2),(-3/2,-3/2),(-2/2,-2/2),(-1/2,-1/2),(-0/2,0/2),(1/2,1/2),(2/2,2/2),(3/2,3/2),(4/2,4/2), (-4/2,-6/2),(-3/2,-5/2),(-2/2,-4/2),(-1/2,-3/2),(-0/2,-2/2),(1/2,-1/2),(2/2,0/2),(3/2,1/2),(4/2,2/2),(5/2,3/2), (-5/2,-4/2),(-4/2,-3/2),(-3/2,-2/2),(-2/2,-1/2),(-1/2,0/2),(-6/2,-3/2),(-5/2,-2/2),(-4/2,-1/2),(-3/2,-0/2), (-2/2,1/2), (-4/2,-5/2),(-3/2,-4/2),(-2/2,-3/2),(-1/2,-2/2),(0/2,-1/2),(-3/2,-6/2),(-2/2,-5/2),(-1/2,-4/2),(-0/2,-3/2),(1/2,-2/2),(5/2,4/2),(4/2,3/2),(3/2,2/2),(2/2,1/2),(1/2,0/2),(4/2,1/2),(3/2,0/2), (2/2,-1/2), (4/2,5/2),(3/2,4/2),(2/2,3/2),(1/2,2/2),(0/2,1/2),(3/2,6/2),(2/2,5/2),(1/2,4/2),(0/2,3/2),(-1/2,2/2),(5/2,2/2),(6/2,3/2), (0/2,-4/2),(1/2,-3/2),(2/2,-2/2),(3/2,-1/2),(4/2,0/2)}
{ \filldraw[red!90!black] \i circle (2pt); }

\draw (4/2,6/2) pic[rotate=45,blue] {cross=3pt};
\draw (5/2,5/2) pic[rotate=45,blue] {cross=3pt};
\draw (6/2,4/2) pic[rotate=45,blue] {cross=3pt};

\node[blue] at (2,-3.5) {$M^{+}_\bR$};

\end{tikzpicture}
\end{tabular}

\caption{\footnotesize This is a description of the windows for the representation $X=(k^2)^{\oplus 10}$ discussed in \Cref{ex:grassmannian_N_10}. To the left is the cylinder window in the case $\theta = 0$ with the blue crosses indicating the $T$-weights in $\overline{\nabla} \setminus \nabla$. As in \Cref{fig:grassmannian_N_6}, we have marked the special weights $P_i := \zeta_{\lambda_i}/2$ for $i=1,2$ that become relevant in the passage from the cylinder window to the barrel window. To the right is a presentation of the perturbation of the barrel window by a parameter $\theta = s\omega^\ast$. Here we take $-s$ to be a small positive real number. The weights indicated by blue crosses are those that are excluded from $\nabla$ by the $\theta$ perturbation. } \label{fig:grassmannian_N_10}
\end{figure}
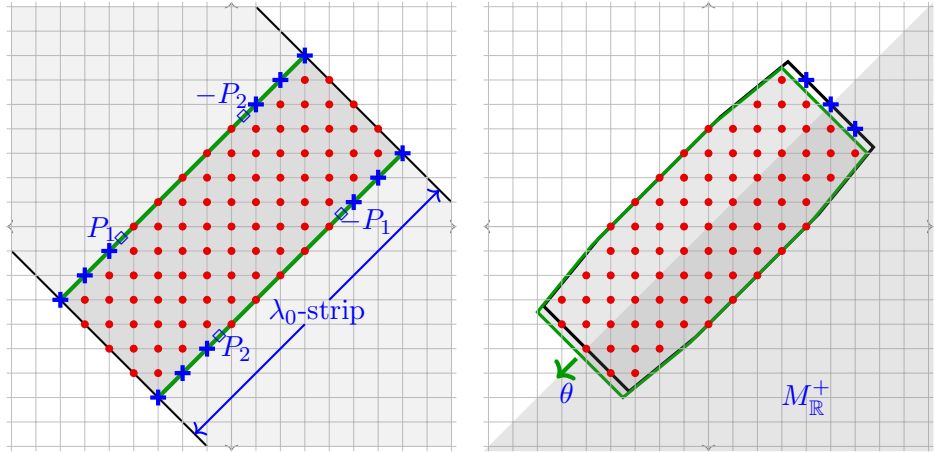

The $\GL_2$ representations satisfying the conditions of \Cref{ex:GL2_reps_example} provide a large class of examples for \Cref{T:main_fano}. For such a representation, we will generally take $\lambda_0 = (-1,-1)$, although any positive multiple of it also works. The subset $(\omega^\ast)^\perp \subset N_\bR$ consists of zero and positive multiples of the cocharacters $\lambda^{\omega^\ast}_1 := (-1,1)$ and $\lambda^{\omega^\ast}_2 := (1,-1)$. Therefore, $\lambda_0, \lambda^{\omega^\ast}_1, \lambda^{\omega^\ast}_2$ suffice to describe the cylinder and barrel windows (see \Cref{D:window}).

\begin{ex}\label{ex:grassmannian_N_10}
Consider the linear representation $X=(k^2)^{\oplus N}$ of $\GL_2$, whose corresponding GIT quotient $X^{\rm ss}(\omega^\ast)/ \GL_2$ is the Grassmannian of planes in $k^N$. Then $\zeta_{\lambda_0}=(-N,-N), \zeta_{\lambda_1^{\omega^\ast}} = (-N+1,-1)$, and  $\zeta_{\lambda_2^{\omega^\ast}}=(-1,-N+1)$, so $\eta_{\lambda_0} = 2N$ and $\eta_{\lambda_1^{\omega^\ast}}=\eta_{\lambda_2^{\omega^\ast}}=N-2$. \Cref{P:exceptional_collections} and \Cref{T:main_fano} imply that the set of $T$-weights in $M^+ \cap (\theta + {\nabla})$ produce a full strong exceptional collection of vector bundles in $\rm Gr (2,10)$ for any $\theta \in M_\bR^W$ that is $\lambda_0$-generic. The barrel window in the case $N=10$ is described in \Cref{fig:grassmannian_N_10} below. One can check that the number of dominant $T$-weights in $\theta + {\nabla}$ is $45$. This full strong exceptional collection on $\rm Gr(2,10)$ differs from that of \cite{K88}.
\end{ex}

\section{Nef-Fano GIT quotients by a rank-\texorpdfstring{$2$}{2} torus}

The goal of this section is to extend the methods of \Cref{S:Fano_rank2} to cover GIT quotients by a rank-two torus with a GIT parameter that is close to the anticanonical character $\omega^\ast$ in the sense of variation of GIT, thus the resulting GIT quotients are nef-Fano. We prove the following theorem and discuss how it extends the results of \cite{BH09} from the Fano case to the nef-Fano case. 

\begin{thm}\label{T:main_nef_fano}
Let $G=T$ be a split torus of rank $2$ over $k$ and let $X$ be an affine scheme with a linear $T$-action satisfying \Cref{H:main_setup}. Assume that the weights of $X$ span $M_\bR$. Let $\theta\in M_\bR$ be a character such that $\theta + \overline{\nabla}$ contains no weights on its boundary. If some weight of $X$ is proportional to $\omega^\ast$, then for some generic character $\ell \in M_\bR$ and $0 < \epsilon \ll 1$,  $X^{\rm ss} (\omega ^\ast + \epsilon \ell)$ has finite $T$-stabilizers, and the line bundles
\[
\left\{ \cO_{X^{\rm ss} (\omega^\ast + \epsilon \ell)} \otimes U : U \in \Rep(T) \text{ a character lying in }\theta + \nabla_{\omega^\ast + \epsilon \ell} \right\}
\]
split-generate $\DCoh(X^{\rm ss}(\omega^\ast + \epsilon \ell)/ T)$. In particular, the exceptional collection from \Cref{P:exceptional_collections} is full.
\end{thm}

\begin{proof}
That $X^{\rm ss} (\omega ^\ast + \epsilon \ell)$ has finite $T$-stabilizers for some generic character $\ell$ and $0 < \epsilon \ll 1$ is a consequence the theory of variation of GIT (see \Cref{ex:vgit_near_anticanonical}), so we need to only prove fullness of the collection.

The ``genericity" assumption that $\theta + \overline{\nabla}$ contains no weights on its boundary implies the equality of sets
\[
    M \cap \left( \theta + \overline{\nabla} \right) = M \cap \left( \theta + \overline{\nabla}_{\omega^\ast + \epsilon \ell} \right) 
\]
for all $0 \le \epsilon \ll 1$. Indeed, $(\omega^\ast)^\perp$ is spanned by a single cocharacter, say $\lambda'$, so 
\begin{equation} \label{eq:lambda_epsilon}
   \lambda'_\epsilon = \lambda'- \epsilon \frac{\langle \lambda',\ell\rangle }{\langle \lambda_0, \omega^\ast+\epsilon\ell \rangle} \lambda_0
\end{equation}
spans $(\omega^\ast + \epsilon \ell)^\perp$. Moreover, there are finitely many $T$-weights in $\theta + \textrm{Int} \overline{\nabla}$ and for such $\chi$, $\epsilon$ can be made sufficiently small so that $\lvert \langle \lambda', \chi - \theta \rangle \rvert < \eta_{\lambda'}/2$ if and only if $\lvert \langle \lambda'_\epsilon, \chi - \theta \rangle \rvert < \eta_{\lambda'_\epsilon}/2$. 

The same genericity assumption together with the definition of the barrel window imply the equality $M \cap \left( \theta + \overline{\nabla} \right) =  M \cap \left( \theta + \nabla \right)$. Thus, adopting the notation $\cC(S) \subset \DCoh(X/T)$ for $S \subset M_\bR$ from the proof of \Cref{T:main_fano}, we wish to show that for any $\chi \in M$, $\cO_X \otimes V(\chi) \in \cC(\theta + \overline{\nabla})$. This is equivalent to showing $\cC(M_\bR)=\cC(\theta+\overline{\nabla})$. The first step, that $\cC(M_\bR) = \cC(\theta+B_{\lambda_0})$, is identical to the proof of \Cref{L:reduction_to_lambda_0}. To complete this proof, we will redo the reduction to the cylinder window in the proof of \Cref{T:main_fano}, showing $\cC(\theta+B_{\lambda_0})=\cC(\theta+\overline{\nabla})$, with modifications accounting for the existence of a weight that is proportional to $\omega^\ast$ and from having a trivial Weyl group, i.e., $M_\bR^W=M_\bR$.

\medskip
\noindent \textit{Reduction to the cylinder window.}
\medskip

Let $\lambda'$ be a cocharacter that spans $(\omega^\ast)^\perp$ and satisfies $\langle \lambda', \ell \rangle <  0$. This pairing cannot equal zero because $\ell$ cannot be proportional to $\omega^\ast$. Define for $R \ge 1$ the enlarged cylinder window 
\[
    \overline{\nabla}^R := \left\{ \chi \in B_{\lambda_0} : \lvert \langle \lambda', \chi \rangle \rvert \le R \eta_{\lambda'}/ 2 \right\}.
\]
Take $Q_{\lambda'}$ as in \eqref{eq:Q_lambda} and define the following subset of the boundary of $\overline{\nabla}^R$. 
\[
    \Psi_R := \left\{  \chi \in  \overline{\nabla}^R: \left\{ \begin{array}{ll}
         \langle \lambda', \chi \rangle = - R \eta_{\lambda'}/ 2 \textrm{ and } Q_{\lambda'} < \langle \lambda_0, \chi \rangle \le \eta_{\lambda_0}/ 2, \text{ or} \\
          \langle \lambda', \chi \rangle = R \eta_{\lambda'}/ 2 \textrm{ and } - \eta_{\lambda_0}/ 2 \le \langle \lambda_0, \chi \rangle < - Q_{\lambda'} .
    \end{array} \right. \right\}
\]

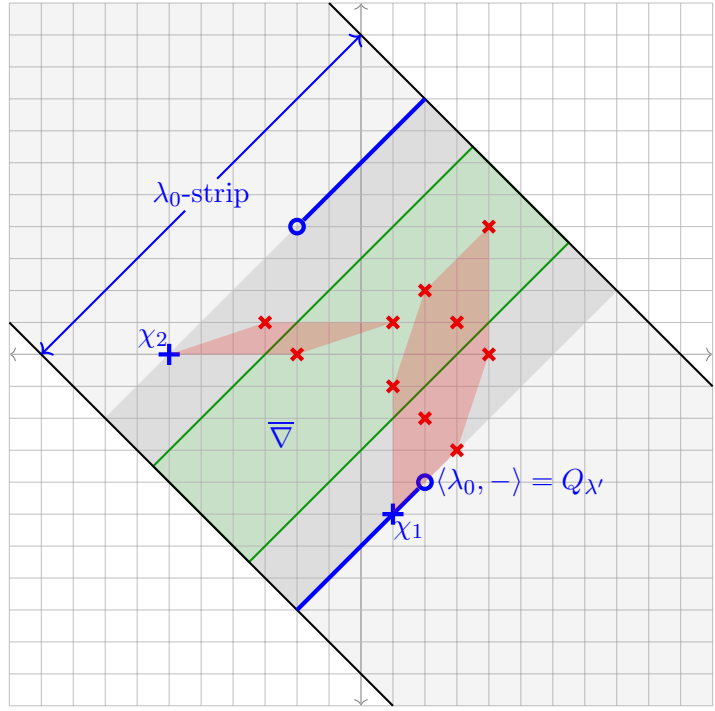
\begin{figure}[ht]
\begin{tikzpicture}[scale=.85]

\filldraw[gray!20!white] (-8/2,-2/2) --  (2/2,8/2) -- (8/2,2/2) -- (-2/2,-8/2) -- cycle;
\fill[fill=green,fill opacity=0.1]  (-6.5/2,-3.5/2) --  (3.5/2,6.5/2) -- (6.5/2,3.5/2) -- (-3.5/2,-6.5/2) -- cycle;

\draw[<->, thin, color=gray] (-5.5,0) -- (5.5,0);
\draw[<->, thin, color=gray] (0,-5.5) -- (0,5.5);
\draw[step=.5, ultra thin, color=gray!50!white] (-5.5,-5.5) grid (5.5,5.5);

\draw[-, thick, green!60!black] (-6.5/2,-3.5/2) --  (3.5/2,6.5/2) -- (6.5/2,3.5/2) -- (-3.5/2,-6.5/2) -- cycle;
\node[blue] at (-2.5/2,-2.5/2) {$\overline{\nabla}$};
\draw[-, thick, black] (-11/2,1/2) --  (1/2,-11/2);
\draw[-, thick, black] (-1/2,11/2) --  (11/2,-1/2);

\fill[fill=gray,fill opacity=0.09] (-11/2,1/2) -- (-11/2, 11/2) -- (-1/2,11/2) -- (11/2, -1/2) -- (11/2, -11/2) -- (1/2,-11/2) -- cycle;
\draw[->, thick, color = blue] (-4.5/2,5.5/2) -- (0/2,10/2);
\draw[->, thick, color = blue] (-5.5/2,4.5/2) -- (-10/2,0/2);
\node[blue] at (-5.0/2,5.0/2) {$\lambda_0$-strip};

\draw[-, ultra thick, color=blue] (1.8/2,-4.2/2) -- (-2/2,-8/2);
\draw[-, ultra thick, color=blue] (-1.8/2,4.2/2) -- (2/2,8/2);
\draw[ultra thick, blue] (2/2,-4/2) circle (3 pt);
\draw[ultra thick, blue] (-2/2,4/2) circle (3 pt);
\node[blue] at (5/2,-4/2) {$\langle \lambda_0, - \rangle = Q_{\lambda'}$};

\draw (-6/2,0/2) pic[rotate=45,blue] {cross=4pt};
\node[blue] at (-6.5/2,0.5/2) {$\chi_2$};
\draw (-2/2,0/2) pic[rotate=0,red!90!black] {cross=3pt};
\draw (1/2,1/2) pic[rotate=0,red!90!black] {cross=3pt};
\draw (-3/2,1/2) pic[rotate=0,red!90!black] {cross=3pt};
\fill[fill=red,fill opacity=0.2] (-6/2,0/2) --  (-2/2, 0/2) -- (1/2,1/2) -- (-3/2,1/2) -- cycle;

\draw (1/2,-5/2) pic[rotate=45,blue] {cross=4pt};
\node[blue] at (1.5/2,-5.5/2) {$\chi_1$};
\draw (1/2,-1/2) pic[rotate=0,red!90!black] {cross=3pt};
\draw (2/2,2/2) pic[rotate=0,red!90!black] {cross=3pt};
\draw (3/2,1/2) pic[rotate=0,red!90!black] {cross=3pt};
\draw (4/2,4/2) pic[rotate=0,red!90!black] {cross=3pt};
\draw (2/2,-2/2) pic[rotate=0,red!90!black] {cross=3pt};
\draw (4/2,0/2) pic[rotate=0,red!90!black] {cross=3pt};
\draw (3/2,-3/2) pic[rotate=0,red!90!black] {cross=3pt};
\fill[fill=red,fill opacity=0.2] (1/2,-5/2) --  (1/2, -1/2) -- (2/2,2/2) -- (4/2,4/2) -- (4/2,0/2) -- (3/2,-3/2) -- cycle;
\end{tikzpicture}
\caption{\footnotesize  This diagram illustrates the reduction to the cylinder window $\overline{\nabla}$ for the $\bG_m^2$ representation $X = \Sym^4 \bC^2$ with $\lambda_0 = (-1,-1)$, $\theta = 0$, $\ell = (1,-1)$, and $\lambda' = (-1, 1)$. The weights of this representation are $(0,4),(1,3),(2,2),(3,1),(4,0)$. Thus, $\eta_{\lambda_0} = 20$, $\eta_{\lambda'} = 6$, and the green shaded region is the cylinder window. For each $T$-weight we pick an appropriate complex depending on a comparison of its $\lambda_0$-pairing with $Q_{\lambda'} = 2$. In this example, the gray shaded rectangle is the enlarged cylinder window for $R=2$. The blue circles on the right and left side facets indicate the weights on the boundary of $\overline{\nabla}^R$ whose $\lambda_0$-pairing is $Q_{\lambda'}$ and $-Q_{\lambda'}$ respectively, thus the blue lines on the boundary of the enlarged cylinder indicate the region $\Psi_R$. The weight $\chi_1$ is in $\Psi_R$ and this falls into Case 1 of \Cref{claim:nef_fano_cylinder_reduction}. The resulting weights $\mu$ are indicated by the red crosses and one can see that these lie either in the interior of $\overline{\nabla}^R$ or they have a smaller pairing with $\lambda_0$. The weight $\chi_2$ lies in $\overline{\nabla}^R \setminus \Psi_R$ and it falls into the case 2 of \Cref{claim:nef_fano_cylinder_reduction_2}. Again, the resulting weights $\mu$ are indicated by the red crosses and one can see that these lie entirely in the interior of $\overline{\nabla}^R$.} \label{fig:nef_fano_cylinder_reduction}
\end{figure}

\begin{claim}\label{claim:nef_fano_cylinder_reduction} If $R>1$, then for any $T$-weight $\chi \in \theta + \Psi_R$, $\cO_X \otimes V(\chi) \in \cC(\theta + ( \overline{\nabla}^R \setminus \Psi_R))$.
\end{claim}

\begin{proof} There are two cases:

\medskip
\noindent \textit{Case 1: $\langle \lambda', \chi - \theta \rangle = - R \eta_{\lambda'}/ 2 \textrm{ and } Q_{\lambda'} < \langle \lambda_0, \chi - \theta \rangle \le \eta_{\lambda_0}/ 2$.}
\medskip

We will prove the claim for the finite set of weights satisfying these properties using induction on $\langle \lambda_0, \chi-\theta \rangle$. Recall the cocharacter $\lambda'_\epsilon$ from \eqref{eq:lambda_epsilon}, and consider the complex $D^\vee_{\lambda'_\epsilon}, \chi$ for $0<\epsilon \ll 1$. It has homology supported in $G \cdot X^{-\lambda'_\epsilon \ge 0}$, which lies in $X^{\rm us}(\omega^\ast + \epsilon \ell)$ by \Cref{L:unstable_locus}. \Cref{L:complex_2} says that $D_{\lambda'_\epsilon, \chi}^\vee$ has one term of the form $\cO_X \otimes V(\chi)$ and other terms that have a filtration whose associated graded pieces are of the form $\cO_X \otimes V(\mu)$, where $\mu = \chi + \beta_{i_1}+\beta_{i_2}+\cdots+\beta_{i_p}$ with $p > 0$, $i_1,\ldots,i_p$ distinct, and $\langle\lambda'_\epsilon,\beta_{i_j} \rangle > 0$ for all $j =1, \ldots, p$.

That $\mu$ remains in the $\lambda_0$-strip follows from the assumed bounds on $\langle \lambda', \chi - \theta \rangle$ and the inequality
\begin{equation} \label{eq:Q_lambda_2}
    Q_{\lambda'} = -\frac{\eta_{\lambda_0}}{2} + \langle \lambda_0, - \det X^{\lambda' \ge 0} \rangle >-\frac{\eta_{\lambda_0}}{2},
\end{equation}
where the first equality is obtained from $-\eta_{\lambda_0} = \langle \lambda_0, \det X \rangle$.

Note that the weights $\beta_{i_j}$ that appear in $\mu$ are independent of $\epsilon$ for $\epsilon>0$ sufficiently small. Since $\lambda'$ satisfies $\langle \lambda', \ell \rangle <0$ and we can take $\epsilon$ to be arbitrarily small, the weights $\beta_{i_j}$ appearing in $\mu$ are precisely those for which $\langle\lambda',\beta_{i_j} \rangle \ge 0$. Therefore we have
\[
  \langle \lambda', \mu - \theta \rangle \le -R \eta_{\lambda'}/2 + \langle \lambda', \det X^{\lambda' \geq 0} \rangle = (-R/2 + 1) \eta_{\lambda'} < R \eta_{\lambda'}/ 2 .
\]
Here we have used $\langle \lambda', \det X^{\lambda' \geq 0} \rangle = \eta_{- \lambda'} = \eta_{\lambda'}$ to deduce the equality and the assumption $R > 1$ for the last inequality. If there is one $\beta_{i_j}$ in $\mu$ with $\langle \lambda',\beta_{i_j} \rangle >0$, we also have $\langle \lambda', \mu-\theta \rangle > \langle \lambda',\chi-\theta \rangle = -R\eta_{\lambda'}/2$, and hence $\mu -\theta \in \overline{\nabla}^{S} \subset \overline{\nabla}^R \setminus \Psi_R$ for some $S < R$.

Otherwise, all $\beta_{i_j}$ appearing in $\mu$ pair zero with $\lambda'$, so $\langle \lambda', \mu - \theta \rangle = -R \eta_{\lambda'}/ 2$. But since $p > 0$, we have $\langle \lambda_0, \chi - \theta \rangle > \langle \lambda_0, \mu - \theta \rangle$, so by the inductive hypothesis we know $\cO_X\otimes V(\mu) \in \cC(\theta + ( \overline{\nabla}^R \setminus \Psi_R))$. The claim follows.

\medskip
\noindent \textit{Case 2: $\langle \lambda', \chi - \theta \rangle = R \eta_{\lambda'}/ 2 \textrm{ and } - \eta_{\lambda_0}/ 2 \le \langle \lambda_0, \chi - \theta \rangle < - Q_{\lambda'}$.}
\medskip

We will prove the claim for the finite set of weights satisfying these properties using induction on $-\langle \lambda_0, \chi-\theta \rangle$. Take the complex $C_{-\lambda'_\epsilon}, \chi$ with $\lambda'_\epsilon$ as in \eqref{eq:lambda_epsilon} and $0<\epsilon \ll 1$. It has homology supported in $G \cdot X^{-\lambda'_\epsilon \ge 0}$, which lies in $X^{\rm us}(\omega^\ast + \epsilon \ell)$ by \Cref{L:unstable_locus}. \Cref{L:complex_1} says that $C_{-\lambda'_\epsilon, \chi}$ has one term of the form $\cO_X \otimes V(\chi)$ and other terms that have a filtration whose associated graded pieces are of the form $\cO_X \otimes V(\mu)$, where $\mu = \chi - \beta_{i_1}-\beta_{i_2}-\cdots-\beta_{i_p}$ with $p > 0$, $i_1,\ldots,i_p$ distinct, and $\langle -\lambda'_\epsilon,\beta_{i_j} \rangle < 0$ for all $j =1, \ldots, p$. As in the previous case, \eqref{eq:Q_lambda_2} and the assumed bounds for $\langle \lambda_0, \chi -\theta \rangle$ imply that $\mu$ remains in the $\lambda_0$-strip.

As in the previous case, the facts that the $\beta_{i_j}$ appearing in $\mu$ are constant for small $\epsilon <0$ and that $\langle \lambda', \ell \rangle <0$ imply that the weights $\beta_{i_j}$ appearing in $\mu$ are precisely those for which $\langle \lambda',\beta_{i_j} \rangle \geq 0$. This implies that $\langle \lambda', \mu - \theta \rangle \ge R \eta_{\lambda'}/2 - \langle \lambda', \det X^{\lambda' \geq 0} \rangle  > - R \eta_{\lambda'}/ 2$, where we have used $\eta_{- \lambda'} = \eta_{\lambda'}$ and the assumption $R > 1$ to deduce the last inequality. If one of the $\beta_{i_j}$ satisfies $\langle \lambda', \beta_{i_j} \rangle>0$, then $\langle \lambda', \mu - \theta \rangle < \langle \lambda', \chi-\theta \rangle = R \eta_{\lambda'}/2$, and hence $\mu -\theta \in \overline{\nabla}^{S} \subset \overline{\nabla}^R \setminus \Psi_R$ for some $S < R$.

On the other hand, if all $\beta_{i_j}$ appearing in $\mu$ pair zero with $\lambda'$, then $\langle \lambda', \mu - \theta \rangle = R \eta_{\lambda'}/ 2$. But since $p > 0$, we have $\langle \lambda_0, \chi - \theta \rangle < \langle \lambda_0, \mu - \theta \rangle$. Thus, by the inductive hypothesis we know $\cO_X \otimes V(\mu) \in \cC(\theta+(\overline{\nabla}^R \setminus \Psi_R))$. The claim follows.

\end{proof}

\begin{claim} \label{claim:nef_fano_cylinder_reduction_2}
If $R>1$, then for any $T$-weight $\chi \in \theta + ( \overline{\nabla}^R \setminus \Psi_R)$, $\cO_X \otimes V(\chi) \in \cC(\theta +  \overline{\nabla}^{S})$ for some $S<R$.
\end{claim}

\begin{proof} There are two cases:

\medskip
\noindent \textit{Case 1: $\langle \lambda', \chi - \theta \rangle = - R \eta_{\lambda'}/ 2 \textrm{ and } -\eta_{\lambda_0}/ 2 \le \langle \lambda_0, \chi - \theta \rangle \le Q_{\lambda'}$.}
\medskip

Take the complex $C_{\lambda', \chi}$. It has homology supported in $G \cdot X^{\lambda' \geq 0}$, which lies in $X^{\rm us} (\omega^\ast + \epsilon \ell)$ by the Hilbert-Mumford criterion and the assumption $\langle \lambda', \ell \rangle < 0$. \Cref{L:complex_1} says that $C_{\lambda', \chi}$ has one term of the form $\cO_X \otimes V(\chi)$ and other terms that have a filtration whose associated graded pieces are of the form $\cO_X \otimes V(\mu)$, where $\mu = \chi - \beta_{i_1}-\beta_{i_2}-\cdots-\beta_{i_p}$ with $p > 0$, $i_1,\ldots,i_p$ distinct, and $\langle\lambda',\beta_{i_j} \rangle < 0$ for all $j =1, \ldots, p$. 

As in the proof of \Cref{claim:nef_fano_cylinder_reduction}, the assumed bounds for $\langle \lambda_0, \chi -\theta \rangle$ imply that $\mu$ remains in the $\lambda_0$-strip $\theta + B_{\lambda_0}$. We also have 
\[
\langle \lambda', \chi - \theta \rangle < \langle \lambda', \mu - \theta \rangle \le (- R/2 + 1) \eta_{\lambda'} < R \eta_{\lambda'}/2,
\]
where the first inequality holds because $p > 0$ and the last inequality holds because $R > 1$. Thus, $\mu -\theta \in \overline{\nabla}^{S}$ with $S < R$.

\medskip
\noindent \textit{Case 2: $\langle \lambda', \chi - \theta \rangle =  R \eta_{\lambda'}/ 2 \textrm{ and } -Q_{\lambda'} \le \langle \lambda_0, \chi - \theta \rangle \le \eta_{\lambda_0}/ 2$.}
\medskip

Take the complex $D_{-\lambda', \chi}^\vee$. It has homology supported in $G \cdot X^{\lambda' \geq 0}$, which lies in $X^{\rm us} (\omega^\ast + \epsilon \ell)$ by the Hilbert-Mumford criterion and the assumption $\langle \lambda', \ell \rangle < 0$. \Cref{L:complex_2} says that $D_{-\lambda', \chi}^\vee$ has one term of the form $\cO_X \otimes V(\chi)$ and other terms that have a filtration whose associated graded pieces are of the form $\cO_X \otimes V(\mu)$, where $\mu = \chi + \beta_{i_1}+\beta_{i_2}+\cdots+\beta_{i_p}$ with $p > 0$, $i_1,\ldots,i_p$ distinct, and $\langle - \lambda',\beta_{i_j} \rangle > 0$ for all $j = 1, \ldots, p$. 

As before, the assumed bounds for $\langle \lambda_0, \chi -\theta \rangle$ imply that $\mu$ remains in the $\lambda_0$-strip $\theta + B_{\lambda_0}$. We also have
\[
\langle \lambda', \chi - \theta \rangle > \langle \lambda', \mu - \theta \rangle \ge ( R/2 - 1) \eta_{\lambda'} > -R \eta_{\lambda'} /2
\]
where the first inequality holds because $p > 0$, the second inequality holds because $\eta_{\lambda'} = \eta_{-\lambda'}$, and the last inequality holds because $R > 1$. Thus, $\mu -\theta \in \overline{\nabla}^{S}$ with $S < R$.
\end{proof}

The set of $R$ for which $\lvert \langle \lambda', \chi - \theta \rangle \rvert = R\eta_{\lambda'}/2$ for some $\chi \in M$ is discrete, so one can perform induction on $R$, using \Cref{claim:nef_fano_cylinder_reduction} and \Cref{claim:nef_fano_cylinder_reduction_2}, to deduce that for any $\chi \in (\theta+B_{\lambda_0})\cap M$, $\cO_X \otimes V(\chi) \in \cC(\theta+\overline{\nabla})$. \Cref{fig:nef_fano_cylinder_reduction} illustrates an example of this inductive argument.
\end{proof}

\subsection{Example: toric Deligne-Mumford stacks of Picard rank two} \label{S:toric_examples}

Smooth toric Deligne-Mumford (DM) stacks are introduced in \cite{BCS03}. They are smooth DM stacks associated to some combinatorial data, called a \emph{stacky fan}, consisting of a triple $(A, \Sigma, \{ v_1,\ldots,v_n \})$ of a finitely generated abelian group $A$, a complete rational simplicial fan $\Sigma \subset A _\bR$, and a choice of non-torsion elements $v_i$ in each one-dimensional cone of $\Sigma$. The associated toric DM stack is denoted $\bP_\Sigma$.

For a smooth toric Fano DM stack $\bP_\Sigma$ with Picard rank at most $2$, the existence of a full strong exceptional collection consisting of line bundles was established by Borisov and Hua in \cite{BH09}. Their technique produces a particular window $P\subset \Pic(\bP_\Sigma) \otimes \bR$ from which the exceptional collection is exhibited as the set of line bundles whose image in $\Pic(\bP_\Sigma)\otimes \bR$ lies in $\theta+P$, where $\theta$ is any generic point. We will show below that whenever $M$ is spanned by the weights of $X$ and $G = \bG_m^2$, then the GIT quotients appearing in \Cref{T:main_fano} and \Cref{T:main_nef_fano} that have Picard rank $2$ are smooth toric DM stacks. In particular, \Cref{T:main_nef_fano} can be viewed as providing a large class of examples where the result of \cite{BH09} for Picard rank 2 can be extended to the case of smooth nef-Fano toric DM stacks.

\subsubsection{Linear GIT quotients by $\bG_m^2$ as toric stacks}

\begin{lem} \label{L:GIT_as_toric_stacks}
Let $X$ be a linear representation of $T = \bG_m^2$ over $\bC$ satisfying \Cref{H:main_setup}. Assume
\begin{enumerate}
    \item the weights of $X$ span $M$,
    \item $X^{\rm ss} (\omega^\ast + \epsilon \ell)$ has finite stabilizers in $T$ for some $\ell \in M_\bR$ and all $0 < \epsilon \ll 1$, and 
    \item the GIT quotient $X^{\rm ss} (\omega^\ast + \epsilon \ell)/T$ has Picard rank $2$,
\end{enumerate}
then $X^{\rm ss} (\omega^\ast + \epsilon \ell)/T$ is a smooth toric DM stack.
\end{lem}

\begin{proof}

By definition of toric DM stacks \cite{BCS03}*{Section~3} it will suffice to produce a stacky fan $\left(A, \Sigma, \{ v_1, \ldots, v_n \} \right)$ such that $X^{\rm us} (\omega^\ast + \epsilon \ell)$ coincides with the complement of the vanishing set of the monomial ideal 
\begin{equation} \label{eq:ideal_of_fan}
    I_{\Sigma}:=\left( \prod_{v_i\not\subset \sigma} z_i : \sigma \text{ is a cone in } \Sigma \right) 
\end{equation}
of $\Gamma (X, \cO_X) = \bC [z_1, \ldots, z_n]$. We will use the toric GIT constructions in \cite{CLS11}*{Chapter~14}. For a character $\chi\in M$, the semistable locus therein is described in terms of the sheaf of sections $\cL(\chi)$ of the trivial line bundle linearized with respect to $T$ by the character $\chi$, thus $X^{\rm ss} (\omega^\ast + \epsilon \ell)$ means $X^{\rm ss}_{-(\omega^\ast + \epsilon \ell)}$ in the notation of \cite{CLS11}*{Section~14.1}. We are using here that $R (\omega^\ast + \epsilon \ell) \in M$ for some $0 < \epsilon \ll 1$ and some integer $R \gg 0$. 

Let $\beta_1, \ldots, \beta_n$ be the weights of $X$ and consider the dual short exact sequences of abelian groups
\begin{equation} \label{eq:gale_exact_sequence}
    0 \rightarrow A^\ast \xrightarrow{\varrho} \bZ^n \xrightarrow{\varphi} M \rightarrow 0 , \quad \textrm{ and } \quad 0 \rightarrow N \xrightarrow{\varphi^\ast} \bZ^n \xrightarrow{\varrho^\ast}A \rightarrow 0 ,
\end{equation}
where $\varphi: \bZ^n \rightarrow M$ sends the $i^{\rm th}$ standard vector in $\bZ^n$ to the weight $\beta_i$ and $A^\ast :=\ker(\varphi)$. The map $\varphi$ is surjective by the assumption that the weights of $X$ span $M$, so one obtains an inclusion $T \subset \bG_m^n$ by applying $\Hom (-,\bG_m)$ to $\varphi$. Here we are identifying the characters of $\bG_m^n$ with $\bZ^n$.

Let $\alpha = (\alpha_1,\ldots,\alpha_n) \in \bZ^n$ be a lift of $-R(\omega^\ast + \epsilon \ell)$ via $\varphi$. We may assume that $\alpha_i < 0$ for all $i= 1, \ldots, n$ because $(-1, \ldots, -1)$ is a lift of $-\omega^\ast = -\det X \in M$ and $0 < \epsilon \ll 1$.  Let $v_1, \ldots, v_n \in A$ be the image of the standard basis of $\bZ^n$ under the map $\varrho^\ast$. Then one can associate to $\alpha$ the polyhedron
\[
     \Delta = \left\{ p\in A_\bR^\ast : p \cdot v_i \geq - \alpha_i \textrm{ for all } i =1,\ldots,n \right\} 
\]
that defines a (possibly degenerate) normal fan $\Sigma \subset A_\bR^\ast$ as done in \cite{CLS11}*{Prop.~14.2.10}. Here $\cdot$ is the standard pairing between $A_\bR^\ast$ and $A_\bR$. Then \cite{CLS11}*{Thm.~14.2.13} says that the toric variety corresponding to $\Sigma$ coincides with the good quotient $Y := X^{\rm ss}(\omega^\ast + \epsilon \ell)/\!/T$. This correspondence implies the following:

\medskip
\noindent (a) \textit{$\Sigma$ is a (non degenerate) fan.}
\medskip

By \cite{CLS11}*{Prop.~14.2.10}, it suffices to show that $\Delta$ is full dimensional. We assume that $\Delta$ is not full dimensional and derive a contradiction.

We first show that there is a non-trivial cocharacter of $T':= \bG_m^n/T \cong \bG_m^{n - 2}$ that acts trivially on $Y$. Let $\cL(\alpha)$ be the sheaf of sections of the trivial line bundle on $X$ with $\bG_m^n$ linearization by $\alpha$, then for some integer $d \gg 0$, $\cL(\alpha)^{\otimes d}$ descends to a very ample line bundle on $Y$, and $U : =\Gamma \left(Y, \cL(\alpha)^{\otimes d} \right)$ is the $T'$ representation with weights $d \Delta \cap A^\ast$. Here we are identifying $A^\ast$ with the characters of $T'$. Thus $Y$ has a $T'$-equivariant embedding into $\bP(U)$. If $\Delta$ is not of full dimension, then it is possible to shift the weights of $U$ by a single character so that they lie in a linear subspace of $\bA^\ast_\bQ$. Thus every character in $d\Delta \cap A^\ast$ pairs zero with some $\lambda \in A_\bQ$. In particular, a sufficiently large integer multiple of $\lambda$ is a cocharacter of $T'$ that acts trivially on $\bP(U)$ and hence on $Y$.

Now the existence of this cocharacter implies that the dimension of the dense $T'$ orbit in $Y$ is strictly smaller than $\dim T' = n-2$. This contradicts the fact that $\dim Y = \dim X - \dim T = n-2$, which holds because $X^{\rm ss}(\omega^\ast)$ is assumed to have finite stabilizers and thus $Y$ is a geometric quotient.

\medskip
\noindent (b) \textit{$\Sigma$ is complete and simplicial:}
\medskip

Note that $Y$ is projective over $\bC$ because the morphism $Y \rightarrow \Spec(\cO_X^T)$ is projective and  $\cO_X^T \cong \bC$ under \Cref{H:main_setup}. Furthermore, the assumption on finite stabilizers implies that $Y$ only has finite quotient singularities. Thus \cite{CLS11}*{Thm.~3.1.19} implies $\Sigma$ is simplicial and complete. 

\medskip
\noindent (c) \textit{$\Sigma$ has exactly $n$ one dimensional cones.}
\medskip

Because $\Sigma$ is complete and simplicial, \cite{CLS11}*{Thm.~5.1.11} implies $Y$ is isomorphic to a geometric quotient of the form $ U /\!/ G $, where $U$ is a particular open subset of $\bC^{\vert \Sigma(1) \vert}$,  $\Sigma(1)$ is the set of one dimensional cones of $\Sigma$, and $G \subset \bG_m^{\vert \Sigma(1) \vert}$ is a subgroup. If $\Sigma$ had fewer than $n$ one dimensional cones, then because $Y$ is of dimension $n-2$, the above construction would imply that $Y$ has Picard rank strictly less than $2$. This would contradict the assumption on the Picard rank.

\medskip
\noindent{\textit{Completing the proof:}}
\medskip

Note that the facets of $\Delta$ occur among the following (possibly empty) subsets $F_1,\ldots,F_n$, where 
\[
    F_i:=\{ p\in \Delta :p\cdot v_i = - \alpha_i \} .
\]
We have shown that $\Sigma$ has exactly $n$ one dimensional cones, thus $F_1, \ldots F_n$ are all the facets of $\Delta$. From the definition of $\Sigma$ (see \cite{CLS11}*{Prop.~14.2.10}), the one dimensional cones are $\sigma_1, \ldots, \sigma_n$, where
\[ 
    \sigma_i=\{ q\in A_{\bR}:p\cdot q \geq 0 \textrm{ for all } p\in F_i\}.
\]
Because we chose $\alpha$ so that $\alpha_i < 0$ for all $i = 1, \ldots, n$, it follows that $v_i \in \sigma_i$ and is in fact a generator. It follows from (a), (b), and (c) that $(A, \Sigma, \{ v_1, \ldots, v_n \})$ is a stacky fan.

Now \cite{CLS11}*{Prop.~14.2.21} says that 
\begin{itemize}
\item[($\star$)] $(x_1,\cdots,x_n) \in X^{\rm ss} (\omega^\ast + \epsilon \ell)$ if and only if $\bigcap_{i\in B_x} F_i\neq \emptyset$, where $B_x=\{i :  x_i=0\}$.
\end{itemize}
The condition $\bigcap_{i\in B_x} F_i\neq \emptyset$ is equivalent to the condition that there is a cone $\sigma \in \Sigma$ containing all $\sigma_i$ for $i \in B_x$. We have already seen that $v_i$ is a generator for $\sigma_i \subset A_\bR$, so the latter condition is equivalent to having a cone $\sigma \in \Sigma$ containing all $v_i$ for $i \in B_x$. Using ($\star$) we obtain that $x \in X^{\rm ss}(\omega^\ast + \epsilon \ell)$ if and only if $x$ lies in the complement of the vanishing set of the ideal \eqref{eq:ideal_of_fan}.
\end{proof}

\subsubsection{The Borisov-Hua window} \label{subsubsection: Borisov_Hua}

Let $(A, \Sigma,\{v_1,\ldots,v_n\})$ be a stacky fan.

\begin{enumerate}
    \item Let $a_1,\ldots,a_n$ be a collection of non-zero rational numbers such that $\sum_{i=1}^n a_i =0$ and $\sum_{i=1}^n a_iv_i =0$. Let $I_+$ be the set of indices where $a_i>0$. For Fano DM toric stacks, the existence and uniqueness up to scaling of such a tuple is shown in \cite{BH09}*{Prop.~5.4}.
    \item Let $r_1,\ldots,r_n$ be another collection of positive real numbers $r_i$ such that $\sum_{i=1}^n r_i =1$ and $\sum_{i=1}^n r_i v_i =0$. The existence of such a tuple also follows from the proof of \cite{BH09}*{Prop.~5.4}.
\end{enumerate}
Let $E_1,\ldots,E_n$ be the standard basis of $\bZ^n$. The collection $ r_1, \ldots, r_n $ defines a linear functional $f$ on $\bZ^n$ sending $E_i$ to $r_i$. Analogously, there is a functional $\psi$ sending $E_i$ to $a_i$. By \cite{BH09}*{Prop.~3.3}, the Picard group of the toric DM stack $\bP_\Sigma$ is isomorphic to the quotient of $\bZ^n$ by the subgroup of elements of the form $\sum_{i=1}^n (p \cdot v_i) E_i$, for all $p \in A^\ast$. With this description, it follows that $\psi$ and $f$ descend to $\Pic(\bP_\Sigma)$ because $(a_i)$ and $(v_i)$ satisfy $\sum_{i=1}^n a_i v_i =0$ and $\sum_{i=1}^n r_i v_i =0$ respectively.

The Borisov-Hua window $P \subset \Pic(\bP_\Sigma)\otimes \bR$ is then defined to be the region cut out by the inequalities
\begin{equation} \label{eq:Borisov_Hua_inequalities} 
\lvert f(-) \rvert \leq \frac{1}{2},\qquad \lvert\psi(-)\rvert\leq \frac{1}{2}\sum_{i\in I_+} a_i. \end{equation}

\subsubsection{A comparison of the barrel window to the Borisov-Hua window}

Under \Cref{H:main_setup}, assume that $X^{\rm ss} (\omega^\ast)$ has finite stabilizers, the weights of $X$ span $M$, and $X^{\rm ss}/T$ has Picard rank $2$. Then by \Cref{L:GIT_as_toric_stacks} we can consider $X^{\rm ss}(\omega^\ast)/T$ as a toric DM stack with the stacky fan $(A, \Sigma,\{v_1,\ldots,v_n\})$ constructed above using \eqref{eq:gale_exact_sequence}, namely the ray generators $v_i$ are the images of the standard vectors in $\bZ^n$ under $\varrho^\ast$. In this context, the subgroup of elements of the form $\sum_{i=1}^n (p \cdot v_i) E_i$ for $p \in A^\ast$ is precisely the image of the map $\varrho: A^\ast \hookrightarrow \bZ^n$ from the first exact sequence in \eqref{eq:gale_exact_sequence}. Thus $\Pic(X^{\rm ss}(\omega^\ast)/T)$ is canonically identified with $M$ via the map $\varphi$ that sends the $i^{\rm th}$ standard vector to the weight $\beta_i$. In particular, if we find a suitable choice of tuples $(a_i)$ and $(r_i)$, then the Borisov-Hua window $P$ is the image in $M_\bR$ of the region in $\bR^n$ cut out by \eqref{eq:Borisov_Hua_inequalities}.

We claim that there is a particular choice of tuples $(a_1, \ldots, a_n)$ and $(r_1, \ldots, r_n)$ in the construction of the Borisov-Hua window such that the parallelogram $P$ in $M_\bR$ defined by these tuples coincides with the cylinder window $\overline{\nabla}$.

\medskip
\noindent (1) \textit{The choice of $\{ a_1, \ldots, a_n \}$}:
\medskip

Let $\lambda'\in N_\bR \setminus 0$ be any cocharacter satisfying $\langle \lambda', \omega^\ast\rangle=0$. For each $i=1,\ldots,n$, define 
\[a_i := - \langle \lambda', \beta_i \rangle.\]
We show that the collection $\{a_1,\cdots,a_n\}$ satisfies the conditions for \ref{subsubsection: Borisov_Hua}, i.e., all $a_i \neq 0$, $\sum_i a_i = 0$, and $\sum_i a_i v_i = 0$.

The assumption that $X^{\rm ss} (\omega^\ast)$ has finite stabilizers together with \Cref{L:unstable_locus} implies that no weights of $X$ lie in the span of $\omega^\ast$, so $a_i \neq 0$ for all $i$, and $\det X^{\lambda' \le 0} = \det X^{\lambda' < 0}$. We then have
\[\sum_{i\in I_+} a_i = \langle \lambda', - \det X^{\lambda' < 0} \rangle = \eta_{\lambda'} \qquad \sum_{i\in I \setminus I_+} a_i = \langle \lambda', - \det X^{\lambda' > 0} \rangle = -\eta_{\lambda'} , \]
which implies that $\sum_i a_i = 0$.

For the second condition, we observe that for every cocharacter $\lambda \in N = M^\ast$, the expression $\sum_{i=1}^n \langle \lambda, \beta_i\rangle v_i$ is the image of $\lambda$ under the composition $M^{\ast} \stackrel{\varphi^\ast} {\hookrightarrow} \bZ^{n} \stackrel{\varrho^\ast} {\twoheadrightarrow} A$, so $\sum_{i=1}^n a_i v_i=0$.

\medskip
\noindent (2) \textit{The choice of $\{ r_1, \ldots, r_n \}$}:
\medskip

Recall the cocharacter $\lambda_0$ from \Cref{H:main_setup} that pairs strictly negatively with respect to all weights of $X$. For each $i=1,\ldots,n$, define 
\[r_i := - \langle\lambda_0,\beta_{i}\rangle/\eta_{\lambda_0},\]
where $\eta_{\lambda_0} = \langle \lambda_0, - \det X \rangle$. Thus $\sum_i^n r_i =1$ and all $r_i$ are strictly positive. That $\sum_{i=1}^n r_i v_i=0$ follows as in the case above. 

The Borisov-Hua window for these choices of tuples $(a_i)$ and $(r_i)$ is the image under $\varphi$ of
\[
    \left\{ (y_1,\ldots,y_n) \in \bR^n : \left\lvert \sum_{i=1}^n - y_i \left\langle \lambda_0, \beta_i \right\rangle \right\rvert  \le \frac{\eta_{\lambda_0}}{2}, \left\lvert \sum_{i=1}^n - y_i \left\langle \lambda', \beta_i \right\rangle \right\rvert  \le \frac{1}{2} \sum_{i\in I_+} a_i \right\},
\]
which is the cylinder window $\overline{\nabla}$ because $\sum_{i\in I_+} a_i = \eta_{\lambda'}$. Hence for these choices of $(a_i)$ and $(r_i)$, and $\theta \in M_\bR$, $\theta + \overline{\nabla} = \theta + P$. 

Let $\theta \in M_\bR$ be a character such that $\theta + \partial \overline{\nabla} $ does not contain any points in $M_\bQ$. This is the genericity condition for $\theta$ in \cite{BH09}*{Thm.~5.11}. The condition also guarantees that $(\theta + \partial \overline{\nabla}) \cap M = \emptyset$, hence $(\theta + \overline{\nabla}) \cap M = (\theta + \nabla) \cap M$. In particular, \Cref{P:exceptional_collections} and \Cref{T:main_fano} recover the Borisov-Hua exceptional collection on $X^{\rm ss}(\omega^\ast) / T$ under the hypotheses of \Cref{L:GIT_as_toric_stacks}.

\bibliography{bib_exceptional}{}

\begin{bibdiv}
\begin{biblist}

\bib{AAGZ13}{article}{
      author={Ananyevskiy, Alexey},
      author={Auel, Asher},
      author={Garibaldi, Skip},
      author={Zainoulline, Kirill},
       title={Exceptional collections of line bundles on projective homogeneous varieties},
        date={2013},
        ISSN={0001-8708},
     journal={Adv. Math.},
      volume={236},
       pages={111\ndash 130},
         url={https://doi.org/10.1016/j.aim.2012.12.016},
      review={\MR{3019718}},
}

\bib{B78}{article}{
      author={Be\u{\i}linson, A.~A.},
       title={Coherent sheaves on $\bf{P}\sp{n}$ and problems in linear algebra},
    language={Russian},
        date={1978},
        ISSN={0374-1990},
     journal={Funktsional. Anal. i Prilozhen.},
      volume={12},
      number={3},
       pages={68\ndash 69},
      review={\MR{509388}},
}

\bib{BCS03}{article}{
      author={Borisov, Lev~A.},
      author={Chen, Linda},
      author={Smith, Gregory~G.},
       title={The orbifold {C}how ring of toric {D}eligne-{M}umford stacks},
        date={2005},
        ISSN={0894-0347},
     journal={J. Amer. Math. Soc.},
      volume={18},
      number={1},
       pages={193\ndash 215},
         url={https://doi.org/10.1090/S0894-0347-04-00471-0},
      review={\MR{2114820}},
}

\bib{BH09}{article}{
      author={Borisov, Lev},
      author={Hua, Zheng},
       title={On the conjecture of {K}ing for smooth toric {D}eligne-{M}umford stacks},
        date={2009},
        ISSN={0001-8708},
     journal={Adv. Math.},
      volume={221},
      number={1},
       pages={277\ndash 301},
         url={https://doi.org/10.1016/j.aim.2008.11.017},
      review={\MR{2509327}},
}

\bib{BS20}{article}{
      author={Bergh, Daniel},
      author={Schn\"{u}rer, Olaf~M.},
       title={Conservative descent for semi-orthogonal decompositions},
        date={2020},
        ISSN={0001-8708},
     journal={Adv. Math.},
      volume={360},
       pages={106882, 39},
         url={https://doi-org.proxy.library.cornell.edu/10.1016/j.aim.2019.106882},
      review={\MR{4031114}},
}

\bib{BT09}{article}{
      author={Bernardi, Alessandro},
      author={Tirabassi, Sofia},
       title={Derived categories of toric {F}ano 3-folds via the {F}robenius morphism},
        date={2009},
        ISSN={0373-3505},
     journal={Matematiche (Catania)},
      volume={64},
      number={2},
       pages={117\ndash 154},
      review={\MR{2800008}},
}

\bib{C11}{article}{
      author={Craw, Alastair},
       title={Quiver flag varieties and multigraded linear series},
        date={2011},
        ISSN={0012-7094},
     journal={Duke Math. J.},
      volume={156},
      number={3},
       pages={469\ndash 500},
      review={\MR{2772068}},
}

\bib{CLS11}{book}{
      author={Cox, David~A.},
      author={Little, John~B.},
      author={Schenck, Henry~K.},
       title={Toric varieties},
      series={Graduate Studies in Mathematics},
   publisher={American Mathematical Society, Providence, RI},
        date={2011},
      volume={124},
        ISBN={978-0-8218-4819-7},
         url={https://doi.org/10.1090/gsm/124},
      review={\MR{2810322}},
}

\bib{CM04}{article}{
      author={Costa, L.},
      author={Mir\'{o}-Roig, R.~M.},
       title={Tilting sheaves on toric varieties},
        date={2004},
        ISSN={0025-5874},
     journal={Math. Z.},
      volume={248},
      number={4},
       pages={849\ndash 865},
         url={https://doi.org/10.1007/s00209-004-0684-6},
      review={\MR{2103545}},
}

\bib{CT20_2}{article}{
      author={Castravet, Ana-Maria},
      author={Tevelev, Jenia},
       title={Derived category of moduli of pointed curves -- {II}},
        date={2020},
      eprint={2002.02889},
        note={arXiv},
}

\bib{CT20_1}{article}{
      author={Castravet, Ana-Maria},
      author={Tevelev, Jenia},
       title={Exceptional collections on certain {H}assett spaces},
        date={2020},
     journal={\'{E}pijournal G\'{e}om. Alg\'{e}brique},
      volume={4},
       pages={Art. 20, 34},
         url={https://doi.org/10.46298/epiga.2021.volume4.6456},
      review={\MR{4213164}},
}

\bib{DH98}{article}{
      author={Dolgachev, Igor~V.},
      author={Hu, Yi},
       title={Variation of geometric invariant theory quotients},
        date={1998},
        ISSN={0073-8301},
     journal={Inst. Hautes \'{E}tudes Sci. Publ. Math.},
      number={87},
       pages={5\ndash 56},
         url={http://www.numdam.org/item?id=PMIHES_1998__87__5_0},
        note={With an appendix by Nicolas Ressayre},
      review={\MR{1659282}},
}

\bib{DLM09}{article}{
      author={Dey, Arijit},
      author={Laso\'{n}, Micha\l},
      author={Micha{\l}ek, Mateusz},
       title={Derived category of toric varieties with {P}icard number three},
        date={2009},
        ISSN={0373-3505},
     journal={Matematiche (Catania)},
      volume={64},
      number={2},
       pages={99\ndash 116},
      review={\MR{2800007}},
}

\bib{E14}{article}{
      author={Efimov, Alexander~I.},
       title={Maximal lengths of exceptional collections of line bundles},
        date={2014},
        ISSN={0024-6107},
     journal={J. Lond. Math. Soc. (2)},
      volume={90},
      number={2},
       pages={350\ndash 372},
         url={https://doi.org/10.1112/jlms/jdu037},
      review={\MR{3263955}},
}

\bib{F19}{article}{
      author={Fonarev, Anton},
       title={Full exceptional collections on {L}agrangian {G}rassmannians},
        date={2022},
        ISSN={1073-7928},
     journal={Int. Math. Res. Not. IMRN},
      number={2},
       pages={1081\ndash 1122},
         url={https://doi.org/10.1093/imrn/rnaa098},
      review={\MR{4368880}},
}

\bib{FM15}{article}{
      author={Faenzi, Daniele},
      author={Manivel, Laurent},
       title={On the derived category of the {C}ayley plane {II}},
        date={2015},
        ISSN={0002-9939},
     journal={Proc. Amer. Math. Soc.},
      volume={143},
      number={3},
       pages={1057\ndash 1074},
         url={https://doi.org/10.1090/S0002-9939-2014-12312-4},
      review={\MR{3293722}},
}

\bib{H17}{article}{
      author={Hara, Wahei},
       title={Strong full exceptional collections on certain toric varieties with {P}icard number three via mutations},
        date={2017},
        ISSN={0373-3505},
     journal={Matematiche (Catania)},
      volume={72},
      number={2},
       pages={3\ndash 24},
         url={https://doi.org/10.4418/2017.72.2.1},
      review={\MR{3731498}},
}

\bib{HLS20}{article}{
      author={Halpern-Leistner, Daniel},
      author={Sam, Steven~V.},
       title={Combinatorial constructions of derived equivalences},
        date={2020},
        ISSN={0894-0347},
     journal={J. Amer. Math. Soc.},
      volume={33},
      number={3},
       pages={735\ndash 773},
         url={https://doi-org.proxy.library.cornell.edu/10.1090/jams/940},
      review={\MR{4127902}},
}

\bib{HP06}{article}{
      author={Hille, Lutz},
      author={Perling, Markus},
       title={A counterexample to {K}ing's conjecture},
        date={2006},
        ISSN={0010-437X},
     journal={Compos. Math.},
      volume={142},
      number={6},
       pages={1507\ndash 1521},
         url={https://doi.org/10.1112/S0010437X06002260},
      review={\MR{2278758}},
}

\bib{K88}{article}{
      author={Kapranov, M.M.},
       title={On the derived categories of coherent sheaves on some homogeneous spaces.},
        date={1988},
        ISSN={00209910},
     journal={Inventiones Mathematicae},
      volume={92},
      number={3},
       pages={479\ndash 508},
         url={http://proxy.library.cornell.edu/login?url=https://search.ebscohost.com/login.aspx?direct=true&db=edselc&AN=edselc.2-52.0-0001112148&site=eds-live&scope=site},
}

\bib{K97}{article}{
      author={King, Alistair},
       title={Tilting bundles on some rational surfaces},
        date={1997},
         url={https://people.bath.ac.uk/masadk/papers/tilt.pdf},
}

\bib{K08}{article}{
      author={Kresch, Andrew},
       title={On the geometry of {D}eligne–{M}umford stacks},
        date={200801},
     journal={Proc. Sympos. Pure Math},
      volume={80},
}

\bib{Ku08}{article}{
      author={Kuznetsov, Alexander},
       title={Exceptional collections for {G}rassmannians of isotropic lines},
        date={2008},
        ISSN={0024-6115},
     journal={Proc. Lond. Math. Soc. (3)},
      volume={97},
      number={1},
       pages={155\ndash 182},
         url={https://doi.org/10.1112/plms/pdm056},
      review={\MR{2434094}},
}

\bib{KP16}{article}{
      author={Kuznetsov, Alexander},
      author={Polishchuk, Alexander},
       title={Exceptional collections on isotropic {G}rassmannians},
        date={2016},
        ISSN={1435-9855},
     journal={J. Eur. Math. Soc. (JEMS)},
      volume={18},
      number={3},
       pages={507\ndash 574},
         url={https://doi.org/10.4171/JEMS/596},
      review={\MR{3463417}},
}

\bib{LM11}{article}{
      author={Laso\'{n}, Micha\l},
      author={Micha{\l}ek, Mateusz},
       title={On the full, strongly exceptional collections on toric varieties with {P}icard number three},
        date={2011},
        ISSN={0010-0757},
     journal={Collect. Math.},
      volume={62},
      number={3},
       pages={275\ndash 296},
         url={https://doi.org/10.1007/s13348-011-0044-x},
      review={\MR{2825714}},
}

\bib{M11}{article}{
      author={Manivel, L.},
       title={On the derived category of the {C}ayley plane},
        date={2011},
        ISSN={0021-8693},
     journal={J. Algebra},
      volume={330},
       pages={177\ndash 187},
         url={https://doi.org/10.1016/j.jalgebra.2010.12.011},
      review={\MR{2774623}},
}

\bib{Mi11}{article}{
      author={Micha{\l}ek, Mateusz},
       title={Family of counterexamples to {K}ing's conjecture},
        date={2011},
        ISSN={1631-073X},
     journal={C. R. Math. Acad. Sci. Paris},
      volume={349},
      number={1-2},
       pages={67\ndash 69},
         url={https://doi.org/10.1016/j.crma.2010.11.027},
      review={\MR{2755699}},
}

\bib{PN17}{article}{
      author={Prabhu-Naik, Nathan},
       title={Tilting bundles on toric {F}ano fourfolds},
        date={2017},
        ISSN={0021-8693},
     journal={J. Algebra},
      volume={471},
       pages={348\ndash 398},
         url={https://doi.org/10.1016/j.jalgebra.2016.09.007},
      review={\MR{3569189}},
}

\bib{P11}{article}{
      author={Polishchuk, A.},
       title={{$K$}-theoretic exceptional collections at roots of unity},
        date={2011},
        ISSN={1865-2433},
     journal={J. K-Theory},
      volume={7},
      number={1},
       pages={169\ndash 201},
         url={https://doi-org.proxy.library.cornell.edu/10.1017/is010004018jkt112},
      review={\MR{2774162}},
}

\bib{PS11}{article}{
      author={Polishchuk, Alexander},
      author={Samokhin, Alexander},
       title={Full exceptional collections on the {L}agrangian {G}rassmannians {$LG(4,8)$} and {$LG(5,10)$}},
        date={2011},
        ISSN={0393-0440},
     journal={J. Geom. Phys.},
      volume={61},
      number={10},
       pages={1996\ndash 2014},
         url={https://doi.org/10.1016/j.geomphys.2011.05.011},
      review={\MR{2822466}},
}

\bib{T00}{article}{
      author={Teleman, Constantin},
       title={The quantization conjecture revisited},
        date={2000},
        ISSN={0003-486X},
     journal={Ann. of Math. (2)},
      volume={152},
      number={1},
       pages={1\ndash 43},
         url={https://doi-org.proxy.library.cornell.edu/10.2307/2661378},
      review={\MR{1792291}},
}

\bib{U14}{article}{
      author={Uehara, Hokuto},
       title={Exceptional collections on toric {F}ano threefolds and birational geometry},
        date={2014},
        ISSN={0129-167X},
     journal={Internat. J. Math.},
      volume={25},
      number={7},
       pages={1450072, 32},
         url={https://doi.org/10.1142/S0129167X14500724},
      review={\MR{3238094}},
}

\bib{VdB91}{article}{
      author={Van~den Bergh, Michel},
       title={Cohen-{M}acaulayness of modules of covariants},
        date={1991},
        ISSN={0020-9910},
     journal={Invent. Math.},
      volume={106},
      number={2},
       pages={389\ndash 409},
         url={https://doi-org.proxy.library.cornell.edu/10.1007/BF01243917},
      review={\MR{1128219}},
}

\bib{SVdB17}{article}{
      author={\v{S}penko, \v{S}pela},
      author={Van~den Bergh, Michel},
       title={Non-commutative resolutions of quotient singularities for reductive groups},
        date={2017},
        ISSN={0020-9910},
     journal={Invent. Math.},
      volume={210},
      number={1},
       pages={3\ndash 67},
         url={https://doi-org.proxy.library.cornell.edu/10.1007/s00222-017-0723-7},
      review={\MR{3698338}},
}

\bib{SVdB21}{article}{
      author={\v{S}penko, \v{S}pela},
      author={Van~den Bergh, Michel},
       title={Semi-orthogonal decompositions of {GIT} quotient stacks},
        date={2021},
        ISSN={1022-1824},
     journal={Selecta Math. (N.S.)},
      volume={27},
      number={2},
       pages={Paper No. 16, 43},
         url={https://doi.org/10.1007/s00029-021-00628-3},
      review={\MR{4227868}},
}

\bib{W03}{book}{
      author={Weyman, Jerzy},
       title={Cohomology of vector bundles and syzygies},
      series={Cambridge Tracts in Mathematics},
   publisher={Cambridge University Press, Cambridge},
        date={2003},
      volume={149},
        ISBN={0-521-62197-6},
         url={https://doi.org/10.1017/CBO9780511546556},
      review={\MR{1988690}},
}

\end{biblist}
\end{bibdiv}
\bibliographystyle{plain}

\end{document}